\documentclass[11pt]{article}
\usepackage[letterpaper, left=1.5in, right=1.5in, bottom=1.5in, top=1in]{geometry}
\usepackage[utf8]{inputenc}

\usepackage{amsfonts,amsmath,amssymb,amsthm,mathtools}
\usepackage{enumitem}
\usepackage[hidelinks]{hyperref}
\usepackage{cleveref}
\usepackage{xcolor}
\crefname{subsection}{subsection}{subsections}
\numberwithin{equation}{section}

\usepackage[backend=bibtex,style=numeric,natbib=true,doi=false,isbn=false,url=false,eprint=false,firstinits=true]{biblatex}
\renewbibmacro{in:}{}
\newcommand{\R}{\mathbb{R}}
\newcommand{\M}{\mathbb{M}}
\newcommand{\N}{\mathbb{N}}
\renewcommand{\d}[1]{\,\mathrm{d}#1}
\renewcommand{\div}{\mathrm{div}}

\newcommand{\Acal}{\mathcal{A}}

\newcommand{\Ical}{\mathrm{I}}
\newcommand{\Jcal}{\mathcal{J}}
\newcommand{\D}{\mathrm{D}}

\newcommand{\Fcal}{\mathcal{F}}
\newcommand{\Mcal}{\mathcal{M}}
\newcommand{\Tcal}{\mathcal{T}}

\newcommand{\PotRec}{\mathcal{R}}
\newcommand{\GrRec}{\mathcal{G}}
\newcommand{\RT}{\mathrm{RT}}
\newcommand{\CR}{\mathrm{CR}}
\newcommand{\pw}{\mathrm{pw}}
\newcommand{\osc}{\mathrm{osc}}

\newcommand{\sfrak}{\mathrm{s}}
\newcommand{\Scal}{\mathcal{S}}

\newtheorem{theorem}{Theorem}[section]
\newtheorem{lemma}[theorem]{Lemma}

\theoremstyle{remark}
\newtheorem{remark}[theorem]{Remark}

\newcounter{cntS}
\newcommand{\newcnstS}{%
	\refstepcounter{cntS}%
	\ensuremath{c_{\thecntS}}}
\newcommand{\cnstS}[1]{\ensuremath{c_{\ref{#1}}}}

\newcounter{cntL}
\newcommand{\newcnstL}{%
	\refstepcounter{cntL}%
	\ensuremath{C_{\thecntL}}}
\newcommand{\cnstL}[1]{\ensuremath{C_{\ref{#1}}}}

\title{Convergent adaptive hybrid higher-order\\ schemes for convex minimization}
\author{Carsten Carstensen\footnote{Department of mathematics, Humboldt-Universit\"at zu Berlin, Germany (cc@math.hu-berlin.de)}~~and Ngoc Tien Tran\footnote{Department of mathematics, Friedrich-Schiller-Universit\"at Jena, Germany (ngoc.tien.tran@uni-jena.de)}}
\date{\today}

\begin{document}
	\begin{titlepage}
		\maketitle
	\end{titlepage}
	\begin{abstract}
		\noindent This paper proposes two convergent adaptive mesh-refining algorithms for the hybrid high-order method in convex minimization problems with two-sided $p$-growth. Examples include the p-Laplacian, an optimal design problem in topology optimization, and the convexified double-well problem. The hybrid high-order method utilizes a gradient reconstruction in the space of piecewise Raviart-Thomas finite element functions without stabilization on triangulations into simplices or in the space of piecewise polynomials with stabilization on polytopal meshes. The main results imply the convergence of the energy and, under further convexity properties, of the approximations of the primal resp. dual variable. Numerical experiments illustrate an efficient approximation of singular minimizers and improved convergence rates for higher polynomial degrees. Computer simulations provide striking numerical evidence that an adopted adaptive HHO algorithm can overcome the Lavrentiev gap phenomenon even with empirical higher convergence rates.
	\end{abstract}
	\section{Introduction}\label{sec:introduction}
	Adaptive mesh-refining is vital in the computational sciences and engineering with optimal rates known in many linear problems \cite{Stevenson2007,CasconKreuzerNochettoSiebert2008,CarstensenFeischlPagePraetorius2014,CarstensenRabus2017}. Besides eigenvalue problems \cite{DaiXuZhou2008,CGedicke2012,CarstensenGallistlSchedensack2015,BoffiGallistlGardiniGastaldi2017} much less is known for stationary nonlinear PDEs. The few positive results in the literature concern mainly conforming FEM with plain convergence results \cite{Veeser2002,BartelsC2008,C2008,DieningKreuzer2008,CarstensenDolzmann2015}. An important exception is the $p$-Laplacian in \cite{BelenkiDieningKreuzer2012}, where the notion of a quasi-norm enables two-sided error control. The next larger class of convex minimization problems from \cite{CPlechac1997} emerged in the relaxation of non-convex minimization problems with enforced microstructures and this is in the focus of this paper. This class is characterized by a two-sided growth condition on a $C^1$ energy density $W$ with an additional convexity control that enables a unique stress $\D W(\D u)$ independent of the multiple minimizers $u$ on the continuous level. In fact, there is no further control of the convex closed set of minimizers in beyond a priori boundedness. This leads to the reliability-efficiency gap \cite{CJochimsen2003} in the a~posteriori error control: If the mesh-size tends to zero, the known guaranteed lower and upper error bounds converge with a different convergence rate. In other words, the efficiency index tends to infinity. This dramatic loss of sharp error control does not prevent convergence of an adaptive algorithm in general, but it makes the analysis of plain convergence much harder and seemingly disables any proof of optimal rates.

	The numerical experiments in \cite{DiPietroSpecogna2016,CarstensenTran2020} motivate this paper on the adaptive HHO. The only nonconforming scheme known to converge for general convex minimization problems is \cite{OrtnerPraetorius2011} for the first-order Crouzeix-Raviart schemes and, according to the knowledge of the authors, there is no contribution for the convergence of an adaptive higher-order nonconforming scheme for nonlinear PDEs in the literature. In fact, this paper is the first one to guarantee plain convergence even for linear PDEs for the HHO schemes at all. The reason is a negative power of the mesh-size in the stabilization terms that is overcome in dG schemes for linear PDEs by over-penalization in \cite{BonitoNochetto2010} to be close to conforming approximations (and then enable arguments for optimal convergence rates) and recently by generalized Galerkin solutions in a limit space in \cite{KreuzerGeorgoulis2021} for plain convergence.
	One advantage of the HHO methodology is the absence of a stabilization parameter and, hence, this argument is not employed in this paper.
	
	The main contributions of this paper are adaptive HHO methods with and without stabilization with guaranteed plain convergence for the class of convex minimization problems from \cite{CPlechac1997}. Three types of results are available for those schemes.
	\begin{enumerate}[wide]
		\item[(a)] If $W$ is $C^1$ and convex with two-sided $p$-growth, then the minimal discrete energies converge to the exact minimal energy.
		\item[(b)] If furthermore $W$ satisfies the convexity control in the class of degenerate convex minimization problems of \cite{CPlechac1997}, then the discrete stress approximations converge to the (unique) exact stress $\sigma$.
		\item[(c)] If $W$ is even strongly convex, then the discrete approximations of the gradients converge to the gradient $\D u$ of the (unique) exact minimizer $u$.
	\end{enumerate}
	The two-sided growth condition excludes problems that exhibit the Lavrentiev gap phenomenon \cite{Lavrentiev1927} and so we only comment on the lowest-order schemes that overcome the Lavrentiev gap owing to the Jensen inequality.
	
	Numerical experiments are carried out on simplicial meshes, but the design of the stabilized HHO method allows for polytopal meshes for a fairly flexible mesh-design, e.g., in 3D. The mesh-refinement of those schemes is less elaborated (e.g., in comparison with \cite{Stevenson2008} on simplicial meshes) and remains as an important aspect for future research.

	The remaining parts of this paper are organized as follows.
	\Cref{sec:main-results} introduces the continuous minimization problem, the adaptive mesh-refining algorithm, and the main results of this paper.
	\Cref{sec:HHO} reviews the discretization with the HHO methodology on simplicial triangulations without stabilization. \Cref{sec:proof-of-main-results} departs from discrete compactness,  proves the plain convergence of an adaptive scheme, and concludes with an application to the Lavrentiev gap.
	\Cref{sec:polytopes} treats  HHO methods on general polytopal meshes with stabilization and proves the results of \Cref{sec:plain-convergence}. Numerical results for three model examples from \Cref{sec:examples} below are presented in \Cref{sec:numerical-examples} with conclusions drawn from the numerical experiments. 
	
	\section{Mathematical setting and main results}\label{sec:main-results}
	This paper analyzes the convergence of an adaptive mesh-refining algorithm based on the hybrid high-order methodology \cite{DiPietroErnLemaire2014,DiPietroErn2015,DiPietroDroniou2020} for convex minimization problems with a two-sided $p$-growth.
	
	\subsection{Continuous problem}\label{sec:continuous-problem}
	Given a bounded polyhedral Lipschitz domain $\Omega \subset \R^n$ and $1 < p < \infty$,
	let $W \in C^1(\M)$ with $\M \coloneqq \R^{m \times n}$ satisfy
	\begin{enumerate}[wide, label = (A\arabic*)]
		\item (convexity) $W$ is convex;\label{assumption-1}
		\item (two-sided growth) $\newcnstS\label{cnst:growth-left-1}|A|^p - \newcnstS\label{cnst:growth-left-2} \leq W(A) \leq \newcnstS\label{cnst:growth-right-1}|A|^p + \newcnstS\label{cnst:growth-right-2}$ for all $A \in \M$.\label{assumption-2}
	\end{enumerate}
	The constants $\cnstS{cnst:growth-left-1}, \cnstS{cnst:growth-right-1} > 0$ and $\cnstS{cnst:growth-left-2}, \cnstS{cnst:growth-right-2} \geq 0$ are universal in this paper and independent of the argument $A \in \M$; the same universality applies to $\cnstS{cnst:cc-primal}, \cnstS{cnst:cc-stress}$ in \eqref{ineq:cc-primal}--\eqref{ineq:cc-stress}.
	Throughout this paper, the boundary $\partial \Omega$ of the domain $\Omega$ is divided into a compact Dirichlet part $\Gamma_\mathrm{D}$ with positive surface measure and a relatively open (and possibly empty) Neumann part $\Gamma_\mathrm{N} = \partial \Omega \setminus \Gamma_\mathrm{D}$. Given $f \in L^{p'}(\Omega;\R^m)$, $g \in L^{p'}(\Gamma_\mathrm{N};\R^m)$ with $1/p + 1/p' = 1$, and $u_\mathrm{D} \in V \coloneqq W^{1,p}(\Omega;\R^m)$, minimize the energy functional
	\begin{align}
		E(v) \coloneqq \int_\Omega (W(\D v) - f \cdot v) \d{x} - \int_{\Gamma_\mathrm{N}} g \cdot v \d{s}\label{def:energy}
	\end{align}
	amongst admissible functions $v \in \Acal \coloneqq u_\mathrm{D} + V_\mathrm{D}$ subject to the Dirichlet boundary condition $v|_{\Gamma_\mathrm{D}} = u_\mathrm{D}|_{\Gamma_\mathrm{D}}$ and $V_\mathrm{D} \coloneqq  \{v \in V: v|_{\Gamma_{\mathrm{D}}} \equiv 0\}$.
	
	\subsection{Adaptive hybrid high-order method (AHHO)}\label{sec:adaptive_algorithm}
	The adaptive algorithm computes a sequence of discrete approximations of the minimal energy $\min E(\Acal)$ in the affine space $\Acal = u_\mathrm{D} + V_\mathrm{D}$ of admissible functions in a successive loop over the steps outlined below. The first version of the adaptive algorithm focuses on the newest-vertex-bisection (NVB) \cite{Stevenson2008} and the first HHO method without stabilization on triangulations into simplices. It will be generalized to polytopal meshes in \Cref{sec:polytopes}.
	
	\begin{enumerate}[wide]
		\item[1. \texttt{INPUT.}] The input is a regular initial triangulation $\Tcal_0$ of $\Omega$ into simplices, a polynomial degree $k \geq 0$, a positive parameter $0 < \varepsilon \leq k+1$, and a bulk parameter $0 < \theta < 1$.
		\item[2. \texttt{SOLVE.}] Let $\Tcal_\ell$ denote the triangulation associated to the level $\ell \in \N_0$ with the set of all sides $\Fcal_\ell$. The hybrid high-order method utilizes the discrete ansatz space $V(\Tcal_\ell) \coloneqq P_k(\Tcal_\ell;\R^m) \times P_k(\Fcal_\ell;\R^m)$ with a split of the discrete variables $v_\ell = (v_{\Tcal_\ell}, v_{\Fcal_\ell})$ into a volume variable $v_{\Tcal_\ell} \in P_k(\Tcal_\ell;\R^m)$ and a skeleton variable $v_{\Fcal_\ell} \in P_k(\Fcal_\ell;\R^m)$ of polynomial degree at most $k \geq 0$ with respect to the simplices ($\Tcal_\ell$) and the sides ($\Fcal_\ell$) in the triangulation $\Tcal_\ell$. The proposed numerical scheme replaces $\D v$ in \eqref{def:energy} by a gradient reconstruction $\GrRec_\ell$ in the space of piecewise Raviart-Thomas finite element functions $\Sigma(\Tcal_\ell) = \RT_k^\pw(\Tcal_\ell;\M)$ for a shape-regular triangulation $\Tcal_\ell$ of $\Omega$ into simplices. The details on the gradient reconstruction $\GrRec_\ell$ are postponed to \Cref{sec:reconstruction-operators}. The discrete problem computes a discrete minimizer $u_\ell$ of
		\begin{align}
			E_\ell(v_\ell) \coloneqq \int_\Omega (W(\GrRec v_\ell) - f \cdot v_{\Tcal_\ell}) \d{x} - \int_{\Gamma_\mathrm{N}} g \cdot v_{\Fcal_\ell} \d{s}
			\label{def:discrete_energy}
		\end{align}
		among $v_\ell = (v_{\Tcal_\ell}, v_{\Fcal_\ell}) \in \Acal(\Tcal_\ell)$ with the discrete analog $\Acal(\Tcal_\ell)$ of $\Acal$ so that $v_{\Fcal_\ell}|_F = \Pi_F^k u_\mathrm{D}$ for any Dirichlet side $F \in \Fcal_\ell(\Gamma_{\mathrm{D}})$, where $\Pi_F^k$ is the $L^2$ projection onto the polynomials $P_k(F)$ of degree at most $k$. Let $\sigma_\ell \coloneqq \Pi_{\Sigma(\Tcal_\ell)} \D W(\GrRec u_\ell) \in \Sigma(\Tcal_\ell)$ be the $L^2$ projection of $\D W(\GrRec u_\ell)$ onto $\Sigma(\Tcal_\ell)$.
		Further details on the hybrid high-order method follow in \Cref{sec:HHO} below.
		
		\item[3. \texttt{REFINEMENT INDICATORS}.] The computation of the refinement indicator $\eta_\ell$ utilizes an elliptic potential reconstruction $\PotRec_\ell u_\ell \in P_{k+1}(\Tcal_\ell;\R^m)$ of the discrete minimizer $u_\ell = (u_{\Tcal_\ell}, u_{\Fcal_\ell}) \in \Acal(\Tcal_\ell)$ computed in \texttt{SOLVE}. The definition of $\PotRec_\ell u_\ell$ follows in \eqref{def:potential_reconstruction}--\eqref{def:potential-reconstruction-integral-mean} below. Any interior side $F \in \Fcal_\ell(\Omega)$ is shared by two simplices $T_+, T_- \in \Tcal_\ell$ with $F = T_+ \cap T_-$. The jump $[\PotRec_\ell u_\ell]_F$ along $F$ is defined by $[\PotRec_\ell u_\ell]_F \coloneqq (\PotRec_\ell u_\ell)|_{T_+} - (\PotRec_\ell u_\ell)|_{T_-} \in P_{k+1}(F;\R^m)$. Given a positive parameter $0 < \varepsilon \leq k+1$, compute the local refinement indicator
		\begin{align}
			\eta_\ell^{(\varepsilon)}(T) &\coloneqq |T|^{(\varepsilon p - p)/n}\|\Pi_T^k(\PotRec_\ell u_\ell - u_T)\|_{L^p(T)}^p + |T|^{\varepsilon p'/n}\|\sigma_\ell - \D W(\GrRec u_\ell)\|_{L^{p'}(T)}^{p'}\nonumber\\
			&\quad + |T|^{p'/n}\|(1 - \Pi_{T}^k) f\|_{L^{p'}(T)}^{p'} + |T|^{1/n} \sum_{F \in \Fcal_\ell(T) \cap \Fcal_\ell(\Gamma_{\mathrm{N}})} \|(1 - \Pi_{F}^k) g\|_{L^{p'}(F)}^{p'}\nonumber\\
			&\quad + |T|^{(\varepsilon p + 1 - p)/n} \Big(\sum_{F \in \Fcal_\ell(T) \cap \Fcal_\ell(\Gamma_{\mathrm{D}})} \|\PotRec_\ell u_\ell - u_\mathrm{D}\|_{L^p(F)}^p\label{def:eta}\\
			&\quad + \sum_{F \in \Fcal_\ell(T) \cap \Fcal_\ell(\Omega)} \|[\PotRec_\ell u_\ell]_F\|_{L^p(F)}^p + \sum_{F \in \Fcal_\ell(T)} \|\Pi_F^k ((\PotRec_\ell u_\ell)|_T - u_F)\|_{L^p(F)}^p\Big)\nonumber
		\end{align}
		for all $T \in \Tcal_\ell$ of volume $|T|$ and sides $\Fcal_\ell(T)$ with the abbreviation $u_T \coloneqq u_{\Tcal_\ell}|_T$ and $u_F \coloneqq u_{\Fcal_\ell}|_F$.
		Let $\eta_\ell^{(\varepsilon)} \coloneqq \sum_{T \in \Tcal_\ell} \eta_\ell^{(\varepsilon)}(T)$.
		The refinement indicator is motivated by the discrete compactness from \Cref{thm:conv_analysis:discrete_compactness} below. In fact, if $\lim_{\ell \to \infty} \eta_\ell^{(\varepsilon)} = 0$, then there exists a $v \in \mathcal{A}$ such that, up to a subsequence, $\GrRec_\ell u_\ell \rightharpoonup \nabla v$ weakly in $L^p(\Omega;\M)$ and $u_{\Tcal_\ell} \rightharpoonup v$ weakly in $L^p(\Omega;\R^m)$ as $\ell \to \infty$. It turns out that $v$ is a minimizer of the continuous energy $E$ from \eqref{def:energy}.
		
		\item[4. \texttt{MARK and REFINE}.] Given a positive bulk parameter $0 < \theta < 1$, select a subset $\mathfrak{M}_\ell \subset \Tcal_\ell$ of minimal cardinality such that
		\begin{align}
			\theta \eta_\ell^{(\varepsilon)} \leq \eta_\ell^{(\varepsilon)}(\mathfrak{M}_\ell) \coloneqq \sum_{T \in \mathfrak{M}_\ell} \eta_\ell^{(\varepsilon)}(T).\label{ineq:Doerfler-marking}
		\end{align}
		This marking strategy is known as D\"orfler marking. The marked simplices are refined by the newest-vertex bisection \cite{Stevenson2008} to define $\Tcal_{\ell+1}$. 
		\item[5. \texttt{OUTPUT.}] The output is a sequence of shape-regular triangulations $(\Tcal_\ell)_{\ell \in \N_0}$, the corresponding discrete minimizers $(u_\ell)_{\ell \in \N_0}$, discrete stresses $(\sigma_\ell)_{\ell \in \N_0}$, and the refinement indicators $(\eta_\ell^{(\varepsilon)})_{\ell \in \N_0}$. On each level $\ell \geq 0$, let $\Jcal_{\ell} u_\ell \in V$ denote the conforming post-processing of $u_\ell$ from \Cref{lem:conforming-companion} below.
	\end{enumerate}

	\subsection{Main results}\label{sec:main-results-theorem}
	The main results establish the convergence of the sequence $(E_\ell(u_\ell))_{\ell \in \N_0}$ of minimal discrete energies computed by AHHO towards the exact minimal energy.
	\begin{theorem}[plain convergence]\label{thm:plain-convergence}
		Given the input $\Tcal_0$, $k \in \N_0$, $0 < \varepsilon \leq k+1$, $0 < \theta < 1$, let $(\Tcal_\ell)_{\ell \in \N_0}$, $(u_\ell)_{\ell \in \N_0}$, and $(\sigma_\ell)_{\ell \in \N_0}$ be the output of the adaptive algorithm AHHO from \Cref{sec:adaptive_algorithm}.
		Assume that $W$ satisfies \ref{assumption-1}--\ref{assumption-2}, then \ref{thm:convergence-a}--\ref{thm:convergence-d} hold.
		\begin{enumerate}[wide, label = (\alph*)]
			\item\label{thm:convergence-a} $\lim_{\ell \to \infty} E_\ell(u_\ell) = \min E(\Acal)$.
			\item\label{thm:convergence-b} The sequence of the post-processing $(\Jcal_\ell u_\ell)_{\ell \in \N_0}$ is bounded in $V = W^{1,p}(\Omega;\R^n)$ and any weak accumulation point of $(\Jcal_\ell u_\ell)_{\ell \in \N_0}$ in $V$ minimizes $E$ in $\Acal$.
			\item\label{thm:convergence-c} Suppose there exists $\newcnstS\label{cnst:cc-primal} > 0$ such that $W$ satisfies, for all $A,B \in \M$,
			\begin{align}
				|A - B|^r \leq \cnstS{cnst:cc-primal}(1 + |A|^s + |B|^s)(W(A) - W(B) - \D W(B):(A - B))
				\label{ineq:cc-primal}
			\end{align}
			with parameters $r, s$ from \Cref{tab:parameters}. Then the minimizer $u$ of $E$ in $\Acal$ is unique and $\lim_{\ell \to \infty} \GrRec_\ell u_\ell = \D u$ (strongly) in $L^p(\Omega;\M)$ holds for the entire sequence.
			\item\label{thm:convergence-d} Suppose there exists $\newcnstS\label{cnst:cc-stress} > 0$ such that $W$ satisfies, for all $A,B \in \M$,
			\begin{align}
				\begin{split}
					|\D W(A) - \D W(B)|^{\widetilde{r}} &\leq \cnstS{cnst:cc-stress}(1 + |A|^{\widetilde{s}} + |B|^{\widetilde{s}})\\
					&\quad \times (W(A) - W(B) - \D W(B):(A - B))
				\end{split}\label{ineq:cc-stress}
			\end{align}
			with parameters $\widetilde{r}, \widetilde{s}$ from \Cref{tab:parameters}. Then the stress $\sigma \coloneqq \D W(\D u) \in L^{p'}(\Omega;\M)$ is unique (independent of the choice of a (possibly nonunique) minimizer $u$) and
			$\lim_{\ell \to \infty} \D W(\GrRec_\ell u_\ell) = \sigma$ (strongly) in $L^{p'}(\Omega;\M)$ and $\sigma_\ell \rightharpoonup \sigma$ (weakly) in $L^{p'}(\Omega;\M)$ hold for the entire sequence.
		\end{enumerate}
	\end{theorem}
	\begin{table}
		\centering
		\begin{tabular}{|c|c|c|c|c|c|c|}
			\hline
			case & $r$ & $s$ & $t$ & $\widetilde{r}$ & $\widetilde{s}$ & $\widetilde{t}$\\
			\hline
			$2 \leq p < \infty$ & $p$ & 0 & 1 & 2 & $p-2$ & $2/p'$\\
			\hline
			$1 < p < 2$ & 2 & $2 - p$ & $2/p$ & $p'$ & 0 & 1\\
			\hline
		\end{tabular}
		\caption{Parameters $r$, $s$, $\widetilde{r}$, $\widetilde{s}$ in \Cref{thm:plain-convergence} and $t,\widetilde{t}$ in \Cref{sec:plain-convergence}}
		\label{tab:parameters}
	\end{table}
	\noindent A second focus is on the classical HHO method \cite{DiPietroErnLemaire2014,DiPietroErn2015,DiPietroDroniou2017} on general polytopal meshes $\Mcal_\ell$ with a stabilization $\sfrak_\ell(\bullet,\bullet)$ defined in \eqref{def:stabilization} below. 
	The convergence of AHHO for the stabilized HHO method on polytopal meshes is established under two assumptions \ref{assumption-mesh-1}--\ref{assumption-mesh-2}. 
	Further details on \ref{assumption-mesh-1}--\ref{assumption-mesh-2} and on the stabilized HHO method follow in \Cref{sec:polytopes}.
	\begin{theorem}[plain convergence for stabilized HHO]\label{thm:plain-convergence-sHHO}
		Given the input $\Mcal_0$, $k \in \N_0$, $0 < \varepsilon \leq \min\{k+1, (k+1)/(p-1)\}$, $0 < \theta < 1$, let $(\Mcal_\ell)_{\ell \in \N_0}$, $(u_\ell)_{\ell \in \N_0}$, and $(\sigma_\ell)_{\ell \in \N_0}$ be the output of the adaptive algorithm from \Cref{sec:adaptive_algorithm}. Suppose that \ref{assumption-mesh-1}--\ref{assumption-mesh-2} hold, then \ref{thm:convergence-a}--\ref{thm:convergence-d} from \Cref{thm:plain-convergence} hold verbatim and $\lim_{\ell \to \infty} \sfrak_\ell(u_\ell;u_\ell) = 0$.
	\end{theorem}
	\noindent Notice that an additional restriction on the parameter $\varepsilon$ is imposed in \Cref{thm:plain-convergence-sHHO} to control the stabilization $\sfrak_{\ell}$. The proofs of \Cref{thm:plain-convergence} and \ref{thm:plain-convergence-sHHO} are postponed to \Cref{sec:proof-of-main-results} and \ref{sec:polytopes}.
	
	\subsection{Examples}\label{sec:examples}
	\Cref{thm:plain-convergence} applies to the following scalar examples with $m = 1$.
	\subsubsection{$p$-Laplace}\label{ex:pLaplace}
	The minimization of the energy $E:\Acal \to \R$ with the energy density $W \in C^1(\R^n)$,
	\begin{align*}
		W(a) \coloneqq |a|^p/p \quad\text{for any } a\in \R^n \text{ with } 1 < p < \infty,
	\end{align*}
	is related to the nonlinear PDE $- \div \sigma = f \in L^{p'}(\Omega)$ with $\sigma \coloneqq \nabla W(\nabla u) = |\nabla u|^{p-2}\nabla u \in L^{p'}(\Omega;\R^n)$ subject to the boundary conditions $\sigma \nu = g$ on $\Gamma_{\mathrm{N}}$ and $u = u_\mathrm{D}$ on $\Gamma_{\mathrm{D}}$. The energy density $W$ satisfies \ref{assumption-1}--\ref{assumption-2} and \eqref{ineq:cc-primal}--\eqref{ineq:cc-stress} \cite{GlowinskiMarrocco1975,CKlose2003}.
	It is worth noticing that the convergence results of this paper are new even for a linear model problem with $p = 2$ for the two HHO algorithms.
	\subsubsection{Optimal design problem}\label{ex:ODP}
	The optimal design problem seeks the optimal distribution of two materials with fixed amounts to fill a given domain for maximal torsion stiffness \cite{KohnStrang1986,BartelsC2008}. For fixed parameters $0 < \xi_1 < \xi_2$ and $0 < \mu_1 < \mu_2$ with $\xi_1\mu_2 = \xi_2\mu_1$, the energy density $W(a) \coloneqq \psi(\xi)$, $a \in \R^n$, $\xi \coloneqq |a| \geq 0$ with
	\begin{align*}
		\psi(\xi) \coloneqq \begin{cases}
			\mu_2 \xi^2/2 &\mbox{if } 0 \leq \xi \leq \xi_1,\\
			\xi_1\mu_2(\xi - \xi_1/2) &\mbox{if } \xi_1 \leq \xi \leq \xi_2,\\
			\mu_1 \xi^2/2 - \xi_1\mu_2(\xi_1/2 - \xi_2/2) &\mbox{if } \xi_2 \leq \xi
		\end{cases}
	\end{align*}
	satisfies \ref{assumption-1}--\ref{assumption-2} and \eqref{ineq:cc-stress} \cite[Prop.~4.2]{BartelsC2008}.
	\subsubsection{Relaxed two-well problem}\label{ex:2well}
	Given distinct $F_1, F_2 \in \R^n$ in the two-well problem of \cite{ChipotCollins1992}, the convex envelope $W$ of $|F - F_1|^2|F - F_2|^2$ for $F \in \R^n$ reads
	\begin{align*}
		W(F) = \max\{0,|F - B|^2 - |A|^2\}^2 + 4\big(|A|^2 |F-B|^2 - (A \cdot (F - B))^2 \big)
	\end{align*}
	with $A = (F_2 - F_1)/2$, $B = (F_1 + F_2)/2$, and satisfies \ref{assumption-1}--\ref{assumption-2} and \eqref{ineq:cc-stress} \cite{CPlechac1997,C2008}.
	
	\subsection{Notation}\label{sec:notation}
	Standard notation for Sobolev and Lebesgue functions applies throughout this paper with the abbreviations $V \coloneqq W^{1,p}(\Omega;\R^m) = W^{1,p}(\Omega)^m$ and $V_\mathrm{D} \coloneqq W^{1,p}_\mathrm{D}(\Omega;\R^m) = \{v \in V: v|_{\Gamma_{\mathrm{D}}} = 0\}$. In particular, $(\bullet, \bullet)_{L^2(\Omega)}$ denotes the scalar product of $L^2(\Omega)$ and $W^{p'}(\div,\Omega;\M) \coloneqq W^{p'}(\div,\Omega)^m$ is the matrix-valued version of
	\begin{align}
		W^{p'}(\div,\Omega) \coloneqq \{\tau \in L^{p'}(\Omega;\R^n): \div \tau \in L^{p'}(\Omega)\}.\label{def:Hdiv}
	\end{align}
	For any $A,B \in \M \coloneqq \R^{m \times n}$, $A:B$ denotes the Euclidean scalar product of $A$ and $B$, which induces the Frobenius norm $|A| \coloneqq (A:A)^{1/2}$ in $\M$.
	The context-depending notation $|\bullet|$ denotes the length of a vector, the Frobenius norm of a matrix, the Lebesgue measure of a subset of $\R^n$, or the counting measure of a discrete set.
	For $1 < p < \infty$, $p' = p/(p-1)$ denotes the H\"older conjugate of $p$ with $1/p + 1/p' = 1$.
	The notation $A \lesssim B$ abbreviates $A \leq CB$ for a generic constant $C$ independent of the mesh-size and $A \approx B$ abbreviates $A \lesssim B \lesssim A$.
	
	\section{Hybrid high-order method without stabilization}\label{sec:HHO}
	This section recalls the discrete ansatz space and reconstruction operators from the HHO methodology \cite{DiPietroErnLemaire2014, DiPietroErn2015, DiPietroDroniou2020} for convenient reading.
	
	\subsection{Triangulation}\label{sec:regular_triangulation}
	A regular triangulation $\Tcal_\ell$ of $\Omega$ in the sense of Ciarlet is a finite set of closed simplices $T$ of positive volume $|T| > 0$ with boundary $\partial T$ and outer unit normal $\nu_T$ such that $\cup_{T \in \Tcal_\ell} T = \overline{\Omega}$ and two distinct simplices are either disjoint or share one common (lower-dimensional) subsimplex (vertex or edge in 2D and vertex, edge, or face in 3D). 
	Let $\Fcal_\ell(T)$ denote the set of the $n+1$ hyperfaces of $T$, called sides of $T$. Define the set of all sides $\Fcal_\ell \coloneqq \cup_{T \in \Tcal_\ell} \Fcal_\ell(T)$, the set of interior sides $\Fcal_\ell(\Omega) \coloneqq \Fcal_\ell\setminus\{F \in \Fcal_\ell: F \subset \partial \Omega\}$,
	the set of Dirichlet sides $\Fcal_{\ell}(\Gamma_{\mathrm{D}}) \coloneqq \{F \in \Fcal_\ell: F \subset \Gamma_{\mathrm{D}}\}$, and the set of Neumann sides $\Fcal_{\ell}(\Gamma_{\mathrm{N}}) \coloneqq \{F \in \Fcal_\ell: F \subset \Gamma_{\mathrm{N}}\}$ of $\Tcal_\ell$.
	
	For any interior side $F \in \Fcal_\ell(\Omega)$, there exist exactly two simplexes $T_+, T_- \in \Tcal_\ell$ such that $\partial T_+ \cap \partial T_- = F$. The orientation of the outer normal unit $\nu_F = \nu_{T_+}|_F = -\nu_{T_-}|_F$ along $F$ is fixed beforehand. Define the side patch $\omega_F \coloneqq \mathrm{int}(T_+ \cup T_-)$ of $F$. Let $[v]_F \coloneqq (v|_{T_+})|_F - (v|_{T_-})|_F \in L^1(F)$ denote the jump of $v \in L^1(\omega_F)$ with $v \in W^{1,1}(T_+)$ and $v \in W^{1,1}(T_-)$ across $F$ (with the abbreviations $W^{1,1}(T_+) \coloneqq W^{1,1}(\mathrm{int}(T_+))$ and $W^{1,1}(T_-) \coloneqq W^{1,1}(\mathrm{int}(T_-))$).
	For any boundary side $F \in \Fcal_\ell(\partial \Omega) \coloneqq \Fcal_\ell \setminus \Fcal_\ell(\Omega)$, there is a unique $T \in \Tcal_\ell$ with $F \in \Fcal_\ell(T)$. Then $\omega_F = \mathrm{int}(T)$, $\nu_F \coloneqq \nu_T$, and $[v]_F \coloneqq (v|_T)|_F$.
	The differential operators $\div_{\pw}$ and $\D_\pw$ depend on the triangulation $\Tcal_\ell$ and denote the piecewise application of $\div$ and $\D$ without explicit reference to $\Tcal_\ell$.
	
	The shape regularity of a triangulation $\Tcal$ is the minimum $\min_{T \in \Tcal} \varrho(T)$ of all ratios $\varrho(T) \coloneqq r_i/r_c \leq 1$ of the maximal radius $r_i$ of an inscribed ball and the minimal radius $r_c$ of a circumscribed ball for a simplex $T \in \Tcal$.
	
	\subsection{Discrete spaces}\label{sec:discrete-spaces}
	The discrete ansatz space of the HHO methods consists of piecewise polynomials on the triangulation $\Tcal_\ell$ and on the skeleton $\partial \Tcal_\ell \coloneqq \cup \Fcal_\ell$. For a simplex or a side $M \subset \R^n$ of diameter $h_M$, let $P_k(M)$ denote the space of polynomials of degree at most $k \in \N_0$ regarded as functions defined in $M$.
	The $L^2$ projection $\Pi_M^k v \in P_k(M)$ of $v \in L^1(M)$ is defined by $\Pi_M^k v \in P_k(M)$ with
	\begin{align*}
		\int_{M} \varphi_k (1-\Pi_M^k)v \d{x} = 0 \quad\text{for any } \varphi_k \in P_k(M).
	\end{align*}
	The gradient reconstruction in $T \in \Tcal_\ell$ maps in the space of Raviart-Thomas finite element functions
	\begin{align*}
		\RT_k(T) &\coloneqq P_k(T;\R^n) + x P_k(T) \subset P_{k+1}(T;\R^n).
	\end{align*}
	Let $P_k(\Tcal_\ell)$, $P_k(\Fcal_\ell)$, and $\RT^\pw_k(\Tcal_\ell)$ denote the space of piecewise functions with respect to the mesh $\Tcal_\ell$ or $\Fcal_\ell$ and with restrictions to $T$ or $F$ in $P_k(T)$, $P_k(F)$, and $\RT_k(T)$. The $L^2$ projections $\Pi_{\Tcal_\ell}^k$ and $\Pi_{\Fcal_\ell}^k$ onto the discrete spaces $P_k(\Tcal_\ell)$ and $P_k(\Fcal_\ell)$ are the global versions of $\Pi_T^k$ and $\Pi_F^k$, e.g., $(\Pi_{\Tcal_\ell}^k v)|_T \coloneqq \Pi_T^k (v|_T)$ for $v \in L^1(\Omega)$. For vector-valued functions $v \in L^1(\Omega;\R^m)$, the $L^2$ projection $\Pi_{\Tcal_\ell}^k$ onto $P_k(\Tcal_\ell;\R^m) \coloneqq P_k(\Tcal_\ell)^m$ applies componentwise. This convention extends to the $L^2$ projections onto $P_k(M;\R^m) \coloneqq P_k(M)^m$ and $P_k(\Fcal_\ell;\R^m) \coloneqq P_k(\Fcal_\ell)^m$.
	The space of lowest-order Crouzeix-Raviart finite element functions reads
	\begin{align}
		\begin{split}
			\CR^1(\Tcal_\ell) \coloneqq \{v_\CR \in P_1(\Tcal_\ell) &: v_\CR \text{ is continuous}\\
			&~~\text{at midpoints of } F \text{ for all } F \in 	\Fcal_\ell(\Omega)\}.
		\end{split}
		\label{def:Crouzeix-Raviart}
	\end{align}
	Define the mesh-size function $h_\ell \in P_0(\Tcal_\ell)$ with $h_\ell|_T \equiv |T|^{1/n}$ for all $T \in \Tcal_\ell$, the (volume data) oscillation $\osc(f,\Tcal_\ell)^{p'} \coloneqq \sum_{T \in \Tcal_\ell} h_T\|(1 - \Pi_{T}^k) f\|_{L^{p'}(\Omega)}^{p'}$, and the (Neumann data) oscillation $\osc_\mathrm{N}(g,\Fcal_{\ell}(\Gamma_{\mathrm{N}}))^{p'} \coloneqq \sum_{F \in \Fcal_\ell(\Gamma_{\mathrm{N}})} h_F\|(1 - \Pi_F^k) g\|_{L^{p'}(F)}^{p'}$ with the diameter $h_F = \mathrm{diam}(F)$ of $F \in \Fcal_\ell$. (Notice that the shape regularity of $\Tcal_\ell$ implies the equivalence $h_F \approx h_T \approx |T|^{1/n}$ for all $T \in \Tcal_\ell, F \in \Fcal_\ell(T)$.)
	
	\subsection{HHO ansatz space}\label{sec:discrete_ansatz_space}
	For fixed $k \in \mathbb{N}_0$, let $V(\Tcal_\ell) \coloneqq P_k(\Tcal_\ell;\R^m) \times P_k(\Fcal_\ell;\R^m)$ denote the discrete ansatz space for $V$ in HHO methods \cite{DiPietroErnLemaire2014,DiPietroErn2015}.
	The notation $v_\ell \in V(\Tcal_\ell)$ means that $v_\ell = (v_{\Tcal_\ell},v_{\Fcal_\ell}) = ((v_T)_{T \in \Tcal_\ell},(v_F)_{F \in \Fcal_\ell})$ for some $v_{\Tcal_\ell} \in P_k(\Tcal_\ell;\R^m)$ and $v_{\Fcal_\ell} \in P_k(\Fcal_\ell;\R^m)$ with the identification $v_T \coloneqq v_{\Tcal_\ell}|_T \in P_k(T;\R^m)$ and $v_F \coloneqq v_{\Fcal_\ell}|_F \in P_k(F;\R^m)$ for all $T \in \Tcal_\ell$, $F \in \Fcal_\ell$.
	The discrete space $V(\Tcal_\ell)$ is endowed with the seminorm
	\begin{align}
		\|v_\ell\|_\ell^p \coloneqq \|\D_\pw v_{\Tcal_\ell}\|_{L^p(\Omega)}^p + \sum_{T \in \Tcal_\ell} \sum_{F \in \Fcal_\ell(T)} h_F^{1-p}\|v_T - v_F\|_{L^p(F)}^p
		\label{def:discrete-norm}
	\end{align}
	for any $v_\ell = (v_{\Tcal_\ell},v_{\Fcal_\ell}) \in V(\Tcal_\ell)$.
	The set $\Fcal_\ell \setminus \Fcal_\ell(\Gamma_{\mathrm{D}})$ of non-Dirichlet sides gives rise to the space $P_k(\Fcal_\ell \setminus \Fcal_\ell(\Gamma_{\mathrm{D}});\R^m)$ of piecewise polynomials $v_{\Fcal_\ell} \in P_k(\Fcal_\ell;\R^m)$ with the convention $v_{\Fcal_\ell}|_F \equiv 0$ on $F \in \Fcal_\ell(\Gamma_{\mathrm{D}})$ to model homogenous Dirichlet boundary conditions along the side $F \subset \Gamma_{\mathrm{D}}$.
	The discrete linear space $V_\mathrm{D}(\Tcal_\ell) \coloneqq P_k(\Tcal_\ell;\R^m) \times P_k(\Fcal_\ell \setminus \Fcal_\ell(\Gamma_{\mathrm{D}});\R^m) \subset V(\Tcal_\ell)$, equipped with the norm $\|\bullet\|_\ell$ from \eqref{def:discrete-norm}, is the discrete analogue to $V_\mathrm{D} = W^{1,p}_\mathrm{D}(\Omega;\R^m)$.
	The interpolation
	\begin{align}
		\Ical_\ell : V \to V(\Tcal_\ell), v \mapsto (\Pi_{\Tcal_\ell}^k v, \Pi_{\Fcal_\ell}^k v)
		\label{def:interpolation}
	\end{align}
	gives rise to the discrete space $\Acal(\Tcal_\ell) \coloneqq \Ical_\ell u_\mathrm{D} + V_{\mathrm{D}}(\Tcal_\ell)$ of admissible functions.
	
	\subsection{Reconstruction operators}\label{sec:reconstruction-operators}
	The reconstruction operators defined in this section link the two components of $v_\ell \in V(\Tcal_\ell)$ and provide discrete approximations $\PotRec_\ell v_\ell$ and $\GrRec_\ell v_\ell$ of the displacement $v \in V$ and its derivative $\D v \in L^2(\Omega;\M)$.
	\paragraph{Potential reconstruction.}
	Given $T \in \Tcal_\ell$ and $v_\ell = (v_{\Tcal_\ell},v_{\Fcal_\ell}) \in V(\Tcal_\ell)$ with the convention $v_T = v_{\Tcal_\ell}|_T$ and $v_F = v_{\Fcal_\ell}|_F$ for all $F \in \Fcal_\ell(T)$ from \Cref{sec:discrete_ansatz_space}, the local potential reconstruction $\PotRec_T v_\ell \in P_{k+1}(T;\R^m)$ satisfies
	\begin{align}
		\begin{split}
			&\int_T \D \PotRec_T v_\ell : \D \varphi_{k+1} \d{x} \\
			&\qquad = -\int_T \Delta \varphi_{k+1} \cdot v_T \d{x} + \sum_{F \in \Fcal(T)} \int_F v_F \cdot (\D \varphi_{k+1} \nu_T)|_F \d{s}
		\end{split}\label{def:potential_reconstruction}
	\end{align}
	for all $\varphi_{k+1} \in P_{k+1}(T;\R^m)$.
	The bilinear form $(\D \bullet, \D \bullet)_{L^2(T)}$ on the left-hand side of \eqref{def:potential_reconstruction} defines a scalar product in the quotient space $P_{k+1}(T;\R^m)/\R^m$ and the right-hand side of \eqref{def:potential_reconstruction} is a linear functional in $P_{k+1}(T;\R^m)/\R^m$. The Riesz representation $\PotRec_T v_\ell \in P_{k+1}(T;\R^m)$ of this linear functional in $P_{k+1}(T;\R^m)/\R^m$ equipped with the energy scalar product is selected by
	\begin{align}
		\int_T \PotRec_T v_\ell \d{x} = \int_{T} v_T \d{x}.
		\label{def:potential-reconstruction-integral-mean}
	\end{align}
	The unique solution $\PotRec_T v_\ell \in P_{k+1}(T;\R^m)$ to \eqref{def:potential_reconstruction}--\eqref{def:potential-reconstruction-integral-mean} gives rise to the potential reconstruction operator $\PotRec_\ell : V(\Tcal_\ell) \to P_{k+1}(\Tcal_\ell;\R^m)$ with restriction $(\PotRec_\ell v_\ell)|_T \coloneqq \PotRec_T v_\ell$ on each simplex $T \in \Tcal_\ell$ for any $v_\ell \in V(\Tcal_\ell)$.
	
	\paragraph{Gradient reconstruction.}
	The gradient is reconstructed in the space $\Sigma(\Tcal_\ell) = \RT_k^\pw(\Tcal_\ell;\M)$ of piecewise Raviart-Thomas finite element functions \cite{AbbasErnPignet2018,CarstensenTran2020}. Given $v_\ell = (v_{\Tcal_\ell},v_{\Fcal_\ell}) \in V(\Tcal_\ell)$, its gradient reconstruction $\GrRec_\ell v_\ell \in \Sigma(\Tcal_\ell)$ solves
	\begin{align}
		\int_{\Omega} \GrRec_\ell v_\ell:\tau_\ell \d{x} &= -\int_{\Omega} v_{\Tcal_\ell} \cdot \div_\pw \tau_\ell \d{x} + \sum_{F \in \Fcal_\ell} \int_F v_F \cdot [\tau_\ell \nu_F]_F \d{s}
		\label{def:gradient_reconstruction}
	\end{align}
	for all $\tau_\ell \in \Sigma(\Tcal_\ell)$.
	In other words, $\GrRec_\ell v_\ell$ is the Riesz representation of the linear functional on the right-hand side of \eqref{def:gradient_reconstruction} in the Hilbert space $\Sigma(\Tcal_\ell)$ endowed with the $L^2$ scalar product. Since $\D_\pw P_{k+1}(\Tcal_\ell;\R^m) \subset \Sigma(\Tcal_\ell)$, it follows that $\D_\pw \PotRec_\ell v_\ell$ is the $L^2$ projection of $\GrRec_\ell v_\ell$ onto $\D_\pw P_{k+1}(\Tcal_\ell;\R^m)$.
	\begin{lemma}[properties of $\GrRec$]\label{lem:stability_gradient_reconstruction}
		Any $v \in V$ and $v_\ell \in V(\Tcal_\ell)$ satisfy (a) $\|v_\ell\|_\ell \approx
		\|{\GrRec_\ell v_\ell}\|_{L^p(\Omega)}$ and (b) $\Pi_{\Sigma(\Tcal_\ell)} \D v = \GrRec_\ell \Ical_\ell v$. There exist positive constants $C_{\mathrm{dF}}$ and $C_{\mathrm{dtr}}$ that only depend on $\Omega$, the shape regularity of $\Tcal_\ell$, $k$, and $p$ such that (c) $\|v_{\Tcal_\ell}\|_{L^p(\Omega)} \leq C_{\mathrm{dF}}\|{\GrRec_\ell v_\ell}\|_{L^p(\Omega)}$ and (d) $\|v_{\Fcal_\ell}\|_{L^p(\Gamma_\mathrm{N})} \leq C_{\mathrm{dtr}}\|{\GrRec_\ell v_\ell}\|_{L^p(\Omega)}$ hold for all $v_\ell = (v_{\Tcal_\ell},v_{\Fcal_\ell}) \in V_\mathrm{D}(\Tcal_\ell)$.
	\end{lemma}
	\begin{proof}
		The proofs of (a)--(b) are outlined in \cite{AbbasErnPignet2018,CarstensenTran2020}. The discrete Sobolev embedding $\|v_{\Tcal_\ell}\|_{L^p(\Omega)} \lesssim \|v_\ell\|_\ell$ follows as in \cite{BuffaOrtner2009,DiPietroErn2010,DiPietroDroniou2017}. Theorem 4.4 in \cite{BuffaOrtner2009} and (c) lead to $\|v_{\Tcal_\ell}\|_{L^p(\Gamma_{\mathrm{N}})} \lesssim \|v_\ell\|_\ell$.
		This and the triangle inequality $\|v_{\Fcal_\ell}\|_{L^p(\Gamma_{\mathrm{N}})} \leq \|v_{\Tcal_\ell}\|_{L^p(\Gamma_{\mathrm{N}})} + \|v_{\Tcal_\ell} - v_{\Fcal_\ell}\|_{L^p(\Gamma_{\mathrm{N}})}$ imply $\|v_{\Fcal_\ell}\|_{L^p(\Gamma_{\mathrm{N}})} \lesssim \|v_\ell\|_\ell + \|v_{\Tcal_\ell} - v_{\Fcal_\ell}\|_{L^p(\Gamma_{\mathrm{N}})}$. The latter term is controlled by $\mathrm{diam}(\Omega)^{1/p'}\|v_\ell\|_\ell$. This concludes the proof of (d).
	\end{proof}
	\subsection{Discrete problem}
	\Cref{lem:stability_gradient_reconstruction} implies the coercivity of $E_\ell$ in $\Acal(\Tcal_\ell)$ with respect to the discrete seminorm $\|\GrRec_\ell \bullet\|_{L^p(\Omega)}$ and the existence and the boundedness of discrete minimizers $u_\ell$ below.
	\begin{theorem}[discrete minimizers]\label{lem:existence-discrete-solutions}
		The minimal discrete energy $\inf E_\ell(\Acal(\Tcal_\ell))$ is attained. There exists a positive constant $\newcnstL\label{cnst:G-u-h} > 0$ that depends only on \cnstS{cnst:growth-left-1}, \cnstS{cnst:growth-left-2}, $\Omega$, $\Gamma_{\mathrm{D}}$, $u_\mathrm{D}$, $f$, $g$, the shape regularity of $\Tcal_\ell$, $k$, and $p$ with $\|\GrRec_\ell u_\ell\|_{L^p(\Omega)} \leq \cnstL{cnst:G-u-h}$ for all discrete minimizers $u_\ell \in \arg\min E_\ell(\Acal(\Tcal_\ell))$. Any discrete stress $\sigma_\ell \coloneqq \Pi_{\Sigma(\Tcal_\ell)} \D W(\GrRec_\ell u_\ell) \in L^{p'}(\Omega;\M)$ satisfies the discrete Euler-Lagrange equations
			\begin{align}
			\int_\Omega \sigma_\ell : \GrRec_\ell v_\ell \d{x} = \int_\Omega f \cdot v_{\Tcal_\ell} \d{x} + \int_{\Gamma_{\mathrm{N}}} g \cdot v_{\Fcal_\ell} \d{s}
			\label{eq:dELE}
		\end{align}
		for all $v_\ell = (v_{\Tcal_\ell}, v_{\Fcal_\ell}) \in V_\mathrm{D}(\Tcal_\ell)$.
		If $W$ satisfies \eqref{ineq:cc-primal}, then $u_\ell = \arg \min E_\ell(\Acal(\Tcal_\ell))$ is unique. If $W$ satisfies \eqref{ineq:cc-stress}, then $\D W(\GrRec_\ell u_\ell) \in L^{p'}(\Omega;\M)$ is unique (independent of the choice of a (possibly non-unique) discrete minimizer $u_\ell$).
	\end{theorem}
	\begin{proof}
		The boundedness $\inf E_\ell(\Acal(\Tcal_\ell)) > - \infty$ of $E_\ell$ in $\Acal(\Tcal_\ell)$ follows from the lower $p$-growth of $W$, the discrete Friedrichs, and the discrete trace inequality from \Cref{lem:stability_gradient_reconstruction}, cf., e.g., \cite{CPlechac1997,Dacorogna2008,CarstensenTran2020}. The direct method in the calculus of variations \cite{Dacorogna2008} implies the existence of discrete minimizers $u_\ell \in \arg\min E_\ell(\Acal(\Tcal_\ell))$. The bound $\|\GrRec_\ell u_\ell\|_{L^p(\Omega)} \leq \cnstL{cnst:G-u-h}$ is a consequence of the coercivity of $E_\ell$ in $\Acal(\Tcal_\ell)$ with respect to $\|\bullet\|_\ell$ as in \cite{CarstensenTran2020}.
		If $W$ satisfies \eqref{ineq:cc-primal}, then $W$ is strictly convex and the discrete minimizer $u_\ell \in \arg\min E_\ell(\Acal(\Tcal_\ell))$ is unique. If $W$ satisfies \eqref{ineq:cc-stress}, then the uniqueness of $\D W(\GrRec_\ell u_\ell)$ follows as in \cite{CPlechac1997,CLiu2015,CarstensenTran2020}.
	\end{proof}
	\begin{remark}[$H(\div)$ conformity]
		The discrete Euler-Lagrange equations \eqref{eq:dELE} imply the continuity of the normal jumps $[\sigma_\ell \nu_F]_F$ of $\sigma_\ell = \Pi_{\Sigma(\Tcal_\ell)} \D W(\GrRec_\ell u_\ell)$ along all interior side $F \in \Fcal_\ell(\Omega)$ \cite[Theorem 3.2]{CarstensenTran2020}. In other words, $\sigma_\ell \in \Sigma(\Tcal_\ell) \cap W^{p'}(\div,\Omega;\M)$ with $\div \sigma_\ell = -\Pi_{\Tcal_\ell}^k f$ and $\sigma_\ell \nu_F = \Pi_F^k g$ for all $F \in \Fcal_\ell(\Gamma_{\mathrm{N}})$.
	\end{remark}
	
	\subsection{Conforming companion}\label{sec:conforming-companion}
	The companion operator $\Jcal_\ell : V(\Tcal_\ell) \to V$ is a right-inverse of the interpolation $\Ical_\ell : V \to V(\Tcal_\ell)$ in spirit of \cite{CarstensenGallistlSchedensack2015,CPuttkammer2020,ErnZanotti2020,CarstensenNataraj2021}. In particular, $\Jcal_\ell$ preserves the moments
	\begin{align}
		\Pi_{\Tcal_\ell}^k \Jcal_\ell v_\ell = v_{\Tcal_\ell} \quad\text{and}\quad \Pi_{\Fcal_\ell}^k \Jcal_\ell v_\ell = v_{\Fcal_\ell} \quad\text{for any } v_\ell = (v_{\Tcal_\ell},v_{\Fcal_\ell}) \in V(\Tcal_\ell).\label{def:conforming-companion}
	\end{align}
	An explicit construction of $\Jcal_\ell v_\ell$ on simplicial meshes is presented in \cite[Section 4.3]{ErnZanotti2020} for simplicial triangulations with the following properties.
	\begin{lemma}[right-inverse]\label{lem:conforming-companion}
		There exists a linear operator $\Jcal_\ell : V(\Tcal_\ell) \to V$ with \eqref{def:conforming-companion} such that any $v_\ell = (v_{\Tcal_\ell},v_{\Fcal_\ell}) \in V(\Tcal_\ell)$ satisfies, for all $T \in \Tcal_\ell$,
		\begin{align}
			&\|\GrRec_\ell v_\ell - \D \Jcal_\ell v_\ell\|_{L^p(T)}^p \lesssim \sum_{E \in \Fcal_\ell(\Omega), E \cap T \neq \emptyset} h_E^{1-p}\|[\PotRec_\ell v_\ell]_E\|_{L^p(E)}^p	\label{ineq:conforming-companion-estimate}\\
			&\qquad + \sum_{F \in \Fcal_\ell(T)} h_F^{1-p}\|\Pi_F^k((\PotRec_\ell v_\ell)_T - v_F)\|_{L^p(F)}^p + h_T^{-p}\|\Pi_{T}^k(\PotRec_\ell v_\ell - v_{T})\|_{L^p(T)}^p.\nonumber
		\end{align}
		In particular, $\Jcal_\ell$ is stable in the sense that $\|\D \Jcal_\ell v_\ell\|_{L^p(\Omega)} \leq \Lambda_0 \|v_\ell\|_\ell$ holds with the constant $\Lambda_0$ that exclusively depends on $k$, $p$, and the shape regularity of $\Tcal_\ell$.
	\end{lemma}
	\begin{proof}
		For $p = 2$, the right-hand side of \eqref{ineq:conforming-companion-estimate} is an upper bound for $\|\D (\PotRec_\ell v_\ell - \Jcal_\ell v_\ell)\|_{L^2(T)}^2$, cf.~\cite[Proof of Proposition 4.7]{ErnZanotti2020}, and scaling arguments confirm this for $1 < p < \infty$. The $L^p$ stability of the $L^2$ projection \cite[Lemma 3.2]{DiPietroDroniou2017} and the orthogonality $\GrRec_\ell v_\ell - \D \Jcal_\ell v_\ell \perp \RT_k(T;\M)$ in $L^2(T;\M)$ imply $\|\GrRec_\ell v_\ell - \D \Jcal_\ell v_\ell\|_{L^p(T)} \lesssim \|\D (\PotRec_\ell v_\ell - \Jcal_\ell v_\ell)\|_{L^p(T)}$. This proves \eqref{ineq:conforming-companion-estimate}.
		The right-hand side of \eqref{ineq:conforming-companion-estimate} can be bounded by
		\begin{align}
			& \sum_{E \in \Fcal_\ell(\Omega), E \cap T \neq \emptyset} h_E^{1-p}\|[\PotRec_\ell v_\ell]_E\|_{L^p(E)}^p + \sum_{F \in \Fcal_\ell(T)} h_F^{1-p}\|\Pi_F^k((\PotRec_\ell v_\ell)_T - v_F)\|_{L^p(F)}^p
			\label{ineq:conforming-companion-stability}\\
			&\qquad + h_T^{-p}\|\Pi_{T}^k(\PotRec_\ell v_\ell - v_{T})\|_{L^p(T)}^p \lesssim \sum_{K \in \Tcal_\ell, K \cap T \neq \emptyset} \sum_{E \in \Fcal_\ell(K)} h_E^{1-p}\|v_K - v_F\|_{L^p(E)}^p\nonumber
		\end{align}
		with a hidden constant that only depends on the shape regularity of $\Tcal_\ell$, $k$, and $p$ \cite[Proof of Proposition 4.7]{ErnZanotti2020}. The sum of this over all simplices $T \in \Tcal_\ell$, a triangle inequality, and the shape regularity of $\Tcal_\ell$ imply $\|\GrRec_\ell v_\ell - \D \Jcal_\ell v_\ell\|_{L^p(\Omega)} \lesssim \|v_\ell\|_\ell$. This, a reverse triangle inequality,
		and the norm equivalence from \Cref{lem:stability_gradient_reconstruction}.a conclude the stability $\|\D \Jcal_{\ell} v_\ell\|_{L^p(\Omega)} \lesssim \|v_\ell\|_\ell$.
	\end{proof}
	
	\section{Proof of \Cref{thm:plain-convergence}}\label{sec:proof-of-main-results}
	This section is devoted to the proof of the convergence results in \Cref{thm:plain-convergence}.
	
	\subsection{Discrete compactness}
	The proof of \Cref{thm:plain-convergence} departs from a discrete compactness in spirit of \cite{BuffaOrtner2009,DiPietroErn2010,DiPietroDroniou2017} and generalizes \cite{DiPietroDroniou2017}.
	Recall the mesh-size function $h_\ell \in P_0(\Tcal_\ell)$ from \Cref{sec:discrete-spaces} with $h_\ell|_T \coloneqq |T|^{1/n}$ for $T \in \Tcal_\ell$ and the seminorm $\|\bullet\|_\ell$ in $V(\Tcal_\ell)$ from \eqref{def:discrete-norm}.
	\begin{theorem}[discrete compactness]
		\label{thm:conv_analysis:discrete_compactness}
		Given a uniformly shape-regular sequence $(\Tcal_\ell)_{\ell \in \N_0}$ of triangulations and $(v_\ell)_{\ell \in \N_0}$ with $v_\ell = (v_{\Tcal_\ell},v_{\Fcal_\ell}) \in \Acal(\Tcal_\ell)$ for all $\ell \in \N_0$. Suppose that the sequence $(\|v_\ell\|_\ell)_{\ell \in \N_0}$ is bounded and suppose that $\lim_{\ell \to \infty} \mu_\ell(v_\ell) = 0$ with
		\begin{align}
			\mu_\ell(v_\ell) \coloneqq \|h_\ell^{k+1}(\GrRec_\ell v_\ell - \D \Jcal_\ell v_\ell)\|_{L^p(\Omega)}^p + \sum_{F \in \Fcal_\ell(\Gamma_{\mathrm{D}})} h_F^{kp+1}\|\Jcal_\ell v_\ell - u_\mathrm{D}\|_{L^p(F)}^p.
			\label{eq:conv_analysis:discrete_compactness_assumption}
		\end{align}
		Then there exist a subsequence $(v_{\ell_j})_{j \in \N_0}$ of $(v_{\ell})_{\ell \in \N_0}$ and a weak limit $v \in \Acal$ such that $\Jcal_{\ell_j} v_{\ell_j} \rightharpoonup v$ weakly in $V$ and $\GrRec_{\ell_j} v_{\ell_j} \rightharpoonup \D v$ weakly in $L^p(\Omega;\M)$ as $j \to \infty$.
	\end{theorem}
	\begin{proof}
		The first part of the proof proves the uniform boundedness
		\begin{align}
			\|\Jcal_\ell v_\ell\|_{W^{1,p}(\Omega)} \lesssim \|v_\ell\|_\ell + \|u_\mathrm{D}\|_{W^{1,p}(\Omega)} \lesssim 1
			\label{ineq:conv-analysis:boundedness-J-vh}
		\end{align}
		of the sequence $\Jcal_\ell v_\ell$ in $W^{1,p}(\Omega;\R^m)$.
		Since $\|\D \Jcal_\ell v_\ell\|_{L^p(\Omega)}\lesssim \|v\|_\ell$ from the stability of $\Jcal_\ell$ in \Cref{lem:conforming-companion}, it remains to show $\|\Jcal_\ell v_\ell\|_{L^p(\Omega)} \lesssim \|v_\ell\|_\ell + \|u_\mathrm{D}\|_{W^{1,p}(\Omega)}$ to obtain \eqref{ineq:conv-analysis:boundedness-J-vh}. The triangle inequality implies
		\begin{align}
			\|\Jcal_\ell v_\ell\|_{L^p(\Omega)} \leq \|\Jcal_\ell v_\ell - v_{\Tcal_\ell}\|_{L^p(\Omega)} + \|v_{\Tcal_\ell} - \Pi_{\Tcal_\ell}^k u_\mathrm{D}\|_{L^p(\Omega)} + \|\Pi_{\Tcal_\ell}^k u_\mathrm{D}\|_{L^p(\Omega)}.
			\label{ineq:conv-analysis:triangle-inequality-split}
		\end{align}
		The right-inverse $\Jcal_\ell$ of the interpolation $\Ical_\ell$ from \Cref{lem:conforming-companion} satisfies the $L^2$ orthogonality $\Jcal_\ell v_\ell - v_{\Tcal_\ell} \perp P_k(\Tcal_\ell;\R^m)$. This, a piecewise application of the Poincar\'e inequality, a triangle inequality, and $\|h_\ell\|_{L^\infty(\Omega)} \leq \mathrm{diam}(\Omega)$ lead to
		\begin{align}
			\|\Jcal_\ell v_\ell - v_{\Tcal_\ell}\|_{L^p(\Omega)} &\lesssim \|h_\ell \D_\pw(\Jcal_\ell v_\ell - v_{\Tcal_\ell})\|_{L^p(\Omega)}\nonumber\\
			& \lesssim \|\D \Jcal_\ell v_\ell\|_{L^p(\Omega)} + \|\D_\pw v_{\Tcal_\ell}\|_{L^p(\Omega)}.
			\label{ineq:conv-analysis:J-vh-vT-split}
		\end{align}
		Since $v_\ell - \Ical_\ell u_\mathrm{D} \in V_\D(\Tcal_\ell)$,
		the Sobolev embedding from \Cref{lem:stability_gradient_reconstruction}.c and a triangle inequality show $\|v_{\Tcal_\ell} - \Pi_{\Tcal_\ell}^k u_\mathrm{D}\|_{L^p(\Omega)} \lesssim \|\GrRec_\ell(v_\ell - \Ical_\ell u_\mathrm{D})\|_{L^p(\Omega)} \leq \|\GrRec_\ell v_\ell\|_{L^p(\Omega)} + \|\GrRec_\ell \Ical_\ell u_\mathrm{D}\|_{L^p(\Omega)}$.
		This, the equivalence $\|\GrRec_\ell v_\ell\|_{L^p(\Omega)} \approx \|v_\ell\|_\ell$ from \Cref{lem:stability_gradient_reconstruction}.a, the commutativity $\GrRec_\ell \Ical_\ell u_\mathrm{D} = \Pi_{\Sigma(\Tcal_\ell)} \D u_\mathrm{D}$ from \Cref{lem:stability_gradient_reconstruction}.b, and the $L^p$ stability of the $L^2$ projection $\Pi_{\Sigma(\Tcal_\ell)}$ \cite[Lemma 3.2]{DiPietroDroniou2017} provide
		\begin{align}
			\|v_{\Tcal_\ell} - \Pi_{\Tcal_\ell}^k u_\mathrm{D}\|_{L^p(\Omega)} \lesssim \|v_\ell\|_\ell + \|\D u_\D\|_{L^p(\Omega)}.
			\label{ineq:conv-analysis:v_T-u_D}
		\end{align}
		\Cref{lem:conforming-companion} and the definition of the discrete norm $\|v_\ell\|_\ell$ in \eqref{def:discrete-norm} prove that the right-hand side of \eqref{ineq:conv-analysis:J-vh-vT-split} is controlled by $\|v_\ell\|_\ell$. Hence, the combination of \eqref{ineq:conv-analysis:triangle-inequality-split}--\eqref{ineq:conv-analysis:v_T-u_D} conclude \eqref{ineq:conv-analysis:boundedness-J-vh}.
		
		The Banach-Alaoglu theorem \cite[Theorem 3.18]{Brezis2011} ensures the existence of a (not relabelled) subsequence of $(\Jcal_\ell v_\ell)_{\ell \in \N_0}$ and a weak limit $v \in V$ such that $\Jcal_\ell v_\ell \rightharpoonup v$ weakly in $V$ as $\ell \to \infty$. \Cref{lem:stability_gradient_reconstruction}.a assures that
		the sequence $(\GrRec_\ell v_\ell)_{\ell \in \N_0}$ is uniformly bounded in $L^p(\Omega;\M)$.
		Hence there exist a (not relabelled) subsequence of $(v_\ell)_{\ell \in \N_0}$ and its weak limit $G \in L^p(\Omega;\M)$ such that $\GrRec_\ell v_\ell \rightharpoonup G$ weakly in $L^p(\Omega;\M)$ as $\ell \to \infty$.
		The second part of the proof verifies $\D v = G$ in $\Omega$ and $v = u_\mathrm{D}$ on $\Gamma_{\mathrm{D}}$ (and so $v \in \Acal$).
		Since $\Ical_\ell \Jcal_\ell v_\ell = v_\ell$, the commutativity $\GrRec_\ell v_\ell = \Pi_{\Sigma(\Tcal_\ell)} \D \Jcal_\ell v_\ell$ from \Cref{lem:stability_gradient_reconstruction}.b proves the $L^2$ orthogonality $\GrRec_\ell v_\ell - \D \Jcal_\ell v_\ell \perp \Sigma(\Tcal_\ell)$. This and an integration by parts verify, for all $\Phi \in C^{\infty}(\overline{\Omega};\M)$ with $\Phi \equiv 0$ on $\Gamma_{\mathrm{N}}$, that
		\begin{align}
			\int_\Omega \GrRec_\ell v_\ell : \Phi \d{x} &= \int_\Omega (\GrRec_\ell v_\ell - \D \Jcal_\ell v_\ell) : \Phi \d{x} + \int_\Omega \D \Jcal_\ell v_\ell : \Phi \d{x}\nonumber\\
			\begin{split}
				&= \int_\Omega (\GrRec_\ell v_\ell - \D \Jcal_\ell v_\ell) : (1 - \Pi_{\Sigma(\Tcal_\ell)}) \Phi \d{x} - \int_\Omega \Jcal_\ell v_\ell \cdot \div \Phi \d{x}\\
				&\qquad + \int_{\Gamma_{\mathrm{D}}} (\Jcal_\ell v_\ell - u_\mathrm{D}) \cdot \Phi\nu \d{s} + \int_{\Gamma_{\mathrm{D}}} u_\mathrm{D} \cdot \Phi\nu \d{s}.
			\end{split}
			\label{eq:conv-analysis-split}
		\end{align}
		The approximation property of piecewise polynomials, also known under the name Bramble-Hilbert lemma \cite[Lemma 4.3.8]{BrennerScott2008}, leads to
		\begin{align}
			\|h_\ell^{-(k+1)}(1 - \Pi_{\Sigma(\Tcal_\ell)}) \Phi\|_{L^{p'}(\Omega)} \lesssim |\Phi|_{W^{k+1,p'}(\Omega)}.
			\label{ineq:approximation-polynomials}
		\end{align}
		This and a H\"older inequality imply
		\begin{align}
			\begin{split}
				\int_\Omega (\GrRec_\ell v_\ell - \D \Jcal_\ell v_\ell) &: (1 - \Pi_{\Sigma(\Tcal_\ell)}) \Phi \d{x}\\
				&\lesssim \|h_\ell^{k+1}(\GrRec_\ell v_\ell - \D \Jcal_\ell v_\ell)\|_{L^p(\Omega)} |\Phi|_{W^{k+1,p'}(\Omega)}.
			\end{split}
			\label{ineq:approximation}
		\end{align}
		The $L^2$ orthogonality $(\Jcal_\ell u_\ell - u_\mathrm{D})|_F \perp P_k(F;\R^m)$ for each side $F \in \Fcal_\ell(\Gamma_{\mathrm{D}})$ on the Dirichlet boundary, a piecewise application of the trace inequality, and \eqref{ineq:approximation-polynomials} imply
		\begin{align}
			\int_{\Gamma_{\mathrm{D}}} (\Jcal_\ell v_\ell - u_\mathrm{D}) \cdot \Phi \nu \d{s} \lesssim \Big(\sum_{F \in \Fcal_\ell(\Gamma_{\mathrm{D}})} h_F^{kp + 1} \|\Jcal_\ell v_\ell - u_\mathrm{D}\|_{L^p(F)}^p\Big)^{1/p}|\Phi|_{W^{k+1,p'}(\Omega)}.
			\label{ineq:Dirichlet-boundary-approximation}
		\end{align}
		The right-hand sides of \eqref{ineq:approximation}--\eqref{ineq:Dirichlet-boundary-approximation} vanish in the limit as $\ell \to \infty$ by assumption \eqref{eq:conv_analysis:discrete_compactness_assumption}. This, \eqref{eq:conv-analysis-split}, $\GrRec_\ell v_\ell \rightharpoonup G$ in $L^p(\Omega;\M)$, and $\Jcal_\ell v_\ell \rightharpoonup v$ in $V$ prove
		\begin{align*}
			\int_\Omega (G : \Phi + v \cdot \div \Phi) \d{x} - \int_{\Gamma_{\mathrm{D}}} u_\mathrm{D} \cdot \Phi \nu \d{s} = 0
		\end{align*}
		for all $\Phi \in C^\infty(\overline{\Omega};\M)$ with $\Phi \equiv 0$ on $\Gamma_{\mathrm{N}}$. This implies $\D v = G$ a.e.~in $\Omega$ with $v = u_\mathrm{D}$ on $\Gamma_{\mathrm{D}}$ and concludes the proof.
	\end{proof}
	\noindent Since $\Jcal_\ell v_\ell$ cannot attain the exact value $u_\mathrm{D}$ on $\Gamma_{\mathrm{D}}$ in general, a (Dirichlet boundary data) oscillation arises in \eqref{eq:conv_analysis:discrete_compactness_assumption}, but
	is controlled by the contributions of $\eta_\ell^{(\varepsilon)}$.
	\begin{lemma}[Dirichlet boundary data oscillation]\label{lem:Dirichlet-data-oscillation}
		Given $F \in \Fcal_\ell(\Gamma_{\mathrm{D}})$, let $T \in \Tcal_\ell$ be the unique simplex with $F \in \Fcal_\ell(T) \cap \Fcal_\ell(\Gamma_{\mathrm{D}})$. Then it holds, for all $v_\ell = (v_{\Tcal_\ell},v_{\Fcal_\ell}) \in V(\Tcal_\ell)$, that
		\begin{align*}
			\|\Jcal_\ell v_\ell - u_\mathrm{D}\|_{L^p(F)}^p &\lesssim \sum_{E \in \Fcal_\ell(\Omega), E \cap F \neq \emptyset} \|[\PotRec_\ell v_\ell]_E\|_{L^p(E)}^p\\
			&\qquad + \|\Pi_F^k(\PotRec_\ell v_\ell - v_F)\|_{L^p(F)}^p + \|\PotRec_\ell v_\ell - u_\mathrm{D}\|_{L^p(F)}^p.
		\end{align*}
	\end{lemma}
	\begin{proof}
		 The proof of \Cref{lem:Dirichlet-data-oscillation} utilizes standard averaging and bubble-function techniques, cf., e.g, \cite{CarstensenEigelHoppeLoebhard2012,Verfuerth2013,VeeserZanotti2019,ErnZanotti2020}; further details are therefore omitted.
	\end{proof}
	\subsection{Plain convergence}\label{sec:plain-convergence}
	Before the remaining parts of this subsection prove \Cref{thm:plain-convergence}, the following lemma establishes the reduction of the mesh-size function $h_\ell$ with $h_\ell|_T \equiv |T|^{1/n}$ for all $T \in \Tcal_\ell$.
	\begin{lemma}[mesh-size reduction]\label{lem:mesh-size-reduction}
		Given the output $(\Tcal_\ell)_{\ell \in \N_0}$ of AHHO from \Cref{sec:adaptive_algorithm}, let $\Omega_\ell \coloneqq \mathrm{int}(\cup (\Tcal_\ell \setminus \Tcal_{\ell+1}))$ for all level $\ell \in \N_0$. Then it holds
		\begin{align}
			\lim_{\ell \to \infty} \|h_{\ell}\|_{L^\infty(\Omega_{\ell})} = 0.\label{eq:convergence-mesh-size}
		\end{align}
	\end{lemma}
	\begin{proof}
		The proof is omitted for two reasons. First this is known from \cite[Lemma 9]{NochettoVeeser2012} and second it is a particular case of \Cref{rem:mesh-size-reduction} below.
	\end{proof}
	\begin{proof}[Proof of \Cref{thm:plain-convergence}]
		This proof is motivated by \cite{C2008,BartelsC2008,MorinSiebertVeeser2008,OrtnerPraetorius2011,NochettoVeeser2012,CDolzmann2015} and is divided into five steps.
		\begin{enumerate}[wide]
			\item[\emph{Step 1 establishes $\lim_{\ell \to \infty} \eta_\ell^{(\varepsilon)}(\Tcal_\ell\setminus\Tcal_{\ell+1}) = 0$.}]
			Since no suitable residual-based a posteriori control is available in the general setting \ref{assumption-1}--\ref{assumption-2}, standard arguments, e.g., reliability, efficiency, or estimator reduction \cite{MorinSiebertVeeser2008,C2008,DieningKreuzer2008} fail. The proof of Step 1 rather relies on a positive power of the mesh-size that arises from the smoothness of test functions in \Cref{thm:conv_analysis:discrete_compactness}. This is done in \cite{OrtnerPraetorius2011} for a similar setting and in \cite{BuffaOrtner2009,DiPietroErn2010,DiPietroDroniou2017} for uniform mesh-refinements.
			Let $\mu_\ell^{(\varepsilon)}(T)$ abbreviate some contributions of $\eta_\ell^{(\varepsilon)}(T)$ from \eqref{def:eta} related to $\mu_\ell(u_\ell)$ in \eqref{eq:conv_analysis:discrete_compactness_assumption}, namely
			\begin{align}
				&\mu_\ell^{(\varepsilon)}(T) \coloneqq |T|^{(\varepsilon p - p)/n}\|\Pi_T^k(\PotRec_\ell u_\ell - u_T)\|_{L^p(T)}^p\nonumber\\
				&\quad + |T|^{(\varepsilon p + 1 - p)/n}\Big(\sum_{F \in \Fcal_\ell(T) \cap \Fcal_\ell(\Gamma_{\mathrm{D}})} \|\PotRec_\ell u_\ell - u_\mathrm{D}\|_{L^p(F)}^p\label{def:mu}\\
				&\quad + \sum_{F \in \Fcal_\ell(T) \cap \Fcal_\ell(\Omega)} \|[\PotRec_\ell u_\ell]_F\|_{L^p(F)}^p + \sum_{F \in \Fcal_\ell(T)} \|\Pi_F^k ((\PotRec_\ell u_\ell)|_T - u_F)\|_{L^p(F)}^p\Big).\nonumber
			\end{align}
			Denote $\mu_\ell^{(\varepsilon)}(\Tcal_\ell \setminus \Tcal_{\ell+1}) \coloneqq \sum_{T \in \Tcal_\ell \setminus \Tcal_{\ell+1}} \mu_\ell^{(\varepsilon)}(T)$. Given any $T \in \Tcal_\ell$ and $F \in \Fcal_\ell(\Gamma_{\mathrm{N}})$, the $L^{p'}$ stability of the $L^2$ projection $\Pi_{\RT_k(T;\M)}$ resp.~$\Pi_T^k$ or $\Pi_F^k$ \cite[Lemma 3.2]{DiPietroDroniou2017} implies $\|\sigma_\ell - \D W(\GrRec_\ell u_\ell)\|_{L^{p'}(T)} \lesssim \|\D W(\GrRec_\ell u_\ell)\|_{L^{p'}(T)}$ resp.~$\|(1 - \Pi_T^k) f\|_{L^{p'}(T)} \lesssim \|f\|_{L^{p'}(T)}$ or $\|(1 - \Pi_F^k) g\|_{L^{p'}(F)} \lesssim \|g\|_{L^{p'}(F)}$.
			Since $\mu_\ell^{(\varepsilon)}(\Tcal_\ell \setminus \Tcal_{\ell+1}) \leq \|h_\ell\|_{L^{\infty}(\Omega_\ell)}^{\varepsilon p}\mu_\ell^{(0)}(\Tcal_\ell\setminus\Tcal_{\ell+1})$ with $\|h_\ell\|_{L^{\infty}(\Omega_\ell)} = \sup_{T \in \Tcal_\ell\setminus\Tcal_{\ell+1}} |T|^{1/n}$ and $\Omega_\ell = \mathrm{int}(\cup (\Tcal_\ell\setminus \Tcal_{\ell+1}))$ from \Cref{lem:mesh-size-reduction}, this leads to
			\begin{align}
				\eta_\ell^{(\varepsilon)}(\Tcal_\ell\setminus\Tcal_{\ell+1}) &\lesssim \|h_\ell\|_{L^\infty(\Omega_\ell)}^{\varepsilon p} \mu_\ell^{(0)}(\Tcal_\ell\setminus \Tcal_{\ell+1}) + \|h_\ell\|_{L^\infty(\Omega_\ell)}^{\varepsilon p'}\|\D W(\GrRec u_\ell)\|_{L^{p'}(\Omega)}^{p'}\nonumber\\
				&\qquad + \|h_\ell\|_{L^\infty(\Omega_\ell)}^{p'}\|f\|_{L^{p'}(\Omega)}^{p'} + \|h_\ell\|_{L^\infty(\Omega_\ell)}\|g\|_{L^{p'}(\Gamma_{\mathrm{N}})}^{p'}.
				\label{ineq:eta-mu-split}
			\end{align}
			The two-sided growth $|A|^p - 1 \lesssim W(A) \lesssim |A|^p + 1$ implies $|\D W(A)|^{p'} \lesssim |A|^p + 1$ \cite[Lemma 2.1]{CarstensenTran2020} and so \Cref{lem:existence-discrete-solutions} provides $\|\D W(\GrRec_\ell u_\ell)\|_{L^{p'}(\Omega)}^{p'} \lesssim \|\GrRec_\ell u_\ell\|_{L^p(\Omega)}^p + |\Omega| \lesssim 1$. This and \eqref{eq:convergence-mesh-size} prove
			\begin{align}
				\begin{split}
					\lim_{\ell \to \infty} \Big(&\|h_\ell\|_{L^\infty(\Omega_\ell)}^{\varepsilon p'}\|\D W(\GrRec u_\ell)\|_{L^{p'}(\Omega)}^{p'}\\
					& + \|h_\ell\|_{L^\infty(\Omega_\ell)}^{p'}\|f\|_{L^{p'}(\Omega)}^{p'} + \|h_\ell\|_{L^\infty(\Omega_\ell)}\|g\|_{L^{p'}(\Gamma_{\mathrm{N}})}^{p'}\Big) = 0,
				\end{split}
				\label{eq:lim-volume}
			\end{align}
			whence, in order to obtain $\lim_{\ell \to \infty} \eta_\ell^{(\varepsilon)}(\Tcal_\ell\setminus\Tcal_{\ell+1}) = 0$, it suffices to prove that $\mu_\ell^{(0)}(\Tcal_\ell\setminus \Tcal_{\ell+1})$ is uniformly bounded.
			The estimate \eqref{ineq:conforming-companion-stability} provides control over all but only one contribution of $\mu_\ell^{(0)}(T)$ in \eqref{def:mu}; that is $\|h_F^{-1/p'}(\PotRec_\ell u_\ell - u_\mathrm{D})\|_{L^p(F)}$ for any $F \in \Fcal_\ell(\Gamma_{\mathrm{D}})$.
			Triangle inequalities and $u_F = \Pi_F^k u_\mathrm{D}$ imply
			\begin{align}
				\begin{split}
					&\|h_F^{-1/p'}(\PotRec_\ell u_\ell - u_\mathrm{D})\|_{L^p(F)} \leq \|h_F^{-1/p'}(1 - \Pi_{F}^k)\PotRec_\ell u_\ell\|_{L^p(F)}\\
					&\quad + \|h_F^{-1/p'}\Pi_F^k(\PotRec_\ell u_\ell - u_F)\|_{L^p(F)} + \|h_F^{-1/p'}(1 - \Pi_{F}^k)u_\mathrm{D}\|_{L^p(F)}.
				\end{split}
				\label{ineq:trace-triangle-inequality}
			\end{align}
			Given any $F \in \Fcal_\ell(\Gamma_{\mathrm{D}})$, let $T \in \Tcal_\ell$ be the unique simplex with $F \in \Fcal_\ell(T) \cap \Fcal_\ell(\Gamma_{\mathrm{D}})$.
			The $L^p$ stability of the $L^2$ projection $\Pi_F^k$ \cite[Lemma 3.2]{DiPietroDroniou2017} and a trace inequality show $\|h_F^{-1/p'}(1 - \Pi_F^k) u_\mathrm{D}\|_{L^p(F)} \lesssim \|\D u_\mathrm{D}\|_{L^p(T)}$ and $\|h_F^{-1/p'}(1 - \Pi_F^k) (\PotRec_\ell u_\ell)|_F\|_{L^p(F)} \lesssim \|\D \PotRec_\ell u_\ell\|_{L^p(T)}$.
			Recall that $\D \PotRec_\ell u_\ell$ is the $L^2$ projection of $\GrRec_\ell u_\ell$ onto $\D_\pw P_{k+1}(\Tcal)$, whence the $L^p$ stability of $L^2$ projections \cite[Lemma 3.2]{DiPietroDroniou2017} proves $\|\D \PotRec_\ell u_\ell\|_{L^p(T)} \lesssim \|\GrRec_\ell u_\ell\|_{L^p(T)}$.
			Hence, the right-hand side of \eqref{ineq:trace-triangle-inequality} is controlled by $\|\GrRec_\ell u_\ell\|_{L^p(T)} + \|\D u_\D\|_{L^p(T)} + \|h_F^{-1/p'}\Pi_F^k(\PotRec_\ell u_\ell - u_F)\|_{L^p(F)}$.
			This, \eqref{ineq:conforming-companion-stability}, and $\|u_\ell\|_\ell \approx \|\GrRec_\ell u_\ell\|_{L^p(\Omega)} \leq \cnstL{cnst:G-u-h}$ from \Cref{lem:stability_gradient_reconstruction}.a and  \Cref{lem:existence-discrete-solutions} lead to
			\begin{align*}
				\mu_\ell^{(0)}(\Tcal_\ell\setminus\Tcal_{\ell+1}) \leq \mu_\ell^{(0)} \coloneqq \sum_{T \in \Tcal_\ell} \mu_\ell^{(0)}(T) \lesssim \|u_\ell\|_\ell^p + \|\D u_\mathrm{D}\|_{L^p(\Omega)}^p \lesssim 1.
			\end{align*}
			Hence, the combination of \eqref{ineq:eta-mu-split}--\eqref{eq:lim-volume} with $\lim_{\ell \to \infty} \|h_\ell\|_{L^\infty(\Omega_\ell)} = 0$ in \eqref{eq:convergence-mesh-size} confirms $\lim_{\ell \to \infty} \eta_\ell^{(\varepsilon)}(\Tcal_\ell\setminus \Tcal_{\ell+1}) = 0$.
			
			\item[\emph{Step 2 establishes $\lim_{\ell \to \infty} \eta_\ell^{(\varepsilon)} = 0$.}] Recall the set $\mathfrak{M}_\ell$ of marked simplices on level $\ell \in \N_0$ from \Cref{sec:adaptive_algorithm}.
			Since all simplices in $\mathfrak{M}_\ell \subset \Tcal_\ell \setminus \Tcal_{\ell+1}$ are refined and the D\"orfler marking enforces $\eta_\ell^{(\varepsilon)} \leq \theta^{-1}\eta_\ell^{(\varepsilon)}(\mathfrak{M}_\ell) \leq \theta^{-1}\eta_\ell^{(\varepsilon)}(\Tcal_\ell \setminus \Tcal_{\ell+1})$ in \eqref{ineq:Doerfler-marking}, the convergence $\lim_{\ell \to \infty} \eta_\ell^{(\varepsilon)}(\Tcal_\ell\setminus\Tcal_{\ell+1}) = 0$ in Step 1 implies $\lim_{\ell \to \infty} \eta_\ell^{(\varepsilon)} = 0$.
			
			\item[\emph{Step 3 provides the lower energy bound (LEB)}]
			\begin{align}
				\begin{split}
					\mathrm{LEB}_\ell \coloneqq E_\ell(u_\ell) &+ \int_\Omega (1 - \Pi_{\Sigma(\Tcal_\ell)}) \D W(\GrRec_\ell u_\ell) : \D u \d{x}\\
					& - \cnstL{cnst:oscillation}\big(\osc(f,\Tcal_\ell) + \osc_{\mathrm{N}}(g,\Fcal_{\ell}(\Gamma_{\mathrm{N}}))\big) \leq E(u).
				\end{split}
				\label{ineq:LEB-II}
			\end{align}
			The convexity of $W \in C^1(\M)$ implies $\D W(\GrRec_\ell u_\ell):(\D u - \GrRec_\ell u_\ell) \leq W(\D u) - W(\GrRec_\ell u_\ell)$ a.e.~in $\Omega$. The integral of this inequality with $\sigma_\ell \coloneqq \Pi_{\Sigma(\Tcal_\ell)} \D W(\GrRec_\ell u_\ell)$ reads
			\begin{align}
				0 \leq \int_\Omega \big(W(\D u) - W(\GrRec_\ell u_\ell) &- (1 - \Pi_{\Sigma(\Tcal_\ell)}) \D W(\GrRec_\ell u_\ell) : \D u\big) \d{x}\nonumber\\
				& - \int_\Omega \sigma_\ell : (\D u - \GrRec_\ell u_\ell) \d{x}.
				\label{ineq:convexity}
			\end{align}
			The commutativity $\Pi_{\Sigma(\Tcal_\ell)} \D u = \GrRec_\ell \Ical_\ell u$ from \Cref{lem:stability_gradient_reconstruction}.b and the discrete Euler-Lagrange equations \eqref{eq:dELE} lead to
			\begin{align*}
				\int_\Omega \sigma_\ell : (\D u - \GrRec_\ell u_\ell) \d{x} = \int_\Omega f \cdot (\Pi_{\Tcal_\ell}^k u - u_{\Tcal_\ell}) \d{x} + \int_{\Gamma_\mathrm{N}} g \cdot (\Pi_{\Fcal_\ell}^k u - u_{\Fcal_\ell}) \d{x}.
			\end{align*}
			The substitution of this in \eqref{ineq:convexity}, the definition of $E$ in \eqref{def:energy}, and the definition of $E_h$ in \eqref{def:discrete_energy} result in
			\begin{align}
				\begin{split}
					0 &\leq E(u) - E_\ell(u_\ell) - \int_\Omega (1 - \Pi_{\Sigma(\Tcal_\ell)}) \D W(\GrRec_\ell u_\ell) : \D u \d{x}\\
					&\quad + \int_\Omega (u - \Pi_{\Tcal_\ell}^k u) \cdot (f - \Pi_{\Tcal_\ell}^k f) \d{s} + \int_{\Gamma_{\mathrm{N}}} (u - \Pi_{\Fcal_\ell}^k u) \cdot (g - \Pi_{\Fcal_\ell}^k g) \d{s}.
				\end{split}\label{ineq:LEB}
			\end{align}
			The final two integrals on the right-hand side of \eqref{ineq:LEB}
			give rise to the data oscillations $\osc(f,\Tcal_\ell)$ and $\osc_{\mathrm{N}}(g,\Fcal_\ell(\Gamma_{\mathrm{N}}))$ defined in \Cref{sec:discrete-spaces}. In fact, a H\"older inequality and a piecewise application of the Poincar\'e inequality show
			\begin{align*}
				\int_\Omega (u - \Pi_{\Tcal_\ell}^k u) \cdot (f - \Pi_{\Tcal_\ell}^k f) \d{x} \lesssim \|\D u\|_{L^p(\Omega)} \osc(f,\Tcal_\ell).
			\end{align*}
			A trace inequality and the Bramble-Hilbert lemma \cite[Lemma 4.3.8]{BrennerScott2008} leads, for all $F \in \Fcal_\ell(\Gamma_{\mathrm{N}})$ and the unique $T \in \Tcal_\ell$ with $F \in \Fcal_\ell(T) \cap \Fcal_\ell(\Gamma_{\mathrm{N}})$, to $\|h_F^{-1/p'}(u - \Pi_F^k u)\|_{L^{p}(F)} \lesssim \|\D u\|_{L^p(T)}$. Consequently,
			\begin{align*}
				\int_{\Gamma_{\mathrm{N}}} (u - \Pi_{\Fcal_\ell}^k u) \cdot (g - \Pi_{\Fcal_\ell}^k g) \d{s} \lesssim \|\D u\|_{L^p(\Omega)} \osc_{\mathrm{N}}(g,\Fcal_\ell(\Gamma_{\mathrm{N}})).
			\end{align*}
			The lower $p$-growth $\cnstS{cnst:growth-left-1}|A|^p - \cnstS{cnst:growth-left-2} \leq W(A)$ for all $A \in \M$ implies the
			coercivity of $E$ in the seminorm $\|\D \bullet\|_{L^p(\Omega)}$ and so the bound $\|\D u\|_{L^p(\Omega)} \leq \newcnstL\label{cnst:Du}$ with a positive constant $\cnstL{cnst:Du}$, that exclusively depends on \cnstS{cnst:growth-left-1}, \cnstS{cnst:growth-left-2}, $\Omega$, $\Gamma_{\mathrm{D}}$, $f$, $g$, and $u_\mathrm{D}$, cf., e.g, \cite[Theorem 4.1]{Dacorogna2008}.
			Thus there exists a positive constant $\newcnstL\label{cnst:oscillation}$ independent of the mesh-size with
			\begin{align}
				\begin{split}
					\int_\Omega (u - \Pi_{\Tcal_\ell}^k u) \cdot (f - \Pi_{\Tcal_\ell}^k f) \d{s} &+ \int_{\Gamma_{\mathrm{N}}} (u - \Pi_{\Fcal_\ell}^k u) \cdot (g - \Pi_{\Fcal_\ell}^k g) \d{s}\\
					&\leq \cnstL{cnst:oscillation}\big(\osc(f,\Tcal_\ell) + \osc_{\mathrm{N}}(g,\Fcal_\ell(\Gamma_{\mathrm{N}}))\big).
				\end{split}
				\label{ineq:oscillation}
			\end{align}
			The combination of this with \eqref{ineq:LEB} concludes the proof of \eqref{ineq:LEB-II}.
			
			\item[\emph{Step 4 establishes $\lim_{\ell \to \infty} E_\ell(u_\ell) = E(u)$}.] On the one hand, the discrete compactness from \Cref{thm:conv_analysis:discrete_compactness} and the weak lower semicontinuity of the energy functional imply $E(u) \leq \liminf_{\ell \to \infty} \mathrm{LEB}_\ell$. On the other hand, $\mathrm{LEB}_\ell \leq E(u)$ from \eqref{ineq:LEB-II}.
			This proves $\lim_{\ell \to \infty} E_\ell(u_\ell) = E(u)$ as follows.
			Given any $\Phi \in C^\infty(\Omega;\M)$, the definition $\sigma_\ell \coloneqq \Pi_{\Sigma(\Tcal_\ell)} \D W(\GrRec_\ell u_\ell)$, a H\"older inequality, and \eqref{ineq:approximation-polynomials} show
			\begin{align}
				\Big|\!\int_{\Omega} (\sigma_\ell - \D W(\GrRec_\ell u_\ell)):\Phi \d{x}\Big| &= \Big|\!\int_{\Omega} (\sigma_\ell - \D W(\GrRec_\ell u_\ell)):(1 - \Pi_{\Sigma(\Tcal_\ell)})\Phi \d{x}\Big|
				\label{ineq:conv_analysis:proof_convergence_stress_approximation}\\
				&\lesssim \|h_\ell^{k+1}(\sigma_\ell - \D W(\GrRec_\ell u_\ell))\|_{L^{p'}(\Omega)}|\Phi|_{W^{k+1,p}(\Omega)}.\nonumber
			\end{align}
			Since $\|h_\ell^{k+1}(\sigma_\ell - \D W(\GrRec_\ell u_\ell))\|_{L^{p'}(\Omega)}^{p'} \leq \eta_\ell^{(k+1)} \lesssim \eta_\ell^{(\varepsilon)} \to 0$ as $\ell \to \infty$ from Step 2, the right-hand side of \eqref{ineq:conv_analysis:proof_convergence_stress_approximation} vanishes in the limit as $\ell \to \infty$. This, the density of $C^\infty(\Omega;\M)$ in $L^p(\Omega;\M)$, and the uniform boundedness of the sequence $(\sigma_\ell - \D W(\GrRec_\ell u_\ell))_{\ell \in \N_0}$ in $L^{p'}(\Omega;\M)$ prove $\sigma_\ell - \D W(\GrRec_\ell u_\ell) \rightharpoonup 0$ (weakly) in $L^{p'}(\Omega;\M)$ as $\ell \to \infty$. In particular,
			\begin{align}
				\lim_{\ell \to \infty} \int_{\Omega} (\sigma_\ell - \D W(\GrRec_\ell u_\ell)):\D u \d{x} = 0.
				\label{eq:conv_analysis:proof_convergence_limit_integral_LEB}
			\end{align}
			Recall $\mu_\ell(u_\ell)$ from \eqref{eq:conv_analysis:discrete_compactness_assumption} and $\mu_\ell^{(\varepsilon)}$ from \eqref{def:mu}. The combination of \eqref{ineq:conforming-companion-estimate} with the bound of the Dirichlet data oscillation from \Cref{lem:Dirichlet-data-oscillation} and the equivalence $h_F \approx h_T \approx |T|^{1/n}$ for all $T \in \Tcal_\ell$, $F \in \Fcal_\ell(T)$ from the shape regularity of $\Tcal_\ell$ result in 
			\begin{align*}
				\mu_\ell(u_\ell) \lesssim \mu_\ell^{(k+1)} \leq \operatorname{diam}(\Omega)^{(k+1-\varepsilon)p}\mu_\ell^{(\varepsilon)} \lesssim \eta_\ell^{(\varepsilon)}.
			\end{align*}
			This and $\lim_{\ell \to \infty} \eta_\ell^{(\varepsilon)} = 0$ from Step 2 imply $\lim_{\ell \to \infty} \mu_\ell(u_\ell) = 0$.
			Since $\|u_\ell\|_\ell \approx \|\GrRec_\ell u_\ell\|_{L^p(\Omega)} \leq \cnstL{cnst:G-u-h}$ from \Cref{lem:stability_gradient_reconstruction}.a and \Cref{lem:existence-discrete-solutions}, the discrete compactness from \Cref{thm:conv_analysis:discrete_compactness} leads to a (not relabelled) subsequence of $(u_\ell)_{\ell \in \N_0}$ and a weak limit $v \in \Acal$ such that $\Jcal_\ell u_\ell \rightharpoonup v$ weakly in $V$ and $\GrRec_\ell u_\ell \rightharpoonup \D v$ weakly in $L^p(\Omega;\M)$ as $\ell \to \infty$. The boundedness of the linear trace operator $\gamma: V \to L^p(\partial \Omega;\R^m)$ \cite[Chapter 9]{Brezis2011} implies $(\Jcal_\ell u_\ell)|_{\partial \Omega} \rightharpoonup v|_{\partial \Omega}$ (weakly) in $L^p(\partial \Omega;\R^m)$. Hence
			\begin{align*}
				\lim_{\ell \to \infty} \int_{\Gamma_{\mathrm{N}}} g \cdot \Jcal_\ell u_\ell \d{s} = \int_{\Gamma_\mathrm{N}} g \cdot v \d{s}.
			\end{align*}
			This, $\Jcal_{\ell} u_\ell \rightharpoonup v$ (weakly) in $V$, $\GrRec_\ell u_\ell \rightharpoonup \D v$ (weakly) in $L^p(\Omega;\M)$, the sequential weak lower semicontinuity of the functional $\int_\Omega W(\bullet) \d{x}$ in $L^p(\Omega;\M)$, and \eqref{def:conforming-companion} verify
			\begin{align}
				E(v) &\leq \liminf_{\ell \to \infty} \Big(\int_\Omega (W(\GrRec_\ell u_\ell) - f \cdot \Jcal_\ell u_\ell) \d{x} - \int_{\Gamma_{\mathrm{N}}} g \cdot \Jcal_\ell u_\ell \d{s}\Big)\label{ineq:conv-analysis-liminf}\\
				& = \liminf_{\ell \to \infty} \Big(E_\ell(u_\ell) - \int_\Omega \Jcal_\ell u_\ell \cdot (1 - \Pi_{\Tcal_\ell}^k) f \d{x} - \int_{\Gamma_{\mathrm{N}}} \Jcal_\ell u_\ell \cdot (1 - \Pi_{\Fcal_\ell}^k) g  \d{s}\Big)\nonumber.
			\end{align}
			As in \eqref{ineq:oscillation}, a piecewise application of the Poincar\'e inequality, the trace inequality, the approximation property of polynomials, and the uniform bound $\|\D \Jcal_\ell u_\ell\|_{L^p(\Omega)} \lesssim 1$ from \eqref{ineq:conv-analysis:boundedness-J-vh} confirm
			\begin{align}
				\begin{split}
					\Big|\!\int_\Omega \Jcal_\ell u_\ell \cdot (1 - \Pi_{\Tcal_\ell}^k) f \d{x}\Big| &+ \Big|\!\int_{\Gamma_{\mathrm{N}}} \Jcal_\ell u_\ell \cdot (1 - \Pi_{\Fcal_\ell}^k) g \d{s}\Big|\\
					&\quad \lesssim \osc(f,\Tcal_\ell) + 	\osc_{\mathrm{N}}(g,\Fcal_\ell(\Gamma_{\mathrm{N}})).
				\end{split}
				\label{ineq:oscillation-discrete-level}
			\end{align}
			Since $\osc(f,\Tcal_\ell)^{p'} + \osc_{\mathrm{N}}(g,\Fcal_\ell(\Gamma_{\mathrm{N}}))^{p'} \lesssim \eta_\ell^{(\varepsilon)}$ and $\lim_{\ell \to \infty} \eta_\ell^{(\varepsilon)} = 0$ from Step 2, the LEB from \eqref{ineq:LEB-II} and \eqref{eq:conv_analysis:proof_convergence_limit_integral_LEB}--\eqref{ineq:oscillation-discrete-level} lead to
			\begin{align*}
				E(u) \leq E(v) \leq \liminf_{\ell \to \infty} E_\ell(u_\ell) =  \liminf_{\ell \to \infty} \mathrm{LEB}_\ell \leq E(u).
			\end{align*}
			Hence $\lim_{\ell \to \infty} E_\ell(u_\ell) = \lim_{\ell \to \infty} \mathrm{LEB}_\ell = E(u)$ for a (not relabelled) subsequence of $(u_\ell)_{\ell \in \N_0}$. Since the above arguments from Step 4 apply to all subsequences of $(u_\ell)_{\ell \in \N_0}$ and the limit $E(u)$ is unique, this holds for the entire sequence.
			
			\item[\emph{Step 5 is the finish of the proof.}] Suppose that $W$ satisfies \eqref{ineq:cc-primal}. Then the arguments from \cite{CPlechac1997,CarstensenTran2020} show, for all $\varrho, \xi \in L^p(\Omega;\M)$ and $r,t$ from \Cref{tab:parameters}, that
			\begin{align}
				\begin{split}
					\|\varrho - \xi\|_{L^p(\Omega)}^r &\leq 3\cnstS{cnst:cc-primal}(|\Omega| + \|\varrho\|^{p}_{L^p(\Omega)} + \|\xi\|_{L^p(\Omega)}^p)^{t/t'}\\
					&\qquad \times \int_\Omega (W(\varrho) - W(\xi) - \D W(\xi) : (\varrho - \xi)) \d{x}.
				\end{split}
				\label{ineq:cc-primal-integrated}
			\end{align}
			The choice $\varrho \coloneqq \D u$ and $\xi \coloneqq \GrRec_\ell u_\ell$ in \eqref{ineq:cc-primal-integrated} and the bounds $\|\D u\|_{L^p(\Omega)} \leq \cnstL{cnst:Du}$ and $\|\GrRec_\ell u_\ell\|_{L^p(\Omega)} \leq \cnstL{cnst:G-u-h}$ lead, with the constant $\newcnstL\label{cnst:a-posteriori-case-I} \coloneqq 3\cnstS{cnst:cc-primal}(|\Omega| + \cnstL{cnst:G-u-h}^p + \cnstL{cnst:Du}^p)^{t/t'}$, to
			\begin{align}
				\begin{split}
					&\cnstL{cnst:a-posteriori-case-I}^{-1} \|\D u - \GrRec_\ell u_\ell\|_{L^p(\Omega)}^r\\
					&\qquad \leq
					\int_\Omega (W(\D u) - W(\GrRec_\ell u_\ell) - \D W(\GrRec_\ell u_\ell) : (\D u - \GrRec_\ell u_\ell)) \d{x}.
				\end{split}
				\label{ineq:cc-primal-discrete-version}
			\end{align}
			The right-hand side of \eqref{ineq:cc-primal-discrete-version} coincides with the right-hand side of \eqref{ineq:convexity}. The latter is bounded by the right-hand side of \eqref{ineq:LEB-II} in Step 3. This implies
			\begin{align}
				\cnstL{cnst:a-posteriori-case-I}^{-1}\|\D u - \GrRec_\ell u_\ell\|_{L^p(\Omega)}^r \leq E(u) - \mathrm{LEB}_\ell.
				\label{ineq:a-posteriori-case-I}
			\end{align}
			Step 4 proves that $E(u) - \mathrm{LEB}_\ell$ vanishes in the limit as $\ell \to \infty$. Thus,
			\begin{align*}
				\lim_{\ell \to \infty} \GrRec_\ell u_\ell = \D u \text{ (strongly) in } L^p(\Omega;\M).
			\end{align*}
			If $W$ satisfies \eqref{ineq:cc-stress}, then \cite[Lemma 4.2]{CarstensenTran2020} implies, for all $\varrho, \xi \in L^p(\Omega;\M)$ and $\widetilde{r},\widetilde{t}$ from \Cref{tab:parameters}, that
			\begin{align}
				\begin{split}
					\|\D W(\varrho) - \D W(\xi)\|_{L^p(\Omega)}^{\widetilde{r}} &\leq 3\cnstS{cnst:cc-stress}(|\Omega| + \|\varrho\|^{p}_{L^p(\Omega)} + \|\xi\|_{L^p(\Omega)}^p)^{\widetilde{t}/\widetilde{t}'}\\
					&\qquad \times \int_\Omega (W(\varrho) - W(\xi) - \D W(\xi) : (\varrho - \xi)) \d{x},
				\end{split}
				\label{ineq:cc-dual-integrated}
			\end{align}
			whence
			the left-hand side of \eqref{ineq:a-posteriori-case-I} can be replaced by $\cnstL{cnst:a-posteriori-case-II}^{-1}\|\sigma - \D W(\GrRec_\ell u_\ell)\|_{L^{p'}(\Omega)}^{\widetilde{r}}$ with $\newcnstL\label{cnst:a-posteriori-case-II} \coloneqq 3\cnstS{cnst:cc-stress}(|\Omega| + \cnstL{cnst:G-u-h}^p + \cnstL{cnst:Du}^p)^{\widetilde{t}/\widetilde{t}'}$. This and Step 4 conclude the proof of
			\begin{align*}
				\lim_{\ell \to \infty} \D W(\GrRec_\ell u_\ell) = \sigma \text{ (strongly) in } L^{p'}(\Omega;\M). \tag*{\qedhere}
			\end{align*}
		\end{enumerate}
	\end{proof}
	\begin{remark}[necessity of $\varepsilon > 0$]\label{rem:eps=0}
		The counter example in \cite[Subsection 3.4]{OrtnerPraetorius2011} shows that the restriction $\varepsilon > 0$ is necessary. Indeed, for $k = 0$, the data $W$, $\Omega$, $\Gamma_{\mathrm{D}}$, $\Gamma_{\mathrm{N}}$, $f$, $g$, $u_\mathrm{D}$, and $(\Tcal_\ell)_{\ell \in \N_0}$ from \cite[Subsection 3.4]{OrtnerPraetorius2011}, there exists a sequence of discrete minimizers $(u_\ell)_{\ell \in \N_0}$ such that $\Jcal_\ell u_\ell \rightharpoonup 0$ weakly in $V$ and $\GrRec_\ell u_\ell \rightharpoonup 0$ weakly in $L^p(\Omega;\M)$ as $\ell \to \infty$, \emph{but} $\lim_{\ell \to \infty} \eta_\ell^{(\varepsilon)} \neq 0$.
	\end{remark}
	\subsection{The Lavrentiev gap}\label{sec:Lavrentiev-gap}
	A particular challenge in the computational calculus of variations is the Lavrentiev phenomenon $\inf E(\Acal) < \inf E(\Acal \cap W^{1,\infty}(\Omega))$ \cite{Lavrentiev1927}. Its presence is equivalent to the failure of standard conforming FEMs \cite[Theorem 2.1]{COrtner2010} in the sense that a wrong minimal energy is approximated. As a remedy, the nonconforming Crouzeix-Raviart FEM in \cite{Ortner2011,OrtnerPraetorius2011,BalciOrtnerStorn2021} can overcome the Lavrentiev gap under fairly general assumptions on $W$: Throughout the remaining parts of this section, let $W \in C^1(\M)$ be convex with the one-sided lower growth 
	\begin{align*}
		\cnstS{cnst:growth-left-1}|A|^p - \cnstS{cnst:growth-left-2} \leq W(A) \text{ for all } A \in \M \text{ and some } 1 < p < \infty.
	\end{align*}
	(A two-sided growth of $W$ excludes a Lavrentiev gap.)
	Since there is no upper growth of $W$, the dual variable $\sigma \coloneqq \D W(\D u)$ is not guaranteed to be in $L^{p'}(\Omega;\M)$. This denies an access to the Euler-Lagrange equations and, therefore, the convergence analysis of \cite{Ortner2011,OrtnerPraetorius2011} solely relies on the Jensen inequality.
	For $k = 0$, HHO methods can overcome the Lavrentiev gap because the Crouzeix-Raviart FEM can.
	\begin{lemma}[lower-energy bound for $k = 0$]
		Let $k = 0$. There exists a positive constant $\newcnstL\label{cnst:HHOk=0LEB}$ such that, for all level $\ell \in \N_0$,
		\begin{align*}
			\min E_\ell(\Acal(\Tcal_\ell)) - \cnstL{cnst:HHOk=0LEB}\big(\|h_\ell f\|_{L^{p'}(\Omega)} - \osc_{\mathrm{N}}(g,\Fcal_\ell(\Gamma_{\mathrm{N}}))\big) \leq \min E(\Acal).
		\end{align*}
	\end{lemma}
	\begin{proof}
		Recall the discrete space $\CR^1(\Tcal_\ell;\R^m)$ of Crouzeix-Raviart finite element functions from \eqref{def:Crouzeix-Raviart}. Define
		\begin{align*}
			\CR^1_\mathrm{D}(\Tcal_\ell;\R^m) \coloneqq \{v_\CR \in \CR^1(\Tcal_\ell;\R^m) : v_\CR(\operatorname{mid}(F)) = 0 \text{ for all } F \in \Fcal_\ell(\Gamma_{\mathrm{D}})\}
		\end{align*}
		and the nonconforming interpolation $\mathrm{I}_\CR : V \to \CR^1(\Tcal_\ell;\R^m)$ \cite{CrouzeixRaviart1973} with
		\begin{align*}
			\mathrm{I}_\CR v (\operatorname{mid}(F)) \coloneqq \int_F v \d{s}/|F| \quad\text{for all } F \in \Fcal_\ell, v \in V.
		\end{align*}
		The discrete CR-FEM minimizes the non-conforming energy
		\begin{align*}
			E_\mathrm{NC}(v_\CR) \coloneqq \int_\Omega (W(\D_\pw v_\CR) - \Pi_{\Tcal_\ell}^0 f \cdot v_\CR) \d{x} - \int_{\Gamma_{\mathrm{N}}} \Pi_{\Fcal_\ell}^0 g \cdot v_\CR \d{s}
		\end{align*}
		among $v_\CR \in \Acal_\mathrm{NC} \coloneqq \Ical_\mathrm{NC} u_\mathrm{D} + \CR^1_\mathrm{D}(\Tcal_\ell;\R^m)$.
		A straight-forward modification of the proof of \cite[Lemma 4]{OrtnerPraetorius2011} shows, for a positive constant $\cnstL{cnst:HHOk=0LEB} > 0$, that
		\begin{align}
			\min E_\mathrm{NC}(\Acal_\mathrm{NC}) - \cnstL{cnst:HHOk=0LEB}\big(\|h_\ell f\|_{L^{p'}(\Omega)} - \osc_{\mathrm{N}}(g,\Fcal_\ell(\Gamma_{\mathrm{N}}))\big) \leq \min E(\Acal)
			\label{ineq:CR-LEB}
		\end{align}
		Notice that $\Ical_\CR$ does \emph{not} provide the $L^2$ orthogonality $\Ical_\CR v - v \perp P_0(\Tcal_\ell;\R^m)$ in $L^2(\Omega;\R^m)$, but $(\Ical_\CR v - v)|_F \perp P_0(F;\R^m)$ in $L^2(F;\R^m)$ for all $F \in \Fcal_\ell(\Gamma_{\mathrm{N}})$ and $v \in V$. Hence
		the Neumann boundary data oscillations $\osc_{\mathrm{N}}(g,\Fcal_\ell(\Gamma_{\mathrm{N}}))$ arise in \eqref{ineq:CR-LEB}, but $\|h_\ell f\|_{L^{p'}(\Omega)}$ cannot be replaced by $\osc(f,\Tcal_\ell)$. For any $v_\CR \in \Acal_\mathrm{NC}$, $v_\ell \coloneqq (\Pi_{\Tcal_\ell}^0 v_\CR, \Pi_{\Fcal_\ell}^0 v_\CR) \in \Acal(\Tcal_\ell)$ satisfies $\GrRec_\ell v_\ell = \D_\pw v_\CR$ and hence, $\min E_\ell(\Acal(\Tcal_\ell)) \leq \min E_\mathrm{NC}(\Acal_\mathrm{NC})$. This and \eqref{ineq:CR-LEB} conclude the proof.
	\end{proof}
	\noindent The discrete compactness from \Cref{thm:conv_analysis:discrete_compactness}, the LEB in \eqref{ineq:CR-LEB}, and straightforward modifications of the proof of \Cref{thm:plain-convergence} lead to $\lim_{\ell \to \infty} E_\ell(u_\ell) = \min E(\Acal)$ for the output $(u_\ell)_{\ell \in \N_0}$ of the adaptive algorithm in \Cref{sec:adaptive_algorithm} with the refinement indicator, for all $T \in \Tcal_\ell$,
	\begin{align*}
		\eta_\ell^{(\varepsilon)}(T) \coloneqq \mu_\ell^{(\varepsilon)}(T) +|T|^{p'/n}\|f\|_{L^{p'}(T)}^{p'} + |T|^{1/n}\sum_{F \in \Fcal_\ell(T) \cap \Fcal_\ell(\Gamma_{\mathrm{N}})} \|(1 - \Pi_F^0) g\|_{L^{p'}(F)}^{p'}.
	\end{align*}
	For $k \geq 1$, the consistency error $\sigma_\ell - \D W(\GrRec_\ell u_\ell)$ arises in \eqref{ineq:LEB-II}, but is not guaranteed to be bounded in $L^{p'}(\Omega;\M)$ in the limit as $\ell \to \infty$ in general. Thus, in the absence of further conditions, the convergence $\lim_{\ell \to \infty} E_\ell(u_\ell) = \min E(\Acal)$ cannot be proven for $k \geq 1$ with this methodology.
	
	\section{Stabilized HHO method on polytopal meshes}\label{sec:polytopes}
	The classical HHO methodology \cite{DiPietroErnLemaire2014, DiPietroErn2015} allows even polytopal partitions of the domain $\Omega$.
	The assumption \ref{assumption-mesh-1} on the mesh follow the works \cite{DiPietroErn2012,DiPietroErnLemaire2014,DiPietroErn2015}. 
	\subsection{Polytopal meshes}
	Let $\Mcal_\ell$ be a finite collection of closed polytopes of positive volume with overlap of volume measure zero that cover $\overline{\Omega} = \cup_{K \in \Mcal_\ell} K$. A side $S$ of the mesh $\Mcal_\ell$ is the (in general disconnected) closed subset of a hyperplane $H_S \subset \Omega$ with positive ($(n-1)$-dimensional) surface measure such that either (a) there exist $K_1, K_2 \in \Mcal_\ell$ with $S = \partial K_1 \cap \partial K_2 \cap H_S$ (interior side) or (b) there exists $K \in \Mcal_\ell$ with $S = \partial K \cap \partial \Omega \cap H_S$ (boundary side). Let $\Sigma_\ell$ denote the set of all sides of $\Mcal_\ell$ and adapt
	the notation $\Sigma_\ell(K)$, $\Sigma_\ell(\Omega)$, $\Sigma_\ell(\Gamma_\mathrm{D})$, and $\Sigma_\ell(\Gamma_{\mathrm{N}})$ from \Cref{sec:regular_triangulation}. The convergence results of this section are established under the assumptions \ref{assumption-mesh-1}--\ref{assumption-mesh-2} below.
	\begin{enumerate}[wide,label = (M\arabic*)]
		\item Assume that there exists a universal constant $\varrho > 0$ such that, for all level $\ell \in \N_0$, $\Mcal_\ell$ admits a shape-regular simplicial subtriangulation $\Tcal_\ell$ with the shape regularity $\geq \varrho$ defined in \Cref{sec:regular_triangulation} and, for each simplex $T \in \Tcal_\ell$, there exists a unique cell $K \in \Mcal_\ell$ with $T \subseteq K$ and $\varrho h_K \leq h_T$.\label{assumption-mesh-1}
		\item Assume the existence of a universal constant $0 < \gamma < 1$ such that, $|\widehat{K}| \leq \gamma |K|$ holds for all $K \in \Mcal_\ell\setminus \Mcal_{\ell+1}$, $\widehat{K} \in \Mcal_{\ell+1}$ with $\widehat{K} \subset K$, and level $\ell \in \N_0$, i.e., the volume measure of all children $\widehat{K}$ of a refined cell $K$ is at most $\gamma|K|$.\label{assumption-mesh-2}
	\end{enumerate}
	The assumption \ref{assumption-mesh-1} is typical for the error analysis of HHO methods on polytopal meshes, cf., e.g., \cite{DiPietroErn2012,DiPietroErnLemaire2014,DiPietroErn2015,DiPietroDroniou2017,ErnZanotti2020}. The assumption \ref{assumption-mesh-2} holds for the newest-vertex bisection on simplicial triangulations with $\gamma = 1/2$.
	\begin{remark}[equivalence of side lengths]\label{rem:equivalence-mesh-size}
		The assumption (M1) ensures that $h_S \approx h_K \approx |K|^{1/n}$ holds for all $K \in \Mcal_\ell$ and $S \in \Sigma_\ell(K)$ with equivalence constants that exclusively depend on the universal constant $\varrho$ in \ref{assumption-mesh-1} \cite[Lemma 1.42]{DiPietroErn2012}.
	\end{remark}
	\begin{lemma}[mesh-size reduction]\label{rem:mesh-size-reduction}
		Suppose that the sequence $(\Mcal_\ell)_{\ell \in \N_0}$ satisfies \ref{assumption-mesh-2}, then the mesh-size function $h_\ell \in P_0(\Mcal_\ell)$ with $h_\ell|_K \coloneqq |K|^{1/n}$ for all $K \in \Mcal_\ell$ satisfies $\lim_{\ell \to \infty} \|h_\ell\|_{L^\infty(\Omega_\ell)} = 0$ for $\Omega_\ell \coloneqq \mathrm{int}(\cup (\Mcal_\ell\setminus\Mcal_{\ell+1}))$.
	\end{lemma}
	\begin{proof}
		Given any $j \in \N_0$ and $\alpha_j \coloneqq \gamma^{j}|\Omega|$, define the set $\Mcal(j) \subset \cup_{\ell \in \N_0} \Mcal_\ell$ of all polytopes $K$ with volume measure $\alpha_{j+1} < |K| \leq \alpha_j$. Since the volume measure of any refined polytope is at least reduced by the factor $\gamma$, the polytopes of $\Mcal(j)$ are \emph{not} children of each other and so $|K \cap T| = 0$ holds for any two distinct polytopes $K, T \in \Mcal(j)$. This implies that the cardinality $|\Mcal(j)|$ of $\Mcal(j)$ satisfies $|\Mcal(j)| < \gamma^{-(j+1)}$.
		For any level $\ell \in \N_0$, select some $K_\ell \in \Mcal_\ell \setminus \Mcal_{\ell+1}$ with $|K_\ell| = \|h_\ell\|_{L^{\infty}(\Omega_\ell)}^n$. Since $K_\ell \notin \Mcal_j$ for all $j > \ell$, the polytopes $K_0,K_1,K_2,\dots$ are pairwise distinct. Given $N \in \N_0$, the number $|\{\ell \in \N_0 : |K_\ell| > \alpha_{N+1}\}|$ of all indices $\ell \in \N_0$ with $|K_\ell| > \alpha_{N+1}$ is bounded by $|\Mcal(0)| + |\Mcal(1)| + \dots + |\Mcal(N)| \leq (\gamma^{-(N+1)} - 1)/(1 - \gamma)$.
		Hence there exists a maximal index $L$ such that $\|h_\ell\|^n_{L^\infty(\Omega_\ell)} = |K_\ell| \leq \alpha_{N+1}$ for all $\ell \geq L$. Notice that \Cref{lem:mesh-size-reduction} follows for simplicial triangulations with $\gamma = 1/2$.
	\end{proof}
	\subsection{Stabilization}\label{sec:stabilization}
	The classical HHO method \cite{DiPietroDroniou2017,AbbasErnPignet2018} utilizes a gradient reconstruction $\GrRec_\ell : V(\Mcal_\ell) \to \Sigma(\Mcal_\ell)$ in the space $\Sigma(\Mcal_\ell) \coloneqq P_k(\Mcal_\ell;\M)$ of matrix-valued piecewise polynomials of total degree at most $k$. The discrete seminorm $\|\bullet\|_\ell$ of $V(\Mcal_\ell)$ and the operators $\Ical_\ell$, $\PotRec_\ell$, $\GrRec_\ell$ of this section are defined by the formulas \eqref{def:discrete-norm}--\eqref{def:gradient_reconstruction} in \Cref{sec:discrete_ansatz_space} with adapted notation, i.e., $\Tcal_\ell$ (resp.~$\Fcal_\ell$) is replaced by $\Mcal_\ell$ (resp.~$\Sigma_\ell$).
	\begin{remark}[need of stabilization]
		The kernel of the gradient reconstruction $\GrRec_\ell$ restricted to $V_\mathrm{D}(\Mcal_\ell)$ is not trivial. For instance, any $v_\ell = (v_{\Mcal_\ell}, 0) \in V_\mathrm{D}(\Mcal_\ell)$ with $v_{\Mcal_\ell} \in P_k(\Mcal_\ell;\R^m)$ and $v_{\Mcal_\ell} \perp P_{k-1}(\Mcal_\ell;\R^m)$ (with the convention $P_{-1}(\Mcal_0;\R^m) \coloneqq \{0\}$) satisfies $\GrRec_\ell v_\ell = 0$ and the norm equivalence in \Cref{lem:stability_gradient_reconstruction}.a fails.
		On simplicial meshes, a gradient reconstruction in any discrete space $\Sigma(\Mcal_\ell)$ with $\RT_k(\Mcal_\ell;\M) \subset \Sigma(\Mcal_\ell)$ is stable, but the commutativity from \Cref{lem:stability_gradient_reconstruction}.b may fail if $\Sigma(\Mcal_\ell)$ is too large, e.g., $\Sigma(\Mcal_\ell) = P_{k+1}(\Mcal_\ell;\M)$ \cite{AbbasErnPignet2018}.
	\end{remark}
	\noindent The stabilization function $\sfrak_\ell : V(\Mcal_\ell) \times V(\Mcal_\ell) \to \R$ in the HHO methodology is defined,
	for any $u_\ell, v_\ell = (v_{\Mcal_\ell},v_{\Sigma_\ell}) \in V(\Mcal_\ell)$ and any side $S \in \Sigma_\ell(K)$ of $K \in \Mcal_\ell$ with diameter $h_S = \mathrm{diam}(S)$, by $\sfrak_\ell(u_\ell; v_\ell) \coloneqq \sum_{K \in \Mcal_\ell} \sfrak_K(u_\ell; v_\ell)$ and
	\begin{align}
		\begin{split}
			\Scal_{K,S} v_\ell \coloneqq \Pi_S^k (v_S - v_K - (1 - \Pi_K^k) (\PotRec_\ell v_\ell)|_K) \in P_k(S;\R^m),\\
			\sfrak_K(u_\ell; v_\ell) \coloneqq \sum_{S \in \Sigma_\ell(K)} h_S^{1-p} \int_S |{\Scal_{K,S} u_\ell}|^{p-2} \Scal_{K,S} u_\ell \cdot \Scal_{K,S} v_\ell \d{s}.
		\end{split}
		\label{def:stabilization}
	\end{align}
	Notice that $\sfrak_\ell(\bullet;\bullet)$ is linear in the second component, but \emph{not} in the first unless $p = 2$. The relevant properties of $\sfrak_\ell(\bullet;\bullet)$ are summarized below.
	\begin{lemma}[stabilization]\label{lem:stabilization}
		Any $u_\ell, v_\ell = (v_{\Mcal_\ell},v_{\Sigma_\ell}) \in V(\Mcal_\ell)$, $v \in V$, and $K \in \Mcal_\ell$ satisfy (a)--(e) with parameters $p,r,s,t$ from \Cref{tab:parameters}.
		\begin{enumerate}[wide, label = (\alph*)]
			\item\label{lem:stabilization-a} $\|v_\ell\|_\ell^p \approx \|\GrRec_\ell v_\ell\|_{L^p(\Omega)}^p + \sfrak_{\ell}(v_\ell;v_\ell)$.
			\item\label{lem:stabilization-b} $s_K(\Ical_\ell v; \Ical_\ell v)^{1/p} \lesssim \min_{\varphi_{h} \in P_{k+1}(K;\R^m)} \|\D (v - \varphi_{h})\|_{L^p(K)}$. In particular, if $v \in W^{k+2,p}(K;\R^m)$, then $s_K(\Ical_\ell v; \Ical_\ell v)^{1/p} \lesssim h_K^{k+1}|v|_{W^{k+2,p}(K)}$.
			\item\label{lem:stabilization-c}
			$\begin{aligned}[t]
				\sfrak_K(v_\ell;v_\ell) &\lesssim h_K^{-p}\|\Pi_{K}^k(\PotRec_\ell v_\ell - v_{K})\|_{L^p(K)}^p\\
				&\qquad + \sum_{S \in \Sigma_\ell(K)} h_S^{1-p}\|\Pi_S^k((\PotRec_\ell v_\ell)_K - v_S)\|_{L^p(S)}^p.
			\end{aligned}$
			\item\label{lem:stabilization-d} $\sfrak_K(u_\ell;v_\ell) \leq \sfrak_K(u_\ell;u_\ell)^{1/p'}\sfrak_K(v_\ell;v_\ell)^{1/p}$.
			\item\label{lem:stabilization-e}
			$\begin{aligned}[t]
				&\sum_{K \in \Mcal_\ell} \sum_{S \in \Sigma_\ell(K)} \|h_S^{-1/p'}\Scal_{K,S} (u_\ell - v_\ell)\|_{L^p(S)}^r\\
				&\lesssim \big(1 + \sfrak_{\ell}(u_\ell;u_\ell) + \sfrak_\ell(v_\ell;v_\ell)\big)^{t/t'} (\sfrak_\ell(v_\ell;v_\ell)/p - \sfrak_\ell(u_\ell;v_\ell) + \sfrak_{\ell}(u_\ell;u_\ell)/p').
			\end{aligned}$
		\end{enumerate}
	\end{lemma}
	\begin{proof}
		The norm equivalence in \ref{lem:stabilization-a} is established \cite[Lemma 4]{DiPietroErnLemaire2014} for $p= 2$ and extended to $1 \leq p < \infty$ in \cite[Lemma 5.2]{DiPietroDroniou2017}; the approximation property \ref{lem:stabilization-b} is \cite[Lemma 3.2]{ErnZanotti2020}.
		The upper bound \ref{lem:stabilization-c} follows immediately from a triangle and a discrete trace inequality. The proof of \ref{lem:stabilization-d} concerns $K \in \Mcal_\ell$ and $S \in \Sigma_\ell(K)$. A Hölder inequality with the exponents $p$, $p'$ and $1-p+1/p' = (1-p)/p'$ show
		\begin{align*}
			&h_S^{1-p} \int_S |\Scal_{K,S} u_\ell|^{p-2} \Scal_{K,S} u_\ell \cdot \Scal_{K,S} v_\ell \d{s}\\
			&\quad\leq \|h_S^{(1-p)/p'}|\Scal_{K,S} u_\ell|^{p-2}\Scal_{K,S} u_\ell\|_{L^{p'}(S)}\|h_S^{-1/p'}\Scal_{K,S} v_\ell\|_{L^p(S)}\\
			&\quad= \|h_S^{-1/p'}\Scal_{K,S} u_\ell\|_{L^p(S)}^{p/p'}\|h_S^{-1/p'}\Scal_{K,S} v_\ell\|_{L^p(S)}.
		\end{align*}
		The sum of this over all $S \in \Sigma_\ell(K)$ and a Cauchy inequality prove \ref{lem:stabilization-d}.
		The proof of \ref{lem:stabilization-e} departs from the function $W(a) \coloneqq |a|^p/p$ for $a \in \R^m$ with the convexity control \eqref{ineq:cc-primal}.
		The integral of \eqref{ineq:cc-primal} over the side $S$ leads to \eqref{ineq:cc-primal-integrated} for all $\varrho, \xi \in L^p(S;\R^m)$ and $\Omega$ (resp.~$\M$) replaced by $S$ (resp.~$\R^m$).
		The choice $\varrho \coloneqq h_S^{-1/p'}\Scal_{K,S} v_\ell$ and $\xi \coloneqq h_S^{-1/p'}\Scal_{K,S} u_\ell$ in \eqref{ineq:cc-primal-integrated} leads to
		\begin{align}
			&(3\cnstS{cnst:cc-primal})^{-1}(|S| + \|h_S^{-1/p'}\Scal_{K,S} u_\ell\|_{L^p(S)}^p + \|h_S^{-1/p'}\Scal_{K,S} v_\ell\|_{L^p(S)}^p)^{-t/t'}\nonumber\\
			&\quad\times \|h_S^{-1/p'}\Scal_{K,S} (u_\ell - v_\ell)\|_{L^p(S)}^r
			\leq \|h_S^{-1/p'} \Scal_{K,S} v_\ell\|^p_{L^p(S)}/p
			\label{ineq:cc-stabilization}\\
			&\quad\qquad - \int_S h_S^{1-p}|\Scal_{K,S} u_\ell|^{p-2}\Scal_{K,S} u_\ell \cdot \Scal_{K,S} v_\ell \d{s} + \|h_S^{-1/p'}\Scal_{K,S} u_\ell\|_{L^p(S)}^p/p'.\nonumber
		\end{align}
		The sum of this over all $S \in \Sigma_\ell(K)$ and $K \in \Mcal_\ell$ concludes the proof of \ref{lem:stabilization-e}.
	\end{proof}
	\subsection{Stabilized HHO method on a polytopal mesh}
	The discrete problem minimizes the discrete energy
	\begin{align}
		E_\ell(v_\ell) \coloneqq \int_\Omega (W(\GrRec_\ell v_\ell) - f \cdot v_{\Mcal_\ell}) \d{x} - \int_{\Gamma_{\mathrm{N}}} g \cdot v_{\Sigma_\ell} \d{s} + \sfrak_\ell(v_\ell;v_\ell)/p\label{def:discrete-energy-stabilized}
	\end{align}
	among $v_\ell = (v_{\Mcal_\ell}, v_{\Sigma_\ell}) \in \Acal(\Mcal_\ell)$.
	\begin{theorem}[discrete minimizers]\label{lem:existence-discrete-solutions-sHHO}
		The minimal discrete energy $\inf E_\ell(\Acal(\Mcal_\ell))$ is attained. There exists a positive constant $\newcnstL\label{cnst:G-u-h-sHHO} > 0$ that merely depends on \cnstS{cnst:growth-left-1}, \cnstS{cnst:growth-left-2}, $\Omega$, $\Gamma_{\mathrm{D}}$, $u_\mathrm{D}$, $f$, $g$, $\varrho$ in \ref{assumption-mesh-1}, $k$, and $p$ with $\|\GrRec_\ell u_\ell\|_{L^p(\Omega)}^p + \sfrak_\ell(u_\ell;u_\ell) \leq \cnstL{cnst:G-u-h-sHHO}^p$ for any discrete minimizer $u_\ell \in \arg\min E_\ell(\Acal(\Mcal_\ell))$. 
		Any discrete stress $\sigma_\ell \coloneqq \Pi_{\Mcal_\ell}^k \D W(\GrRec_\ell u_\ell)$ satisfies the discrete Euler-Lagrange equations
		\begin{align}
			\int_\Omega \sigma_\ell : \GrRec_\ell v_\ell \d{x} = \int_\Omega f \cdot v_{\Mcal_\ell} \d{x} + \int_{\Gamma_{\mathrm{N}}} g \cdot v_{\Sigma_\ell} \d{s} - \sfrak_\ell(u_\ell;v_\ell)
			\label{eq:dELE-sHHO}
		\end{align}
		for all $v_\ell = (v_{\Mcal_\ell}, v_{\Sigma_\ell}) \in V_\mathrm{D}(\Mcal_\ell)$. 
		If $W$ satisfies \eqref{ineq:cc-primal}, then $u_\ell = \arg \min E_\ell(\Acal(\Mcal_\ell))$ is unique. If $W$ satisfies \eqref{ineq:cc-stress}, then $\D W(\GrRec_\ell u_\ell) \in L^{p'}(\Omega;\M)$ is unique (independent of the choice of a (possibly non-unique) discrete minimizer $u_\ell$).
	\end{theorem}
	\begin{proof}
		The proof follows that of \Cref{lem:existence-discrete-solutions}. The norm equivalence in \Cref{lem:stabilization}.a and the lower growth of $W$ lead to the coercivity of $E_\ell$ in $\Acal(\Mcal_\ell)$ with respect to the seminorm $\|\bullet\|_\ell^p \approx \|\GrRec_\ell \bullet\|_{L^p(\Omega)}^p + \sfrak_\ell(\bullet;\bullet)$ from \Cref{lem:stabilization}.a. This implies the existence of discrete minimizers and the bound $\|\GrRec_\ell u_\ell\|_{L^p(\Omega)}^p + \sfrak_\ell(u_\ell;u_\ell) \leq \cnstL{cnst:G-u-h-sHHO}^p$ for all $u_\ell \in \arg \min E_\ell(\Acal(\Mcal_\ell))$.
		If $W$ satisfies \eqref{ineq:cc-primal}, then the strict convexity of $W$ and of $\sfrak_\ell$ in \Cref{lem:stabilization}.c leads to the uniqueness of $u_\ell = \arg\min E_\ell(\Acal(\Mcal_\ell))$. If $W$ satisfies \eqref{ineq:cc-stress}, then the uniqueness of $\D W(\GrRec_\ell u_\ell)$ follows as in \cite{CPlechac1997,CLiu2015,CarstensenTran2020}.
	\end{proof}
	\noindent The following lemma extends \Cref{lem:conforming-companion} to polytopal meshes.
	\begin{lemma}\label{lem:conforming-companion-sHHO}
		There exists a linear operator $\Jcal_{\ell} : V(\Mcal_\ell) \to V$ such that any $v_\ell = (v_{\Mcal_\ell}, v_{\Sigma_\ell}) \in V(\Mcal_\ell)$ satisfies
		\begin{align}
			\Pi_{\Mcal_\ell}^k \Jcal_\ell v_\ell = v_{\Mcal_\ell} \quad\text{and}\quad \Pi_{\Sigma_\ell}^k \Jcal_{\ell} v_\ell = v_{\Sigma_\ell}
			\label{def:conforming-companion-sHHO}
		\end{align}
		and, for any $K \in \Mcal_\ell$, the estimate
		\begin{align}
			\begin{split}
				&\|\GrRec_\ell v_\ell - \D \Jcal_\ell v_\ell\|_{L^p(K)}^p \lesssim \sum_{S \in \Sigma_\ell(\Omega), S \cap K \neq \emptyset} h_S^{1-p}\|[\PotRec_\ell v_\ell]_S\|_{L^p(S)}^p\\
				&\qquad + \sum_{S \in \Sigma_\ell(K)} h_S^{1-p}\|(\PotRec_\ell v_\ell)|_K - v_S\|_{L^p(S)}^p + h_K^{-p} \|\PotRec_\ell v_\ell - v_K\|_{L^p(K)}^p.
			\end{split}
			\label{ineq:conforming-companion-estimate-sHHO}
		\end{align}
		In particular, $\Jcal_\ell$ is stable in the sense that $\|\D \Jcal_\ell v_\ell\|_{L^p(\Omega)} \leq \Lambda_1 \|v_\ell\|_\ell$ holds with the constant $\Lambda_1$ that exclusively depends on $k$, $p$, and $\varrho$ in \ref{assumption-mesh-1}.
	\end{lemma}
	\begin{proof}
		The construction of the conforming operator $\Jcal_\ell$ on polytopal meshes in \cite[Section 5]{ErnZanotti2020} utilizes averaging and bubble-function techniques on the subtriangulation $\Tcal_\ell$ and give rise an upper bound of $\|\GrRec_\ell v_\ell - \D \Jcal_\ell v_\ell\|_{L^p(K)}^p$, namely
		\begin{align}
			\begin{split}
				\sum_{S \in \Sigma_\ell(\Omega), S \cap K \neq \emptyset} h_S^{1-p}\|[\PotRec_\ell v_\ell]_S\|_{L^p(S)}^p + \sum_{T \in \Tcal_\ell, T \subset K} h_T^{-p}\|\Pi_T^k(\PotRec_\ell v_\ell - v_K)\|^p_{L^p(T)}\\
				+ \sum_{S \in \Sigma_\ell(K)} \sum_{F \in \Fcal_\ell, F \subset S} h_F^{1-p} \|\Pi_{F}((\PotRec_\ell v_\ell)|_K - v_S)\|^p_{L^p(F)}.
			\end{split}
			\label{def:RHS-conforming-companion-polytopes}
		\end{align}
		Since the $L^2$ projection $\Pi_T^k$ (resp.~$\Pi_F^k$) is stable in $L^p(T;\R^m)$ (resp.~$L^p(F;\R^m)$) \cite[Lemma 3.2]{DiPietroDroniou2017}, it can be omitted in \eqref{def:RHS-conforming-companion-polytopes}. This, the equivalence $h_T \approx h_K$ for all $K \in \Mcal_\ell$, $T \in \Tcal_\ell$ with $T \subset K$ from \ref{assumption-mesh-1}, and $h_F \approx h_S$ for all $S \in \Sigma_\ell$, $F \in \Fcal_\ell$ with $F \subset S$ from \ref{assumption-mesh-1} and \Cref{rem:equivalence-mesh-size} show \eqref{ineq:conforming-companion-estimate-sHHO}.
		This implies the stability $\|\D \Jcal_\ell v_\ell\|_{L^p(\Omega)} \lesssim \|v_\ell\|_\ell$, cf., e.g., \cite[Subsection 4.3]{ErnZanotti2020} for more details. Notice that the computation of the right-hand side of \eqref{ineq:conforming-companion-estimate-sHHO} does \emph{not} require explicit information on the subtriangulation $\Tcal_\ell$.
	\end{proof}
	\begin{remark}[discrete compactness]\label{rem:discrete-compactness-sHHO}
		The discrete compactness from \Cref{thm:conv_analysis:discrete_compactness} holds verbatim with $\Tcal_\ell$ (resp.~$\Fcal_\ell$) replaced by $\Mcal_\ell$ (resp.~$\Sigma_\ell$). Notice that $\GrRec_\ell$ from \Cref{sec:stabilization} and $\Jcal_\ell$ from \Cref{lem:conforming-companion-sHHO} in this section are different objects. With adapted notation, all arguments from the proof of \Cref{thm:conv_analysis:discrete_compactness} apply verbatim.
		Indeed, the commutativity $\Pi_{\Sigma(\Mcal_\ell)} \D v = \GrRec_\ell \Ical_\ell v$ for all $v \in V$ from \Cref{lem:stability_gradient_reconstruction}.b remains valid \cite{DiPietroDroniou2017,AbbasErnPignet2018}. This and \eqref{def:conforming-companion-sHHO} imply the $L^2$ orthogonality $\GrRec_\ell v_\ell - \D \Jcal_\ell v_\ell \perp \Sigma(\Mcal_\ell)$ for any $v_\ell \in V(\Mcal_\ell)$. This is the key argument in the proof of \Cref{thm:conv_analysis:discrete_compactness} and provides a positive power of the mesh-size in \eqref{eq:conv_analysis:discrete_compactness_assumption}.
	\end{remark}
	\subsection{Proof of \Cref{thm:plain-convergence-sHHO}}
	Given any $K \in \Mcal_\ell$, \Cref{lem:conforming-companion-sHHO} motivates the refinement indicator
	\begin{align*}
		\eta_\ell^{(\varepsilon)}(K) &\coloneqq \mu_\ell^{(\varepsilon)}(K) + |K|^{\varepsilon p'/n}\|\sigma_\ell - \D W(\GrRec_\ell u_\ell)\|_{L^{p'}(K)}^{p'}\\
		&\quad + |K|^{p'/n}\|(1 - \Pi_K^k) f\|_{L^{p'}(K)}^{p'} + |K|^{1/n}\sum_{S \in \Sigma_\ell(K) \cap \Sigma_\ell(\Gamma_{\mathrm{N}})}\|(1 - \Pi_S^k) g\|_{L^{p'}(S)}^{p'}
	\end{align*}
	with $\begin{aligned}[t]
		&\mu_\ell^{(\varepsilon)}(K) \coloneqq |K|^{(\varepsilon p - p)/n} \|\PotRec_\ell u_\ell - u_K\|_{L^p(K)}^p\\
		&~ + |K|^{(\varepsilon p + 1 - p)/n} \Big(\sum_{S \in \Sigma_\ell(K)\cap \Sigma_\ell(\Gamma_{\mathrm{D}})} \|\PotRec_\ell u_\ell - u_\D\|_{L^p(S)}^p\\
		&~ +  \sum_{S \in \Sigma_\ell(K)\cap \Sigma_\ell(\Omega)} \|[\PotRec_\ell u_\ell]_S\|_{L^p(S)}^p + \sum_{S \in \Sigma_\ell(K)}  \|(\PotRec_\ell u_\ell)|_K - u_S\|_{L^p(S)}^p\Big).
	\end{aligned}$\\
 	The remaining parts of this section are devoted to the proof of the convergence results in \Cref{thm:plain-convergence-sHHO}.
	\begin{proof}[Proof of \Cref{thm:plain-convergence-sHHO}]
		The proof follows that of \Cref{thm:plain-convergence}.
		\begin{enumerate}[wide]
			\item[\emph{Step 1 establishes $\lim_{\ell \to \infty} \eta_\ell^{(\varepsilon)} = 0$.}] The key argument from Step 1 of the proof of \Cref{thm:plain-convergence} is the positive power of the mesh size in $\eta_\ell^{(\varepsilon)}$ in the sense that
			\begin{align*}
				\eta_\ell^{(\varepsilon)}(\Mcal_\ell\setminus\Mcal_{\ell+1}) &\lesssim \|h_\ell\|_{L^\infty(\Omega_\ell)}^{\varepsilon p}(\|u_\ell\|_\ell^p + \|\D u_\mathrm{D}\|_{L^p(\Omega)}^p) + \|h_\ell\|_{L^\infty(\Omega_\ell)}^{\varepsilon p'} \|\D W(\GrRec_\ell u_\ell)\|_{L^{p'}(\Omega)}^{p'}\\
				&\qquad + \|h_\ell\|_{L^{\infty}(\Omega_\ell)}^{p'}\|f\|_{L^{p'}(\Omega)}^{p'} + \|h_\ell\|_{L^\infty(\Omega_\ell)}\|g\|_{L^{p'}(\Gamma_{\mathrm{N}})}^{p'}.
			\end{align*}
			Hence $\lim_{\ell \to \infty} \|h_\ell\|_{L^\infty(\Omega_\ell)} = 0$ from \Cref{rem:mesh-size-reduction} implies $\lim_{\ell \to \infty} \eta_\ell^{(\varepsilon)}(\Mcal_\ell\setminus\Mcal_{\ell+1}) = 0$. This and the Dörfler marking in \eqref{ineq:Doerfler-marking} conclude $\lim_{\ell \to \infty} \eta_\ell^{(\varepsilon)} = 0$.
			\item[\emph{Step 2 provides a LEB with the extra stabilization term $\sfrak_{\ell}(u_\ell;\Ical_\ell u)$, namely}]
			\begin{align}
				\mathrm{LEB}_\ell &\coloneqq E_\ell(u_\ell) + \int_\Omega (1 - \Pi_{\Mcal_\ell}^k) \D W(\GrRec_\ell u_\ell) : \D u \d{x} - \sfrak_\ell(u_\ell;\Ical_\ell u)\label{ineq:LEB-sHHO}\\
				&\quad - \cnstL{cnst:oscillation}\big(\osc(f,\Mcal_\ell) + \osc_{\mathrm{N}}(g,\Sigma_{\ell}(\Gamma_{\mathrm{N}}))\big) \leq E(u) - \sfrak_\ell(u_\ell;u_\ell)/p' \leq E(u).\nonumber
			\end{align}
			The commutativity $\Pi_{\Sigma(\Mcal_\ell)} \D u = \GrRec_\ell \Ical_\ell u$ from \Cref{lem:stability_gradient_reconstruction}.b and the discrete Euler-Lagrange equations \eqref{eq:dELE-sHHO} show that
			\begin{align}
				\begin{split}
					\int_\Omega \sigma_\ell : (\D u - \GrRec_\ell u_\ell) \d{x} &= \int_\Omega f \cdot (\Pi_{\Mcal_\ell}^k u - u_{\Mcal_\ell}) \d{x}\\
					&\quad \int_{\Gamma_\mathrm{N}} g \cdot (\Pi_{\Sigma_\ell}^k u - u_{\Sigma_\ell}) \d{x} - \sfrak_\ell(u_\ell; \Ical_\ell u - u_\ell).
				\end{split}
				\label{eq:dELE-sHHO-LEB}
			\end{align}
			This, \eqref{ineq:convexity}, and \eqref{ineq:oscillation} (with adapted notation) conclude the proof of \eqref{ineq:LEB-sHHO}.
			\item[\emph{Step 3 establishes $\lim_{\ell \to \infty} E_\ell(u_\ell) = E(u)$}.] 	Notice from \eqref{ineq:conforming-companion-estimate-sHHO} that $\eta_\ell^{(\varepsilon)}$ is an upper bound for $\mu_\ell(u_\ell)$ in \eqref{eq:conv_analysis:discrete_compactness_assumption}.
			Hence the discrete compactness (from \Cref{rem:discrete-compactness-sHHO}) implies the existence of a (not relabelled) subsequence of $(u_\ell)_{\ell \in \N_0}$ and a weak limit $v \in \Acal$ such that $\Jcal_{\ell} u_\ell \rightharpoonup v$ weakly in $V$ and $\GrRec_\ell u_\ell \rightharpoonup \D v$ weakly in $L^p(\Omega;\M)$ as $\ell \to \infty$. The only difference between the LEB in \eqref{ineq:LEB-sHHO} and that in \eqref{ineq:LEB-II} for simplicial meshes is the additional term $\sfrak_{\ell}(u_\ell;\Ical u)$ in this proof. 
			\begin{lemma}[convergence of $\sfrak_{\ell}(u_\ell;\Ical u)$]\label{lem:convergence-stabilization}
				Given a sequence $(u_\ell)_{\ell \in \N_0}$ with $u_\ell \in V(\Mcal_\ell)$ for all $\ell \in \N_0$, suppose that $\sfrak_\ell(u_\ell;u_\ell) \leq \newcnstL\label{cnst:stabilization}$ for a universal constant $\cnstL{cnst:stabilization}$ independent of the level $\ell$ and $\lim_{\ell \to \infty} \eta_\ell^{(\varepsilon)} = 0$ with $\varepsilon \leq \min\{k+1,(k+1)/(p-1)\}$. Then
				\begin{align}
					\lim_{\ell \to \infty} \sfrak_\ell(u_\ell;\Ical u) = 0.
					\label{eq:lim-stabilization}
				\end{align}
			\end{lemma}
			\begin{proof}[Proof of \Cref{lem:convergence-stabilization}]
				The proof of \eqref{eq:lim-stabilization} first establishes this for smooth functions.
				Given any $\varphi \in C^\infty(\overline{\Omega};\R^m)$, the H\"older inequality from \Cref{lem:stabilization}.d, $h_K \approx |K|^{1/n}$ from \Cref{rem:equivalence-mesh-size}, and the interpolation error from \Cref{lem:stabilization}.b prove
				\begin{align}
					\begin{split}
						|\sfrak_{\ell}(u_\ell;\Ical_\ell \varphi)| &\leq \sum_{K \in \Mcal_\ell} \sfrak_{K}(u_\ell;u_\ell)^{1/p'}\sfrak_{K}(\Ical_\ell \varphi;\Ical_\ell\varphi)^{1/p}\\
						& \lesssim \Big(\sum_{K \in \Mcal_\ell} 	|K|^{(k+1)p'/n}\sfrak_{K}(u_\ell;u_\ell)\Big)^{1/p'} |\varphi|_{W^{k+2,p}(\Omega)}.
					\end{split}
					\label{eq:limit_stabilization_smooth}
				\end{align}
				\Cref{lem:stabilization}.c implies that $\eta_\ell^{(\varepsilon)}$ controls the stabilization in the sense that
				\begin{align}
					\sum_{K \in \Mcal_\ell} |K|^{(k+1)p'/n}\sfrak_{K}(u_\ell;u_\ell) \lesssim \|h_\ell\|_{L^\infty(\Omega)}^{((k+1)p' - \varepsilon p)/n}\eta_\ell^{(\varepsilon)}.
					\label{ineq:convergence-stabilization-control}
				\end{align}
				The restriction $\varepsilon \leq (k+1)/(p-1)$ provides $(k+1)p' - \varepsilon p > 0$. Hence $\lim_{\ell \to \infty} \eta_\ell^{(\varepsilon)} = 0$ implies that the right-hand side of \eqref{ineq:convergence-stabilization-control} vanishes in the limit as $\ell \to \infty$. This and \eqref{eq:limit_stabilization_smooth}--\eqref{ineq:convergence-stabilization-control} lead to $\lim_{\ell \to \infty} \sfrak_{\ell}(u_\ell;\Ical_\ell \varphi) = 0$ for all $\varphi \in C^\infty(\overline{\Omega};\R^m)$.
				Given any $\delta > 0$, let $\varphi \in C^\infty(\overline{\Omega};\R^m)$ such that $\|\D(u - \varphi)\|_{L^p(\Omega)} \leq \delta$.
				The interpolation error from \Cref{lem:stabilization}.b proves $\sfrak_\ell(\Ical_\ell (u - \varphi); \Ical_\ell (u - \varphi)) \leq \newcnstL^p\label{cnst:sHHO:plain_convergence_stability} \|{\D (u - \varphi)}\|_{L^p(\Omega)}^p \leq \cnstL{cnst:sHHO:plain_convergence_stability}^p \delta^p$ with a universal constant $\cnstL{cnst:sHHO:plain_convergence_stability} > 0$. The convergence $\lim_{\ell \to \infty} \sfrak_{\ell}(u_\ell;\Ical_\ell \varphi) = 0$ implies the existence of $N \in \N_0$ with $|\sfrak_\ell(u_\ell;\Ical_\ell \varphi)| \leq \delta$ for all $\ell \geq N$. This, a triangle inequality, a H\"older inequality, and the bound $\sfrak_\ell(u_\ell;u_\ell) \leq \cnstL{cnst:stabilization}$ (by assumption) verify
				\begin{align*}
					|\sfrak_\ell(u_\ell; \Ical_\ell u)| &\leq |\sfrak_\ell(u_\ell;\Ical_\ell \varphi)| + |\sfrak_\ell(u_\ell; \Ical_\ell (u - \varphi))| \leq |\sfrak_\ell(u_\ell;\Ical_\ell \varphi)|\\
					&\qquad + \sfrak_\ell(u_\ell;u_\ell)^{1/p'}\sfrak_\ell(\Ical_\ell (u - \varphi); \Ical_\ell (u - \varphi))^{1/p} \leq (1 + \cnstL{cnst:stabilization}^{1/p'}\cnstL{cnst:sHHO:plain_convergence_stability}) \delta.
				\end{align*}
				This concludes the proof of $\lim_{\ell \to \infty} \sfrak_\ell(u_\ell;\Ical u) = 0$ in \eqref{eq:lim-stabilization}.
			\end{proof}
			\noindent We return to proof of \Cref{thm:plain-convergence-sHHO} and recall $\sfrak_\ell(u_\ell;u_\ell) \leq \cnstL{cnst:G-u-h-sHHO}^p$ from \Cref{lem:existence-discrete-solutions-sHHO} and $\lim_{\ell \to \infty} \eta_\ell^{(\varepsilon)} = 0$ from Step 1.
			Hence \Cref{lem:convergence-stabilization} applies and \eqref{eq:lim-stabilization} follows.
			With this additional argument \eqref{eq:lim-stabilization} and the remaining conclusions, that lead to \eqref{ineq:conv-analysis-liminf} in the proof of \Cref{thm:plain-convergence}, $E(u) \leq E(v) \leq \liminf_{\ell \to \infty} \mathrm{LEB}_\ell \leq E(u)$ follows for the weak limit $v$.
			This implies $\lim_{\ell \to \infty} E_\ell(u_\ell) = \lim_{\ell \to \infty} \mathrm{LEB}_\ell = E(u)$.
			Since $\sfrak_\ell(u_\ell;u_\ell)/p' \leq E(u) - \mathrm{LEB}_\ell$ from \eqref{ineq:LEB-sHHO}, $\sfrak_\ell(u_\ell;u_\ell)$ vanishes in the limit as $\ell \to \infty$. If $W$ satisfies \eqref{ineq:cc-primal}, then the choice $\varrho \coloneqq \D u$ and $\xi \coloneqq \GrRec_\ell u_\ell$ in \eqref{ineq:cc-primal-integrated}, \eqref{eq:dELE-sHHO-LEB} and the data oscillations from \eqref{ineq:oscillation} imply
			\begin{align}
				\cnstL{cnst:a-posteriori-case-I-sHHO}^{-1}\|\D u - \GrRec_\ell u_\ell\|_{L^p(\Omega)}^r + \sfrak_\ell(u_\ell;u_\ell)/p' \leq E(u) - \mathrm{LEB}_\ell
				\label{ineq:a-posteriori-case-I-sHHO}
			\end{align}
			with the constant $\newcnstL\label{cnst:a-posteriori-case-I-sHHO} \coloneqq 3\cnstS{cnst:cc-primal}(|\Omega| + \cnstL{cnst:Du}^p + \cnstL{cnst:G-u-h-sHHO}^p)^{t/t'}$ and $r,t$ from \Cref{tab:parameters}. This shows $\lim_{\ell \to \infty} \GrRec_\ell u_\ell = \D u$ (strongly) in $L^p(\Omega;\M)$. If $W$ satisfies \eqref{ineq:cc-stress}, then \eqref{ineq:cc-dual-integrated} holds and 
			$\cnstL{cnst:a-posteriori-case-I-sHHO}^{-1}\|\D u - \GrRec_\ell u_\ell\|_{L^p(\Omega)}^r$ on the left-hand side of \eqref{ineq:a-posteriori-case-I-sHHO} can be replaced by $\cnstL{cnst:a-posteriori-case-II-sHHO}^{-1}\|\sigma - \D W(\GrRec_\ell u_\ell)\|_{L^{p'}(\Omega)}^{\widetilde{r}}$ with $\newcnstL\label{cnst:a-posteriori-case-II-sHHO} \coloneqq 3\cnstS{cnst:cc-stress}(|\Omega| + \cnstL{cnst:Du}^p + \cnstL{cnst:G-u-h-sHHO}^p)^{\widetilde{t}/\widetilde{t}'}$ and $\widetilde{r},\widetilde{t}$ from \Cref{tab:parameters}. Hence $\lim_{\ell \to \infty} \D W(\GrRec_\ell u_\ell) = \sigma$ (strongly) in $L^{p'}(\Omega;\M)$.\qedhere
		\end{enumerate}
	\end{proof}
	\section{Numerical examples}\label{sec:numerical-examples}
	Some remarks on the implementation precede the numerical benchmarks for the three examples of \Cref{sec:examples} and the experiments in the Foss-Hrusa-Mizel example with the Lavrentiev gap in \Cref{sec:num-ex:FHM}.
	
	\subsection{Implementation}
	
	The realization in MATLAB follows that of \cite[Subsubsetion 5.1.1]{CarstensenTran2020} with the parameters
	$\texttt{FunctionTolerance} = \texttt{OptimalityTolerance} = \texttt{StepTolerance} = 10^{-15}$ and $\texttt{MaxIterations} = \texttt{Inf}$ for improved accuracy.
	
	The class of minimization problems at hand allows, in general, for multiple exact and discrete solutions. The numerical experiments select one (of those) by the approximation in \texttt{fminunc} with the initial value computed as follows. On the coarse initial triangulations $\Tcal_0$, the initial value $v_0 = (v_{\Tcal_0},v_{\Fcal_0}) \in V(\Tcal_0)$ is defined by $v_{\Tcal_0} \equiv 1$, $v_{\Fcal_0}|_F \equiv 1$ on any $F \in \Fcal_0(\Omega)$, and $v_{\Fcal_0}|_F = \Pi_F^k u_\mathrm{D}$ for all $F \in \Fcal_0(\Gamma_{\mathrm{D}})$. On each refinement $\Tcal_{\ell+1}$ of some triangulation $\Tcal_\ell$, the initial approximation is defined by a prolongation of the output $u_\ell$ of the call \texttt{fminunc} on the coarse triangulation $\Tcal_\ell$. The prolongation maps $u_\ell$ onto $v_{\ell+1} \coloneqq \Ical_{\ell+1} \Jcal_{\ell} u_\ell \in V(\Tcal_{\ell+1})$.
	
	The numerical integration of polynomials is exact with the quadrature formula in \cite{HammerStroud1956}: For non-polynomial functions such as $W(\GrRec_\ell v_\ell)$ with $v_\ell \in V(\Tcal_\ell)$, the number of chosen quadrature points allows for exact integration of polynomials of order $p(k+1)$ with the growth $p$ of $W$ and the polynomial order $k$ of the discretization; the same quadrature formula also applies to the integration of the dual energy density $W^*$ in \eqref{def:dual-energy}. The implementation is based on the in-house AFEM software package in MATLAB \cite{AlbertyCFunken1999,CBrenner2017}.
	Adaptive computations are carried out with $\theta = 0.5$, $\varepsilon = (k+1)/100$, and the polynomial degrees $k$ from \Cref{fig:legend}. Undisplayed computer experiments suggest only marginal influence of the choice of $\varepsilon$ on the convergence rates of the errors.
	
	The uniform or adaptive mesh-refinement leads to
	convergence history plots of the energy error $|E(u) - E_\ell(u_\ell)|$ or the stress error $\|\sigma - \nabla W(\GrRec_\ell u_\ell)\|_{L^{p'}(\Omega)}^2$ plotted against the number of degrees of freedom (ndof) in \Cref{fig:convergence-energy-p-Laplace}--\Cref{fig:convergence-energy-FHM} below. (Recall the scaling $\mathrm{ndof} \propto h_{\max}^2$ in 2D for uniform mesh refinements with maximal mesh size $h_{\max}$ in a log-log plot.)
	In the numerical experiments without a priori knowledge of $u$, the reference value displayed for $\min E(\Acal)$ stems from an Aitken extrapolation of the numerical results for a sequence of uniformly refined triangulations.
	\begin{figure}[h!]
		\centering
		\includegraphics{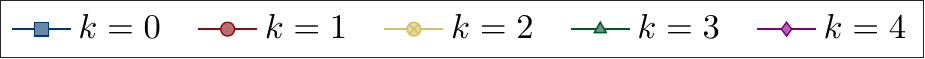}
		\caption{Polynomial degrees $k = 0,\dots,4$ in the numerical benchmarks of \Cref{sec:numerical-examples}}
		\label{fig:legend}
	\end{figure}
	\subsection{The $p$-Laplace equation}\label{sec:num-ex:p-Laplace}
	\begin{figure}[ht!]
		\begin{minipage}[t]{0.49\textwidth}
			\centering
			\includegraphics[scale=0.87]{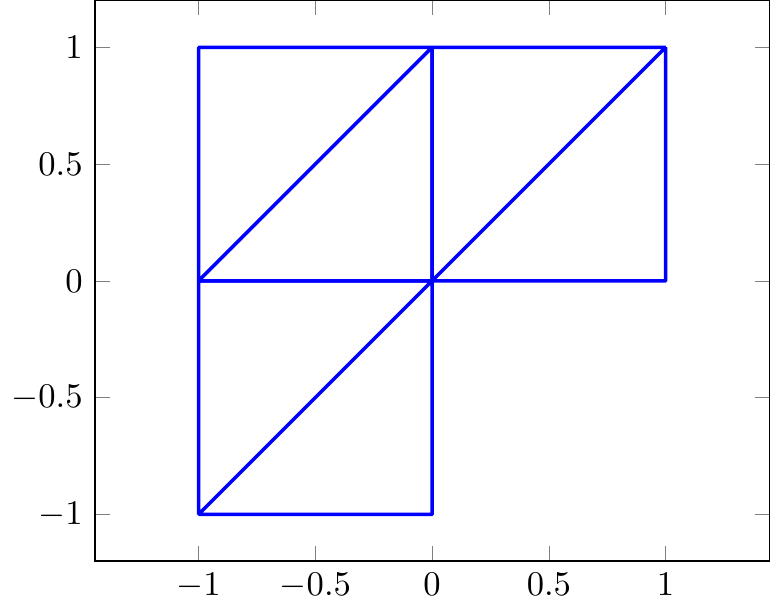}
		\end{minipage}\hfill
		\begin{minipage}[t]{0.49\textwidth}
			\centering
			\includegraphics[scale=0.87]{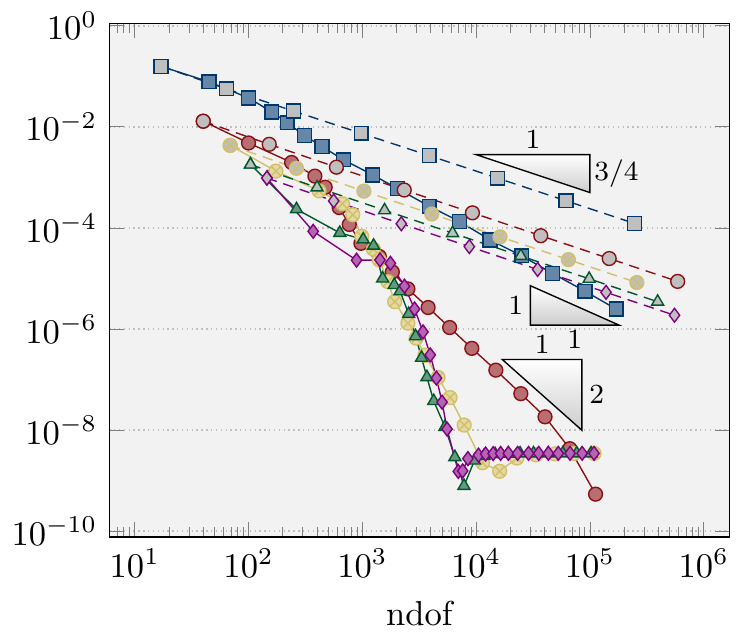}
		\end{minipage}\hfill
		\caption{Initial triangulation $\Tcal_0$ (left) of the L-shaped domain and convergence history plot (right) of $|E(u) - E_\ell(u_\ell)|$ with $k$ from \Cref{fig:legend} on uniform (dashed line) and adaptive (solid line) triangulations for the $p$-Laplace benchmark in \Cref{sec:num-ex:p-Laplace}}
		\label{fig:convergence-energy-p-Laplace}
	\end{figure}
	\begin{figure}[ht!]
		\begin{minipage}[t]{0.49\textwidth}
			\centering
			\includegraphics[scale=0.87]{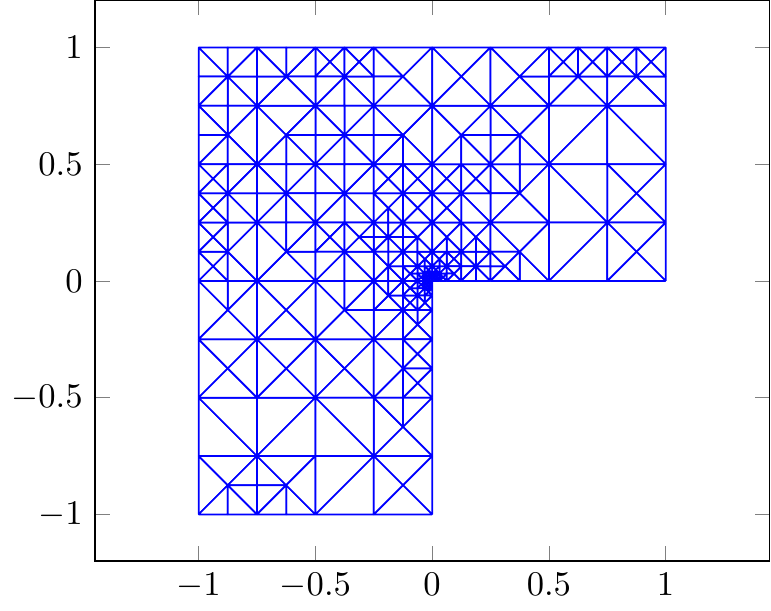}
		\end{minipage}\hfill
		\begin{minipage}[t]{0.49\textwidth}
			\centering
			\includegraphics[scale=0.87]{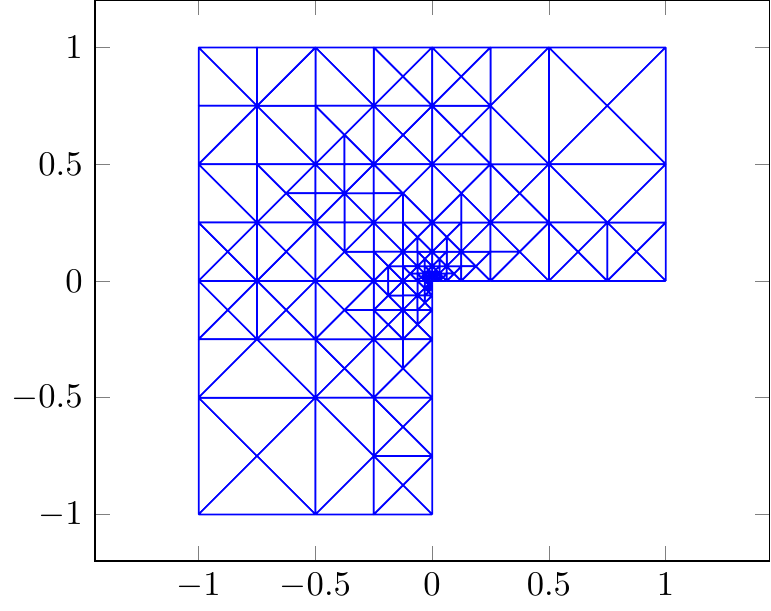}
		\end{minipage}\hfill
		\caption{Adaptive triangulations of the L-shaped domain into 492 triangles (1238 dofs) for $k = 0$ (left) and 490 triangles (7824 dofs) for $k = 3$ (right) for the $p$-Laplace benchmark in \Cref{sec:num-ex:p-Laplace}}
		\label{fig:adaptive-triangulation-p-Laplace}
	\end{figure}
	\begin{figure}[ht!]
		\begin{minipage}[t]{0.49\textwidth}
			\centering
			\includegraphics[scale=0.87]{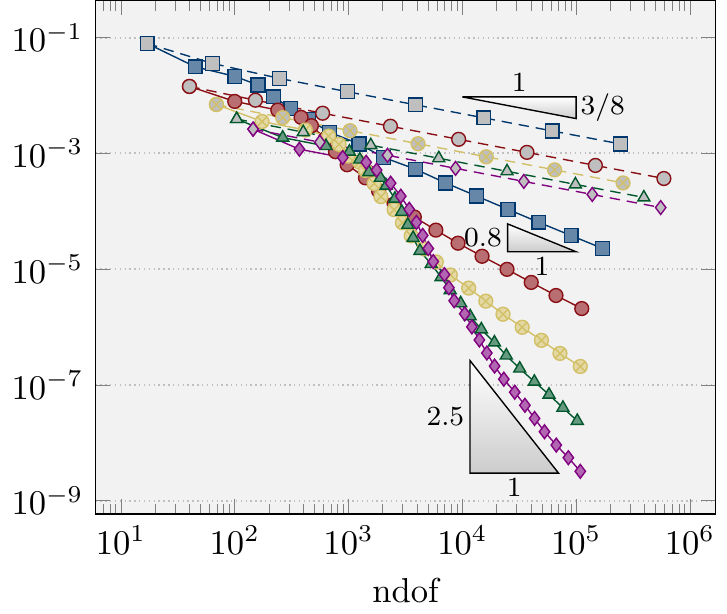}
		\end{minipage}\hfill
		\begin{minipage}[t]{0.49\textwidth}
			\centering
			\includegraphics[scale=0.87]{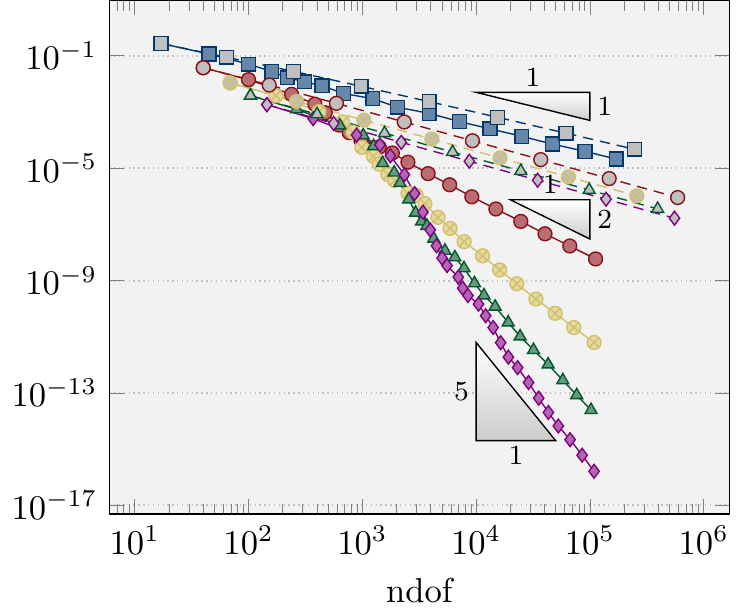}
		\end{minipage}\hfill
		\caption{Convergence history plot of $\|\nabla u - \GrRec_\ell u_\ell\|_{L^4(\Omega)}^2$ (left) and $\|\sigma - \nabla W(\GrRec_\ell u_\ell)\|_{L^{4/3}(\Omega)}^2$ (right) with $k$ from \Cref{fig:legend} on uniform (dashed line) and adaptive (solid line) triangulations for the $p$-Laplace benchmark in \Cref{sec:num-ex:p-Laplace}}
		\label{fig:convergence-displacement-stress-p-Laplace}
	\end{figure}
	The third numerical benchmark from \cite[Section 6]{CKlose2003} for the $p$-Laplace problem in \Cref{ex:pLaplace} considers $p = 4$, the right-hand side
	\begin{align*}
		f(r,\varphi) \coloneqq 343/2048 r^{-11/8} \sin(7\varphi/8),
	\end{align*}
	on the L-shaped domain $\Omega \coloneqq (-1,1)^2\setminus([0,1)\times(-1,0])$ with the initial triangulation $\Tcal_0$ displayed in \Cref{fig:convergence-energy-p-Laplace}.a, the Dirichlet boundary data $u_\mathrm{D}(r,\varphi) \coloneqq r^{7/8}\sin(7\varphi/8)$
	$\Gamma_{\mathrm{D}} \coloneqq ({0} \times [-1,0]) \cup ([0,1] \times {0})$, and the Neumann boundary data
	\begin{align*}
		g(r,\varphi) \coloneqq 343/512 r^{-3/8} (-\sin(\varphi/8), \cos(\varphi/8)) \cdot \nu
	\end{align*}
	in polar coordinates with the outer normal unit vector $\nu$ on $\Gamma_{\mathrm{N}} \coloneqq \partial \Omega\setminus \Gamma_{\mathrm{D}}$.
	The minimal energy $\min E(\Acal) = -1.4423089582447$ is attained at the unique minimizer
	\begin{align*}
		u(r,\varphi) \coloneqq r^{7/8}\sin(7\varphi/8).
	\end{align*}
	Since $u$ is singular at the origin, reduced convergence rates are expected for uniform mesh-refining. \Cref{fig:convergence-energy-p-Laplace}.b displays the suboptimal convergence rates $0.75$ for the energy error $|E(u) - E_\ell(u_\ell)|$ and all polynomial degrees $k = 0, \dots, 4$.
	The adaptive mesh-refining algorithm refines towards the origin as depicted in \Cref{fig:adaptive-triangulation-p-Laplace} and we observed a stronger local refinement for larger polynomial degree $k$.
	Since $W$ satisfies \eqref{ineq:cc-primal}--\eqref{ineq:cc-stress}, the interest is on the displacement error $\|\nabla u - \GrRec_\ell u_\ell\|_{L^4(\Omega)}$ and the stress error $\|\sigma - \nabla W(\GrRec_\ell u_\ell)\|_{L^{4/3}(\Omega)}$. On uniformly refined meshes, $\|\nabla u - \GrRec_\ell u_\ell\|_{L^4(\Omega)}^2$ converges with the suboptimal convergence rate $0.375$ and adaptive computation improves the convergence rate to $0.8$ for $k = 0$ and $2.5$ for $k = 4$ as depicted in \Cref{fig:convergence-displacement-stress-p-Laplace}.a. \Cref{fig:convergence-displacement-stress-p-Laplace}.b displays the convergence rate $1$ for the stress error $\|\sigma - \nabla W(\GrRec_\ell u_\ell)\|_{L^{4/3}(\Omega)}^2$ on uniform triangulations for all $k = 0, \dots, 4$. This is optimal for $k = 0$, but \emph{not} for $k \geq 1$. The adaptive mesh-refining algorithm recovers the optimal convergence rates $k+1$ for $k \geq 1$.
	
	\subsection{The optimal design problem}\label{sec:num-ex:ODP}
	Consider $W$ from \Cref{ex:ODP} for $\mu_1 = 1$, $\mu_2 = 2$, $\xi_1 = \sqrt{2\lambda \mu_1/\mu_2}$, and $\xi_2 = \mu_2\xi_1/\mu_1$ with the fixed parameter $\lambda = 0.0145$ on the L-shaped domain $\Omega \coloneqq (-1,1)^2\setminus([0,1)\times(-1,0])$ from \cite[Figure 1.1]{BartelsC2008}. Let $f \equiv 1$ in $\Omega$ and $u_\mathrm{D} \equiv 0$ on $\Gamma_{\mathrm{D}} = \partial \Omega$ with the reference value $\min E(\Acal) = -0.0745512$.
	\begin{figure}[h!]
		\begin{minipage}[t]{0.49\textwidth}
			\centering
			\includegraphics[scale=0.48]{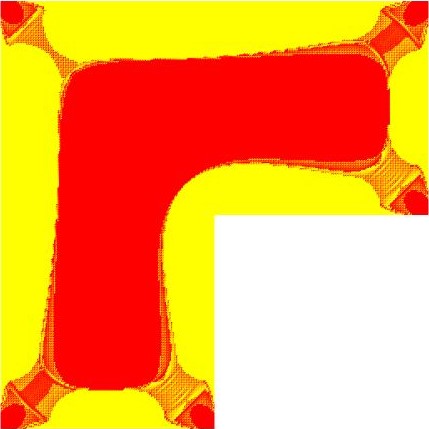}
		\end{minipage}\hfill
		\begin{minipage}[t]{0.49\textwidth}
			\centering
			\includegraphics[scale=0.87]{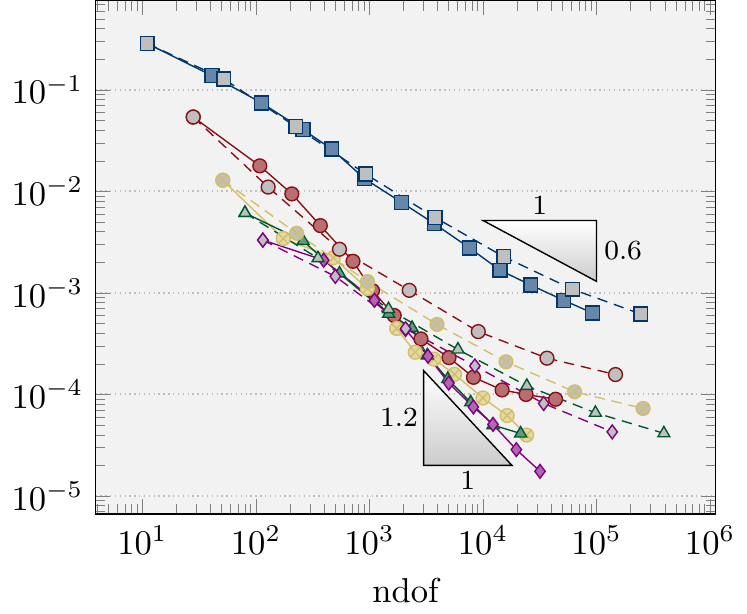}
		\end{minipage}\hfill
		\caption{Material distribution of the L-shaped domain (left) and convergence history plot (right) of $\mathrm{RHS}_\ell$ in \eqref{ineq:RHS} with $k$ from \Cref{fig:legend} on uniform (dashed line) and adaptive (solid line) triangulations for the optimal design problem in \Cref{sec:num-ex:ODP}}
		\label{fig:convergence-ODP}
	\end{figure}
	\begin{figure}[h!]
		\begin{minipage}[t]{0.49\textwidth}
			\centering
			\includegraphics[scale=0.87]{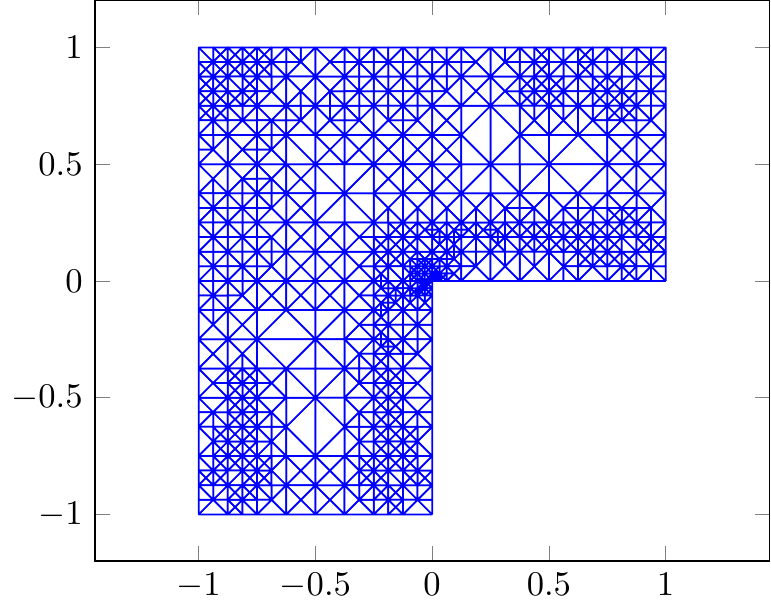}
		\end{minipage}\hfill
		\begin{minipage}[t]{0.49\textwidth}
			\centering
			\includegraphics[scale=0.87]{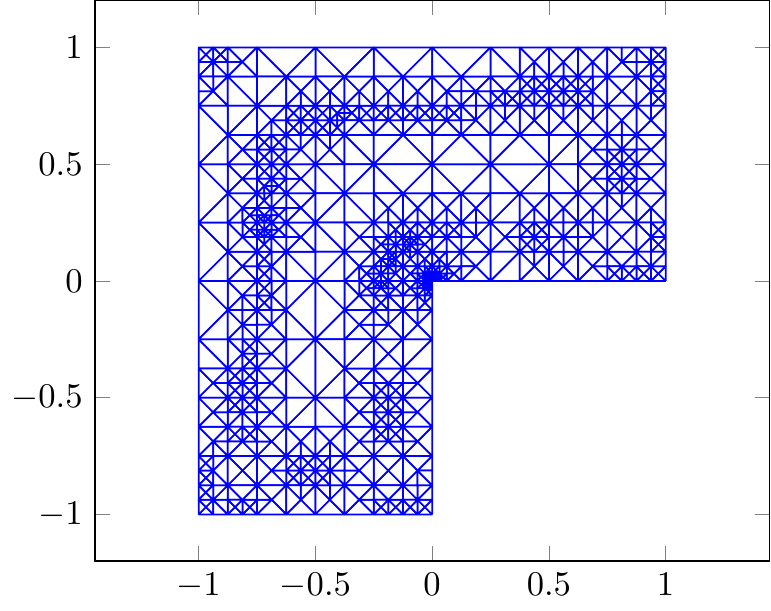}
		\end{minipage}\hfill
		\caption{Adaptive triangulation of the L-shaped domain into 1510 triangles (3721 dofs) for $k = 0$ (left) and 1351 triangles (21450 dofs) for $k = 3$ (right) for the optimal design problem in \Cref{sec:num-ex:ODP}}
		\label{fig:adaptive-triangulations-ODP}
	\end{figure}
	
	The material distribution in \Cref{fig:convergence-ODP}.a consists of two homogenous phases, an interior (red) and a boundary (yellow) layer, and a transition layer, also called microstructure zone with a fine mixture of the two materials \cite{BartelsC2008,CLiu2015,CGuentherRabus2012,CarstensenTran2020}. The approximated volume fractions $\Lambda(|\Pi_{\Tcal_\ell}^0 \GrRec_\ell u_\ell|)$ for a discrete minimizer $u_\ell$ with $\Lambda(\xi) = 0$ if $0 \leq \xi \leq \xi_1$, $\Lambda(\xi) = (\xi - \xi_1)/(\xi_2 - \xi_1)$ if $\xi_1 \leq \xi \leq \xi_2$, and $\Lambda(\xi) = 1$ if $\xi \geq \xi_2$, define the colour map of the fraction plot of \Cref{fig:convergence-ODP}.
	Since $W$ satisfies \eqref{ineq:cc-stress}, \Cref{thm:plain-convergence} implies the convergence of $|E(u) - E_\ell(u_\ell)|$ and $\|\sigma - \nabla W(\GrRec_\ell u_\ell)\|_{L^{2}(\Omega)}$. Since the exact solution is unknown, the numerical experiment computes $\mathrm{RHS}_\ell$ in
	\begin{align}
		\begin{split}
			&\|\sigma - \nabla W(\GrRec_\ell u_\ell)\|_{L^2(\Omega)}^2 + |E(u) - E_\ell(u_\ell)|\\
			&\qquad \lesssim \mathrm{RHS}_\ell \coloneqq E_\ell(u_\ell) - E^*(\sigma_\ell) + \osc(f,\Tcal_\ell) + \|\GrRec_\ell u_\ell - \nabla \Jcal_\ell u_\ell\|_{L^2(\Omega)}^2
			\label{ineq:RHS}
		\end{split}
	\end{align}
	from \cite[Theorem 4.6]{CarstensenTran2020} with the convex conjugate $W^* \in C(\M)$ \cite[Corollary 12.2.2]{Rockafellar1970} and the dual energy
	\begin{align}
		E^*(\sigma_\ell) \coloneqq -\int_\Omega W^*(\sigma_\ell) \d{x}.\label{def:dual-energy}
	\end{align}
	\Cref{fig:convergence-ODP}.b displays the suboptimal convergence rate 0.4 for $\mathrm{RHS}_\ell$ on uniform triangulations. The adaptive algorithm refines towards the reentrant corner and the boundaries of the microstructure zone as displayed in \Cref{fig:adaptive-triangulations-ODP}. This improves the convergence rates up to $1.2$ for $k = 4$. Undisplayed computer experiments show significant improvement for the convergence rates of $\mathrm{RHS}_\ell$ for examples with small microstructure zones in agreement with the related empirical observations in \cite{CGuentherRabus2012}.
	
	\subsection{The relaxed two-well benchmark}\label{sec:num-ex:2well}
	\begin{figure}[ht!]
		\begin{minipage}[t]{0.49\textwidth}
			\centering
			\includegraphics[scale=0.87]{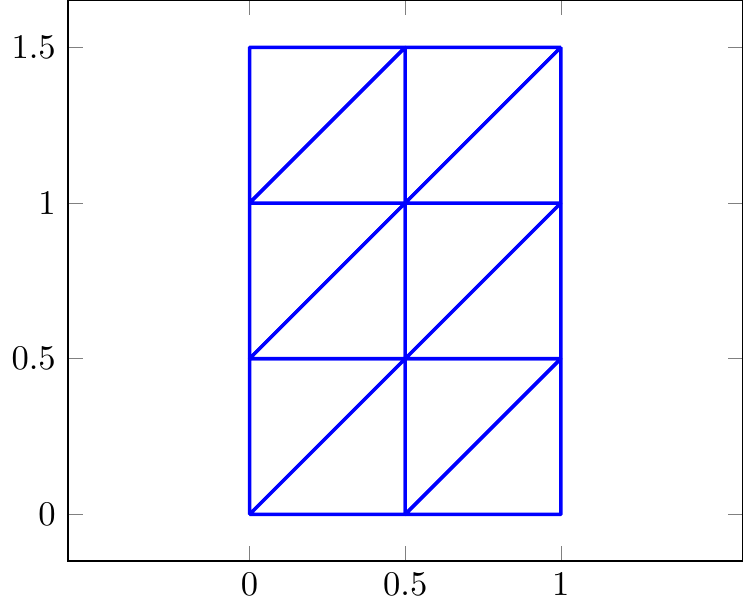}
		\end{minipage}\hfill
		\begin{minipage}[t]{0.49\textwidth}
			\centering
			\includegraphics[scale=0.87]{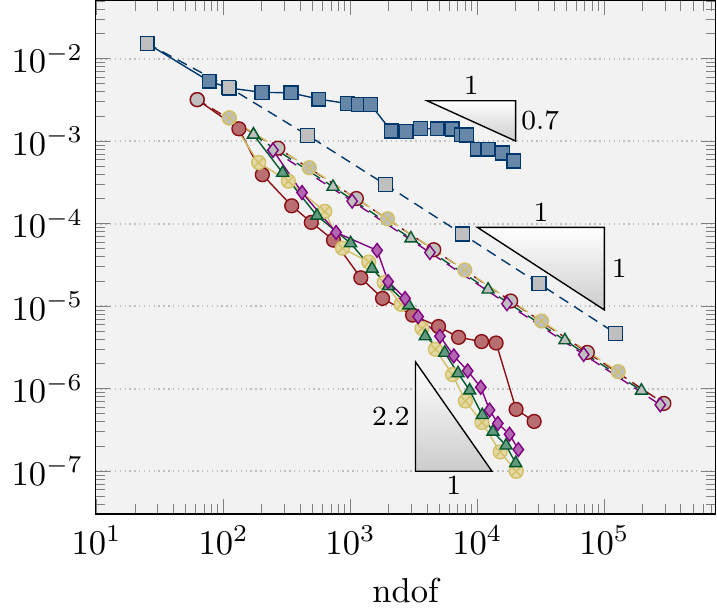}
		\end{minipage}\hfill
		\caption{Initial triangulation (left) of the rectangular domain $\Omega$ and convergence history plot (right) of $|E(u) - E_\ell(u_\ell)|$ with $k$ from \Cref{fig:legend} on uniform (dashed line) and adaptive (solid line) triangulations for the two-well benchmark in \Cref{sec:num-ex:2well}}
		\label{fig:initial-triangulation-convergence-energy-2well}
	\end{figure}
	\noindent Let $\Omega \coloneqq (0,1) \times (0,3/2)$ with pure Dirichlet boundary $\Gamma_{\mathrm{D}} \coloneqq \partial \Omega$.
	The computational benchmark from \cite{CKlose2003} considers the two distinct wells $F_1 = -(3,2)/\sqrt{13} = -F_2$ in the definition of $W$ from \Cref{ex:2well} and introduces an additional quadratic term $\|\zeta - v\|_{L^2(\Omega)}^2$ in the energy
	\begin{align*}
		E(v) \coloneqq \int_\Omega (W(\nabla v) - f v) \d{x} + \|\zeta - v\|_{L^2(\Omega)}^2/2
	\end{align*}
	for all $v \in \Acal \coloneqq u_\mathrm{D} + W^{1,4}_0(\Omega)$ with $f(x,y) \coloneqq - 3\varrho^5/128 - \varrho^3/3$, $\zeta(x,y) \coloneqq \varrho^3/24 + \varrho$,
	\begin{align*}
		u(x,y) \coloneqq u_\mathrm{D}(x,y) \coloneqq \begin{cases}
			f(x,y) &\mbox{if } -1/2 \leq \varrho \leq 0,\\
			\zeta(x,y) &\mbox{if } 0 \leq \varrho \leq 1/2
		\end{cases}
	\end{align*}
	at $(x,y) \in \R^2$ and $\varrho \coloneqq (3(x-1) + 2y)/\sqrt{13}$. Since $E$ is strictly convex in $\Acal$, the minimal energy $\min E(\Acal) = E(u) = 0.1078147674$ is attained at the unique minimizer $u$. The discrete minimizer $u_\ell = (u_{\Tcal_\ell}, u_{\Fcal_\ell})$ of the discrete energy
	\begin{align*}
		E_\ell(v_\ell) \coloneqq \int_\Omega (W(\GrRec_\ell v_\ell) - f v_{\Tcal_\ell}) \d{x} + \|\zeta - v_{\Tcal_\ell}\|_{L^2(\Omega)}^2/2 \text{ among } v_\ell = (v_{\Tcal_\ell}, v_{\Fcal_\ell}) \in \Acal(\Tcal_\ell)
	\end{align*}
	is unique in the volume component $u_{\Tcal_\ell}$ only. The convergence analysis can be extended to the situation at hand with the refinement indicator $\widetilde{\eta}_\ell^{(\varepsilon)}(T) \coloneqq \eta_\ell^{(\varepsilon)}(T) + |T|\|(1 - \Pi_T^k) \zeta\|_{L^2(T)}^2$ and leads to $\lim_{\ell \to \infty} E_\ell(u_\ell) = E(u)$, $\lim_{\ell \to \infty} \nabla W(\GrRec_\ell u_\ell) = \sigma$ (strongly) in $L^{4/3}(\Omega;\R^2)$, and $\lim_{\ell \to \infty} u_{\Tcal_\ell} = u$ (strongly) in $L^4(\Omega)$.
	\begin{figure}[t!]
		\begin{minipage}[t]{0.49\textwidth}
			\centering
			\includegraphics[scale=0.87]{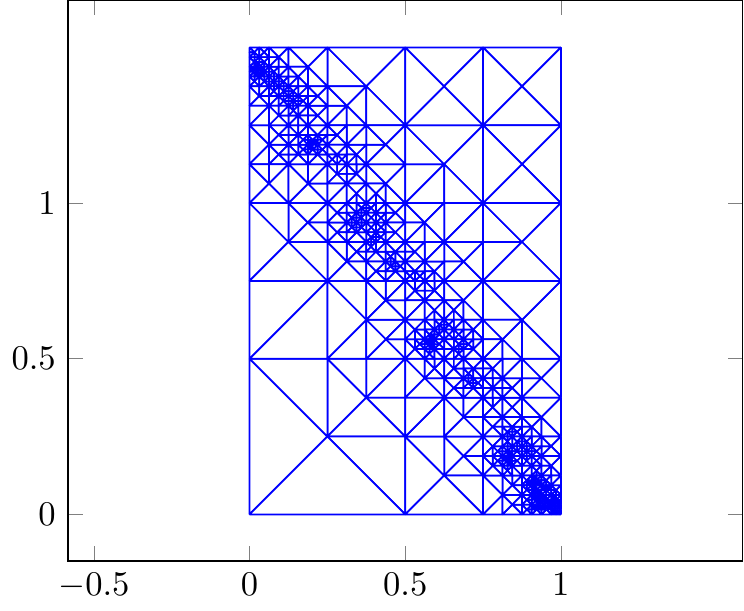}
		\end{minipage}\hfill
		\begin{minipage}[t]{0.49\textwidth}
			\centering
			\includegraphics[scale=0.87]{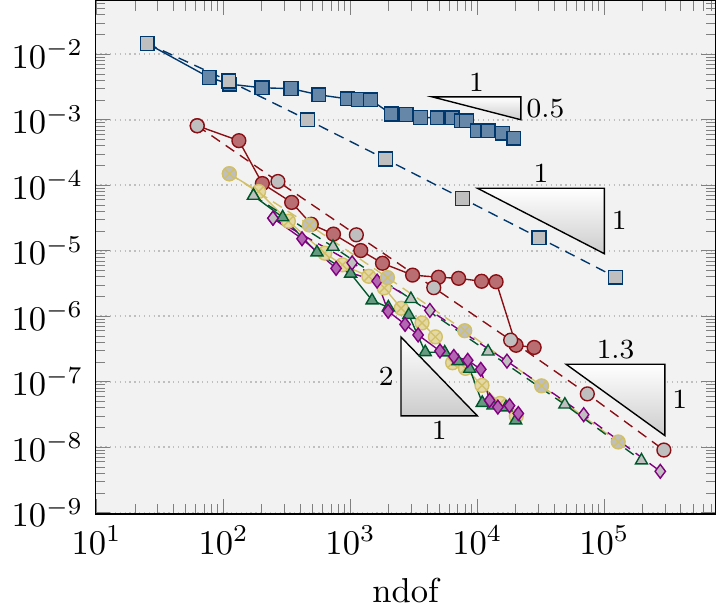}
		\end{minipage}\hfill
		\caption{Adaptive triangulation (left) of the rectangular domain $\Omega$ into 1192 triangles (7104 dofs) for $k = 1$ and convergence history plot (right) of $\|u - u_{\Tcal_\ell}\|_{L^2(\Omega)}^2$ with $k$ from \Cref{fig:legend} on uniform (dashed line) and adaptive (solid line) triangulations for the two-well benchmark in \Cref{sec:num-ex:2well}}
		\label{fig:adaptive-triangulation-convergence-volume-2well}
	\end{figure}
	\begin{figure}[t!]
		\begin{minipage}[t]{0.49\textwidth}
			\centering
			\includegraphics[scale=0.87]{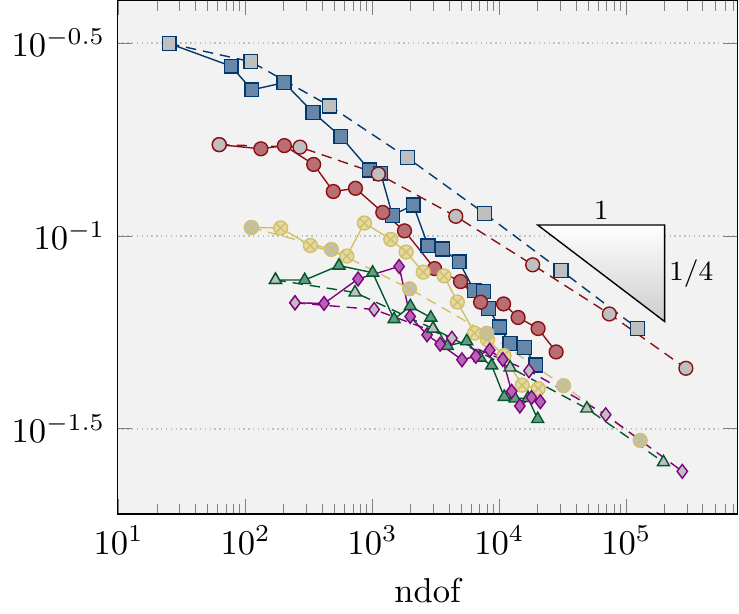}
		\end{minipage}\hfill
		\begin{minipage}[t]{0.49\textwidth}
			\centering
			\includegraphics[scale=0.87]{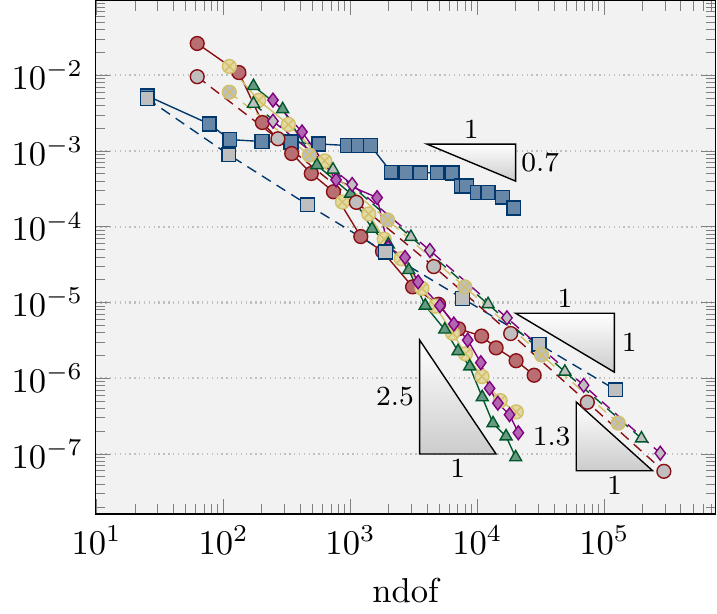}
		\end{minipage}\hfill
		\caption{Convergence history plot of $\|\nabla u - \GrRec_\ell u_\ell\|_{L^4(\Omega)}^2$ (left) and $\|\sigma - \nabla W(\GrRec_\ell u_\ell)\|_{L^{4/3}(\Omega)}^2$ (right) with $k$ from \Cref{fig:legend} on uniform (dashed line) and adaptive (right) triangulations for the two-well benchmark in \Cref{sec:num-ex:2well}}
		\label{fig:convergence-displacement-stress-2well}
	\end{figure}

	The exact solution $u$ is piecewise smooth and the derivative $\nabla u$ jumps across the interface $\Gamma = \mathrm{conv}\{(1,0),(0,3/2)\}$. For an aligned initial triangulation, where $\Gamma$ coincides with the sides of the triangulation, the numerical results from \cite{CarstensenTran2020} display optimal convergence rates $k+1$ for $|E(u) - E_\ell(u_\ell)|$, $\|\sigma - \sigma_\ell\|_{L^{4/3}(\Omega)}^2$ $\|u - u_{\Tcal_\ell}\|_{L^2(\Omega)}^2$, and $\|\nabla u - \GrRec_\ell u_\ell\|_{L^4(\Omega)}^2$ on uniformly refined meshes.
	Since a~priori information on $u$ is not available in general, this numerical benchmark considers the non-aligned initial triangulation $\Tcal_0$ in \Cref{fig:initial-triangulation-convergence-energy-2well}.a, where $\Gamma$ cannot be resolved exactly (even not with adaptively refined triangulations of $\Tcal_0$). In this case, \cite{CJochimsen2003} predicted
	\begin{align*}
		\|(1 - \Pi_{\Tcal_\ell}^0) u\|_{L^4(\Omega)} + \|(1 - \Pi_{\Tcal_\ell}^0) \sigma\|_{L^{4/3}(\Omega)} \lesssim H_\ell,
		\|(1 - \Pi_{\Tcal_\ell}^0) \nabla u\|_{L^4(\Omega)} \lesssim H_\ell^{1/4}
	\end{align*}
	for $H_\ell \coloneqq \|h_\ell\|_{L^\infty(\Omega)}$.
	These expected (optimal) convergence rates on uniform meshes are indeed observed empirically for the lowest-order HHO scheme. \Cref{fig:initial-triangulation-convergence-energy-2well}.b, \Cref{fig:adaptive-triangulation-convergence-volume-2well}.b, and \Cref{fig:convergence-displacement-stress-2well} displays the convergence rate $1$, $1$, $1/4$, and $1$ for $|E(u) - E_\ell(u_\ell)|$, $\|u - u_{\Tcal_\ell}\|_{L^2(\Omega)}^2$, $\|\nabla u - \GrRec_\ell u_\ell\|_{L^4(\Omega)}^2$, and $\|\sigma - \nabla W(\GrRec_\ell u_\ell)\|_{L^{4/3}(\Omega)}^2$, respectively.
	This improves the convergence rate $3/4$ of the stress error from the lowest-order Courant FEM in \cite{CJochimsen2003}.
	The adaptive algorithm generates adaptive meshes with a strong local mesh-refinement near the interface $\Gamma$ and improve the convergence rate of $|E(u) - E_\ell(u_\ell)|$ to $2.2$ in \Cref{fig:initial-triangulation-convergence-energy-2well}.b, of $\|u - u_{\Tcal_\ell}\|_{L^2(\Omega)}^2$ to $2$ in \Cref{fig:adaptive-triangulation-convergence-volume-2well}.b, and of $\|\sigma - \nabla W(\GrRec_\ell u_\ell)\|_{L^{4/3}(\Omega)}^2$ to $2.5$ in \Cref{fig:convergence-displacement-stress-2well}.b for polynomial degrees $k \geq 2$. For $k = 1$, adaptive mesh refinements only leads to marginal improvements. Since optimal convergence rates are obtained for $\|u - u_{\Tcal_\ell}\|_{L^2(\Omega)}$ and $\|\sigma - \nabla W(\GrRec_\ell u_\ell)\|_{L^{4/3}(\Omega)}$ with $k = 0$ on uniform meshes, there is not much gain from adaptive computation.
	
	\subsection{Modified Foss-Hrusa-Mizel benchmark}\label{sec:num-ex:FHM}
	The final example considers a modified Foss-Hrusa-Mizel \cite{FossHrusaMizel2003} benchmark in \cite{OrtnerPraetorius2011}, extended to the domain $\Omega \coloneqq (-1,1) \times (0,1)$ with $\Gamma_1 \coloneqq [-1,0] \times \{0\}$, $\Gamma_2 \coloneqq [0,1] \times \{0\}$, $\Gamma_3 \coloneqq \{x = (x_1,x_2) \in \partial \Omega: x_1 = -1 \text{ or } x_1 = 1 \text{ or } x_2 = 1\}$, and the initial triangulation $\Tcal_0$ of \Cref{fig:initial-triangulation-Lavrentiev-gap-verification}.a. Define the energy density $W(A) \coloneqq (|A|^2 - 2\det A)^4 + |A|^2/2$ for all $A \in \M \coloneqq \R^{2 \times 2}$, the set
	\begin{align*}
		\Acal \coloneqq \{v = (v_1,v_2) \in W^{1,2}(\Omega;\R^2) : v_1 \equiv 0 \text{ on } \Gamma_1, v_2 \equiv 0 \text{ on } \Gamma_2, v = u_\mathrm{D} \text{ on } \Gamma_3\}
	\end{align*}
	of admissible functions in $W^{1,2}(\Omega;\R^2)$ with $u_\mathrm{D} \coloneqq (\cos(\varphi/2), \sin(\varphi/2))$ in polar coordinates, and the vanishing right-hand side $f \equiv 0$. The minimal energy $E(u) = \min E(\Acal) = 0.88137023556$ of
	\begin{align*}
		E(v) \coloneqq \int_\Omega W(\D v) \d{x} \text{ among } v \in \Acal
	\end{align*}
	is attained at $u \coloneqq r^{1/2}(\cos(\varphi/2), \sin(\varphi/2))$ in polar coordinates. The energy density $W \in C^1(\M)$ is convex and satisfies the lower growth $W(A) \geq |A|^2/2$ of order $p = 2$, but \emph{no} upper growth of order $2$.	
	\begin{figure}[h!]
		\begin{minipage}[t]{0.49\textwidth}
			\centering
			\includegraphics[scale=0.87]{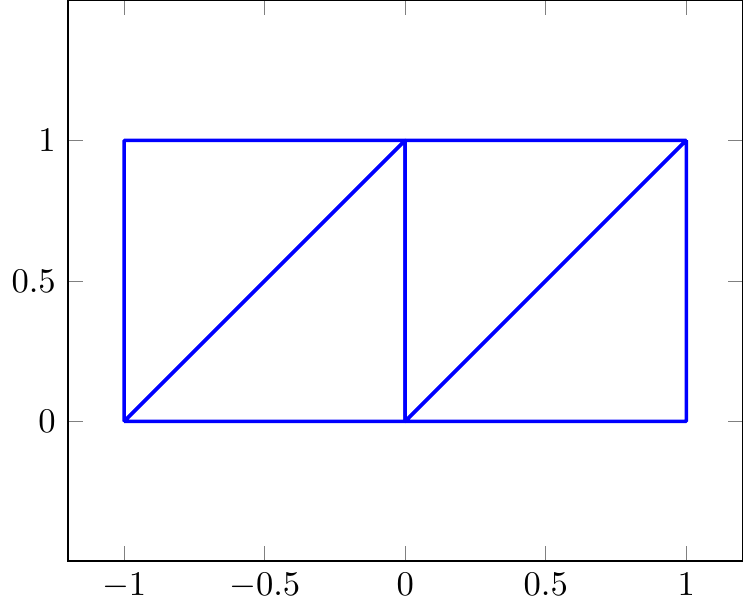}
		\end{minipage}\hfill
		\begin{minipage}[t]{0.49\textwidth}
			\centering
			\includegraphics[scale=0.87]{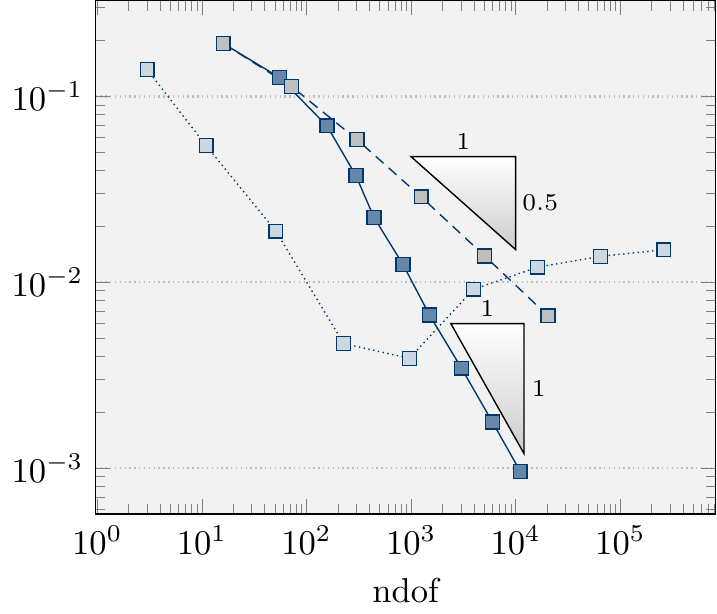}
		\end{minipage}\hfill
		\caption{Initial triangulation $\Tcal_0$ (left) of $\Omega$ and empirical verification of the Lavrentiev gap (right) for the modified Foss-Hrusa-Mizel benchmark in \Cref{sec:num-ex:FHM}: convergence history plot of $|E(u) - E_\ell(u_\ell)|$ for the Courant FEM (dotted line) and the lowest-order HHO method on uniform (dashed line) and adaptive (solid line) triangulations}
		\label{fig:initial-triangulation-Lavrentiev-gap-verification}
	\end{figure}
	\begin{figure}[h!]
		\begin{minipage}[t]{0.49\textwidth}
			\centering
			\includegraphics[scale=0.87]{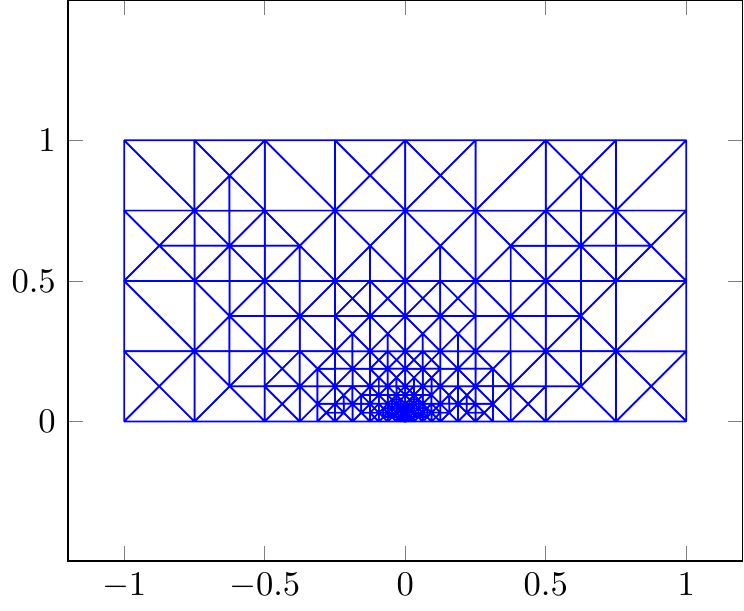}
		\end{minipage}\hfill
		\begin{minipage}[t]{0.49\textwidth}
			\centering
			\includegraphics[scale=0.87]{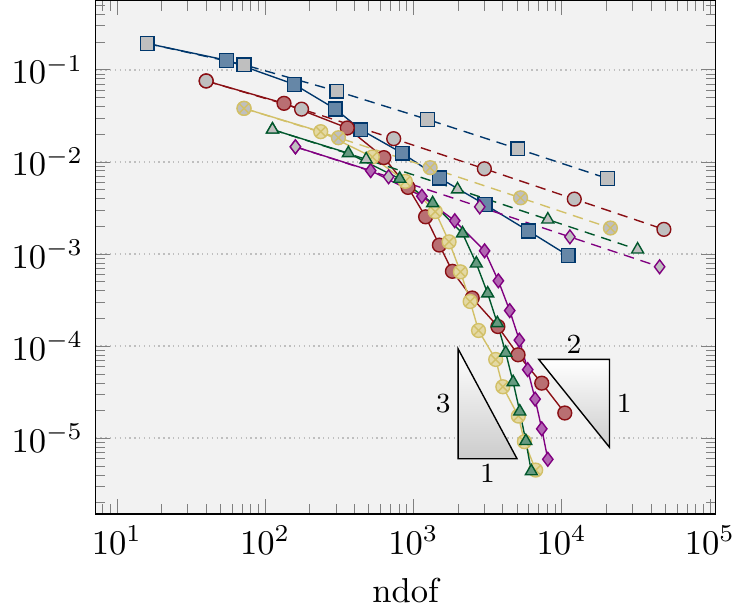}
		\end{minipage}\hfill
		\caption{Adaptive triangulation of $\Omega$ (left) into 614 triangles (7336 dofs) for $k = 1$ and convergence history plot (right) of $|E(u) - E_\ell(u_\ell)|$ with $k$ from \Cref{fig:legend} on uniform (dashed line) and adaptive (solid line) triangulations for the modified Foss-Hrusa-Mizel benchmark in \Cref{sec:num-ex:FHM}}
		\label{fig:convergence-energy-FHM}
	\end{figure}
	The application of the discrete compactness to this model example with free boundary requires the modified refinement indicator
	\begin{align*}
		&\eta_\ell^{(\varepsilon)}(T) \coloneqq |T|^{\varepsilon - 1}\|\Pi_{T}^k((\PotRec_\ell u_\ell)|_T - u_T)\|^2_{L^2(T)} + |T|^{\varepsilon - 1/2}\\
		&\qquad\times\Big(\sum_{F \in \Fcal_\ell(T) \cap \Fcal_\ell(\Gamma_1)} \|(\PotRec_\ell u_\ell)|_F \cdot e_1\|_{L^2(F)}^2 + \sum_{F \in \Fcal_\ell(T) \cap \Fcal_\ell(\Gamma_2)} \|(\PotRec_\ell u_\ell)|_F \cdot e_2\|_{L^2(F)}^2\\
		&\qquad\qquad + \sum_{F \in \Fcal_\ell(T) \cap \Fcal_\ell(\Gamma_3)} \|(\PotRec_\ell u_\ell)|_F - u_\mathrm{D}\|_{L^2(F)}^2 + \sum_{E \in \Fcal_\ell(T) \cap \Fcal_\ell(\Omega)} \|[\PotRec_\ell u_\ell]_F\|_{L^2(F)}^2\\
		&\qquad\qquad + \sum_{F \in \Fcal_\ell(T)} \|\Pi_F((\PotRec_\ell u_\ell)_T - u_F)\|_{L^2(F)}^2\Big)
	\end{align*}
	with the $j$-th canonical unit vector $e_j \in \R^2$.
	Since the presence of the Lavrentiev gap is equivalent to the failure of conforming FEMs \cite[Theorem 2.1]{COrtner2010}, the lowest-order HHO can be utilized to detect the Lavrentiev gap, cf.~\Cref{sec:Lavrentiev-gap}. \Cref{fig:initial-triangulation-Lavrentiev-gap-verification}.b provides empirical evidence that there is a Lavrentiev gap: $|E(u) - E_\ell(u_\ell)|$ converges with the suboptimal convergence rate $0.5$ on uniformly refined meshes, but the Courant FEM seems to approximate a wrong energy. The adaptive mesh-refining algorithm refines towards the origin as depicted in \Cref{fig:convergence-energy-FHM}.a.
	It is outlined in \Cref{sec:Lavrentiev-gap} that a convergence proof of AHHO for minimization problems with the Lavrentiev gap is impossible with the known mathematical methodology for $k \geq 1$.
	It comes as a welcome surprise that optimal convergence rates $k+1$ are obtained for any polynomial degrees $k$ on adaptively refined meshes in \Cref{fig:convergence-energy-FHM}.b.
	
	\subsection{Conclusions}
	The numerical results from \Cref{sec:numerical-examples} confirm the theoretical findings in \Cref{thm:plain-convergence}. In particular, the convergence of the energy $\lim_{\ell \to \infty} \min E_\ell(\Acal(\Tcal_\ell)) = \min E(\Acal)$ is observed in all examples. 
	The introduced adaptive mesh-refining algorithm of \Cref{sec:adaptive_algorithm} provides efficient approximations of singular solutions and even leads to improved empirical convergence rates.
	The choice of the parameter $\varepsilon$ only has marginal influence on the convergence rates and convergence is observed for $\varepsilon = 0$ in undisplayed computer experiments.
	Better convergence rates are obtained for larger polynomial degrees $k$. The computer experiments provide empirical evidence that the HHO method can overcome the Lavrentiev gap for any polynomial degree $k$.
	
	\paragraph{Acknowledgement.} This work has been supported by the Deutsche Forschungsgemeinschaft (DFG) in the Priority Program 1748 Reliable simulation techniques in solid mechanics: Development of non-standard discretization methods, mechanical and mathematical analysis under the project CA 151/22.

	\printbibliography
\end{document}